\documentclass{article}

\usepackage{amsmath,amssymb} 
\usepackage{mathrsfs} 
\usepackage{bm} 

\usepackage{hyperref}
\usepackage{graphicx} 
\usepackage{tikz-cd} 
\usepackage{mathtools}  
\usepackage{braket}

\usetikzlibrary{positioning,intersections,calc,decorations.markings,arrows.meta}
\def\hsep{1cm}
\def\loopdiam{0.3cm}
\def\unode{0.1cm}
\tikzset{inner sep=0.05cm, minimum size=0.5cm}
\tikzset{
block/.style={
  draw, 
  rectangle, 
  minimum height=0.5cm, 
  minimum width=\hsep*3/2, align=center
  }
}
\tikzset{
	->-/.style={decoration={markings,mark=at position 0.5 with
     {\arrow[xshift=2pt]{Latex[length=4pt,#1]}}},postaction={decorate}},
	-<-/.style={decoration={markings,mark=at position 0.5 with
     {\arrow[xshift=-2pt,rotate=180]{Latex[length=4pt,#1]}}},postaction={decorate}}
}

\usepackage{amsthm} 
\theoremstyle{definition} 
\newtheorem{dfn}{Definition}[section]
\newtheorem*{dfn*}{Definition}
\newtheorem{rmk}[dfn]{Remark}
\newtheorem{ex}[dfn]{Example}

\theoremstyle{plain}
\newtheorem{lem}[dfn]{Lemma}
\newtheorem*{lem*}{Lemma}
\newtheorem{prop}[dfn]{Proposition}
\newtheorem*{prop*}{Proposition}
\newtheorem{thm}[dfn]{Theorem}
\newtheorem*{thm*}{Theorem}
\newtheorem{cor}[dfn]{Corollary}
\newtheorem*{cor*}{Corollary}

\newtheorem*{conj*}{Conjecture}

\numberwithin{equation}{section}
\numberwithin{table}{section}
\numberwithin{figure}{section}

\newcommand{\figref}[1]{Figure \ref{#1}}
\newcommand{\tblref}[1]{Table \ref{#1}}

\newenvironment{sd}{\begin{array}{c} \begin{tikzpicture}}{\end{tikzpicture} \end{array}}

\newcommand{\R}{\mathbb{R}}
\newcommand{\C}{\mathbb{C}}
\newcommand{\T}{\mathbb{T}}
\newcommand{\A}{{C(\mathrm{Qut}(\mathcal{G}))}}

\renewcommand{\tilde}{\widetilde}
\newcommand{\ol}[1]{\overline{#1}}
\newcommand{\abs}[1]{{\left\lvert{#1}\right\rvert}}

\mathchardef\hyphen="2D
\newcommand{\quotient}[2]{     
\mathchoice{  \text{\raise1ex\hbox{$#1$}\!\Big/\!\lower1ex\hbox{$#2$}} }
                  {  {#1}\,/\,{#2}  }
                  {  {#1}\,/\,{#2}  }
                  {  {#1}\,/\,{#2}  }
}

\newcommand{\id}{\mathrm{id}}

\newcommand{\Spec}{\mathop{\mathrm{Spec}}\nolimits}

\newcommand{\Tr}{\mathop{\mathrm{Tr}}\nolimits}

\newcommand{\ad}{\mathop{\mathrm{ad}}\nolimits}

\newcommand{\Rep}{\mathop{\mathrm{Rep}}\nolimits}

\newcommand{\Qut}{\mathop{\mathrm{Qut}}\nolimits}

\newcommand{\QSet}{\mathop{\mathrm{QSet}}\nolimits}
\newcommand{\QBij}{\mathop{\mathrm{QBij}}\nolimits}
\newcommand{\QIso}{\mathop{\mathrm{QIso}}\nolimits}
\newcommand{\QAut}{\mathop{\mathrm{QAut}}\nolimits}

\title{Classification of Quantum Graphs on $M_2$ and their Quantum Automorphism Groups}
\author{Junichiro Matsuda\thanks{Department of Mathematics, Graduate School of Science, Kyoto University, 606-8502, Kitashirakawa Oiwakecho, Sakyo-ku, Kyoto, Japan, email: {\tt j.matsuda@math.kyoto-u.ac.jp}. Supported by the establishment of university fellowships towards the creation of science technology innovation, Japan.}}
\date{}

\begin{document}

\maketitle

\begin{abstract}
Motivated by string diagrammatic approach to undirected tracial quantum graphs by Musto, Reutter, Verdon \cite{Musto2018compositional}, in the former part of this paper we diagrammatically formulate directed nontracial quantum graphs by Brannan, Chirvasitu, Eifler, Harris, Paulsen, Su, Wasilewski \cite{Brannan2019bigalois}.
In the latter part, we supply a concrete classification of undirected reflexive quantum graphs on $M_2$ and their quantum automorphism groups in both tracial and nontracial settings. We also obtain quantum isomorphisms between tracial quantum graphs on $M_2$ and certain classical graphs, which reproves the monoidal equivalences between $SO(3)$ and $S_4^+$, and $O(2)$ and $H_2^+$. 

\end{abstract}


\section{Introduction}

The notion of quantum graphs (called noncommutative graphs in \cite{Duan2012zero}) was first introduced by Duan, Severini, Winter \cite{Duan2012zero} in terms of operator systems as the confusability graph of a quantum channel in quantum information theory. 
As an analogue of the fact that simple undirected classical graphs are irreflexive symmetric relations, Weaver \cite{Weaver2021quantum} formulated quantum graphs as reflexive symmetric quantum relations on a von Neumann algebra, which extends \cite{Duan2012zero}, and quantum relations were introduced by Kuperberg, Weaver \cite{Kuperberg2012neumann}.
Following those works, Musto, Reutter, Verdon \cite{Musto2018compositional} formulated finite quantum graphs as adjacency operators on tracial finite quantum sets,
and Brannan et al. \cite{Brannan2019bigalois} generalized them for nontracial settings. 

The key tool of \cite{Musto2018compositional} are string diagrams formulated by Vicary \cite{Vicary2011categorical}, but it should be treated with care if applied to nontracial quantum graphs in \cite{Brannan2019bigalois}.
So in the former part of this paper, we discuss the diagrammatic formulation of nontracial quantum graphs. 

Brannan et al. \cite{Brannan2019bigalois} also introduced the quantum automorphism groups and bigalois extensions of quantum graphs in order to refine the notion of quantum isomorphisms between quantum graphs.
The quantum automorphism group of classical graphs was first introduced by 
Bichon \cite[Definition 3.1]{Bichon2003quantum} in a slightly different way from \cite{Brannan2019bigalois}. The origin of the formulation in \cite{Brannan2019bigalois} is due to Banica \cite[Definition 3.2]{Banica2005quantum}, following the quantum symmetry group of finite spaces introduced by Wang \cite[Definition 2.3]{Wang1998quantum}.

Although some abstract constructions of a quantum graph from others are given categorically by Musto, Reutter, Verdon \cite{Musto2019morita} and algebraically by Brannan, Eifler, Voigt, Weber \cite{Brannan2020quantum}, few nontrivial concrete examples of them were known. 
This motivated the author to compute and classify undirected reflexive quantum graphs and their quantum automorphism groups on the most basic noncommutative algebra $M_2$ as a first step.

This research was concluded during the MSc studies of the author and completed during the first term of his doctoral study. 
After completing the results obtained in this paper, he found that Gromada \cite{Gromada2021some} independently studied partially the same topic.
Gromada classified undirected tracial quantum graphs on $M_2$ \cite[section 3.3]{Gromada2021some} in an insightful way using Lie algebras and the correspondence between the adjacency operators on tracial $M_2$ and projections in $M_2 \otimes M_2^{op}$.

\subsection{Presentation of results}

In section 2, we review basic properties of quantum graphs and generalize the string diagrammatic formulation in \cite{Musto2018compositional} to nontracial cases. An important difference is that the tracial cases allow topological deformation of diagrams while the nontracial cases do not allow deformation through a cusp. 
We compare several properties of directed quantum graphs, in particular, an equivalence between reality and complete positivity of quantum graphs is proved.
As a straightforward generalization of \cite[Proposition 5.19]{Musto2018compositional}, we show that the category of quantum automorphisms of a quantum graph is isomorphic to the finite-dimensional representation category of the quantum automorphism group algebra of the quantum graph.
We also introduce the regularity of quantum graphs, which helps the classification in the following sections.

In section 3, we directly compute the reflexive undirected quantum graphs on $M_2$ and classify them up to quantum and classical isomorphisms.
In the tracial case, they are regular and classified by their degree $d\in\{1,2,3,4\}$. In the nontracial case, they are not always regular but still have a similar form.

In section 4, we identify the quantum automorphism groups of the quantum graphs on $M_2$ classified in section 3. In the tracial case, $SO(3)$ and $O(2)$ appear as quantum automorphism groups. In the nontracial case, the quantum special orthogonal groups $SO_q(3)$ and the unitary torus $\T=U(1)$ appear. 

Observing the spectra, the regular tracial quantum graphs on $M_2$ are isospectral to regular classical graphs on four vertices, which implies the possibility of quantum isomorphisms between them.
Therefore we compute the bigalois extension, the universal coefficient algebra of quantum isomorphisms between quantum graphs introduced by \cite[Definition 4.1]{Brannan2019bigalois}, to find that they are indeed quantum isomorphic. 
Since a quantum isomorphism of quantum graphs induces a monoidal equivalence of their quantum automorphism groups by \cite[Theorem 4.7]{Brannan2019bigalois},
it follows that $SO(3)$ and $S_4^+$, $O(2)$ and $H_2^+$ are monoidally equivalent respectively. 
Although this is already known in quantum group theory \cite{Banica1999symmetries,Banica2007hyperoctahedral}, it exhibits a new approach to monoidal equivalence using quantum graph theory.

Gromada \cite[Proposition 8.1]{Gromada2021some} also obtains the same quantum isomorphisms and monoidal equivalence differently using quantum Cayley graphs as a twist of classical Cayley graphs.

\section*{Acknowledgement}
The author would like to show his greatest appreciation to professor Beno\^{\i}t Collins who is his supervisor at Kyoto University.
Without his guidance and persistent help, this work would not have been possible. The author would like to offer his special thanks to Professor Michael Brannan at University of Waterloo for his insightful advice on quantum graph theory. The author would like to express his gratitude to professor Matthew Kennedy at University of Waterloo for leading him to the quantum graph world.
In addition the author thanks his colleague Akihiro Miyagawa for indicating spelling errors and holding mathematical discussions with him.

\section{Basics in quantum graphs}

Throughout this paper, we denote by $(\cdot)^*$ the involution in a $*$-algebra, and by $(\cdot)^\dagger$ the adjoint of an operator between Hilbert spaces.

\subsection{Quantum sets}

Let $B$ be a finite dimensional unital $C^*$-algebra.
$B$ is equipped with the bilinear multiplication map 
$B\times B \ni (a,b) \mapsto ab \in B$, 
which induces a linear multiplication operator 
$m:B \otimes B \ni a \otimes b \mapsto ab \in B$ 
by the universality of tensor product.
We identify $x \in B$ with a linear map $\C \ni 1 \mapsto x \in B$,
in particular $1$ denotes the multiplicative unit in $B$ and the unital $*$-homomorphism $\C \hookrightarrow B$.

Let $\psi$ be a faithful state on $B$,
then $\langle x,y \rangle_\psi = \braket{y|x}_\psi = \psi(y^* x)$
denotes the inner product on $L^2(B,\psi) = B$. 
The subscript $\psi$ of the inner product is often abbreviated if there is no concern of confusion. 
Via the Hilbert adjoint with respect to $\braket{\cdot|\cdot}_\psi$, $x\in B$ induces 
$x^\dagger=\bra{x}=\psi(x^* \cdot) : B\to\C$, 
and the algebra $(B,m,1)$, a vector space $B$ equipped with the multiplication $m$ and the unit $1$, satisfying
\[
	\begin{array}{c}
	\textrm{associativity}
	\\
	(x \ y) \ z = x \ (y \ z) \mathrlap{\ \forall x,y,z \in B}
	\\
	m(m \otimes \id)=m(\id \otimes m)
	\end{array}
\qquad \quad
	\begin{array}{c}
	\textrm{existence of a unit} 
	\\
	1 \ x=x=x \ 1 \mathrlap{\ \forall x \in B}
	\\
	m(1 \otimes \id)=\id=m(\id \otimes 1)
	\end{array}
\]
 induces a coalgebra $(B,m^\dagger,\psi)$, a vector space $B$ equipped with the comultiplication $m^\dagger$ and the counit $\psi=1^\dagger$, satisfying
\[
	\begin{array}{c}
	\textrm{coassociativity}
	\\
	(m^\dagger \otimes \id)m^\dagger=(\id \otimes m^\dagger)m^\dagger
	\end{array}
	\quad 
	\begin{array}{c}
	\textrm{existence of a counit} 
	\\
	(\psi \otimes \id)m^\dagger=\id=(\id \otimes \psi)m^\dagger
	\end{array}.
\]

Following Vicary \cite{Vicary2011categorical}, we adopt the string diagram notation of operators, which enables our visual understanding and topological calculation. 
For operators $f:H_0 \to H_1$ and $g:H_1 \to H_2$ between Hilbert spaces,
we associate Hilbert spaces with strings, operators with nodes, and read diagrams from bottom to top:
\[
f=\begin{sd}
	\node[circle,draw] (f) {$f$};
	\draw (f)--++(0,-1/2) node[below](H0){$H_0$};
	\draw (f)--++(0,1/2) node[above](H1){$H_1$};
\end{sd},\qquad
g=\begin{sd}
	\node[circle,draw] (g) {$g$};
	\draw (g)--++(0,-1/2) node[below](H1){$H_1$};
	\draw (g)--++(0,1/2) node[above](H2){$H_2$};
\end{sd}.
\]
The composition $gf=g \circ f: H_0 \to H_2$ and 
the tensor product $f \otimes g: H_0 \otimes H_1 \to H_1 \otimes H_2$ 
are denoted by the vertical and horizontal composition of the diagrams respectively,
and the Hilbert adjoint $f^\dagger: H_1 \to H_0$ by the vertical mirroring of the diagram:
\[
g \circ f=\begin{sd}
	\node[circle,draw] (f) {$f$}; 
	\node[circle,draw,above=1/2 of f.center] (g) {$g$};
	\draw (f)--++(0,-1/2) node[below](H0){$H_0$};
	\draw (f)--(g);
	\draw (g)--++(0,1/2) node[above](H1){$H_1$};
\end{sd},
\qquad
f \otimes g=\begin{sd}
	\node[circle,draw] (f) {$f$};
	\draw (f)--++(0,-1/2) node[below](H0){$H_0$};
	\draw (f)--++(0,1/2) node[above](H1){$H_1$};
	\node[circle,draw,right=1/2 of f] (g) {$g$};
	\draw (g)--++(0,-1/2) node[below](H1b){$H_1$};
	\draw (g)--++(0,1/2) node[above](H2){$H_2$};
\end{sd},
\qquad
f^\dagger=\begin{sd}
	\node[circle,draw] (f`) {$f^\dagger$};
	\draw (f`)--++(0,1/2) node[above](H0){$H_0$};
	\draw (f`)--++(0,-1/2) node[below](H1){$H_1$};
\end{sd}.
\]
When a Hilbert space $H$ and its dual $H^*$ or a $C^*$-algebra $B$ appear in a string diagram, we draw $H$ as an oriented string from bottom to top,
$H^*$ as an oriented string from top to bottom, and
$B$ as an unoriented string:
\[
\id_H=
\begin{sd} \node(H){$H$}; \draw[->-](H)--++(0,3/2); \end{sd},
\quad
\id_{H^*}=
\begin{sd} \node(H*){$H^*$}; \draw[-<-] (H*)--(0,3/2); \end{sd},
\quad
\id_B=
\begin{sd} \node(B){$B$}; \draw (B)--(0,3/2); \end{sd}
\] 
We denote the coupling operators of $H$ and $H^*$ and their adjoints by
\begin{align}
\begin{sd}
	\node (H1) at (-0.5,0) {$H$};
	\node (H2) at (0.5,0) {$H^*$};
	\draw[->-] (H1.north)  arc(180:0:0.5) coordinate[midway] (m1)  (H2) ;
	\node[above=0.5 of m1]{$\C$};
	\node (s) at (1.2,0) {$v \otimes f$};
	\node (t) at (1.2,1.5) {$f(v)$};
	\draw[|->] (s)--(t);
\end{sd},
&
\begin{sd}
	\node (H1) at (-0.5,0) {$H^*$};
	\node (H2) at (0.5,0) {$H$};
	\draw[-<-] (H1.north)  arc(180:0:0.5) coordinate[midway] (m1)  (H2) ;
	\node[above=0.5 of m1]{$\C$};
	\node (s) at (1.2,0) {$f \otimes v$};
	\node (t) at (1.2,1.5) {$f(v)$};
	\draw[|->] (s)--(t);
\end{sd},			
&
\begin{sd}
	\node (H1) at (-0.5,0) {$H$};
	\node (H2) at (0.5,0) {$H^*$};
	\draw[-<-] (H1.south)  arc(-180:0:0.5) coordinate[midway] (m1)  (H2) ;
	\node[below=0.5 of m1]{$\C$};
	\node (s) at (1.6,0) {$\sum v_i \otimes v_i^\dagger$};
	\node (t) at (1.6,-1.5) {$1$};
	\draw[<-|] (s)--(t);
\end{sd},
&
\begin{sd}
	\node (H1) at (-0.5,0) {$H^*$};
	\node (H2) at (0.5,0) {$H$};
	\draw[->-] (H1.south)  arc(-180:0:0.5) coordinate[midway] (m1)  (H2) ;
	\node[below=0.5 of m1]{$\C$};
	\node (s) at (1.6,0) {$\sum v_i^\dagger \otimes v_i$};
	\node (t) at (1.6,-1.5) {$1$};
	\draw[<-|] (s)--(t);
\end{sd}						\label{B(H)-unit}
\end{align}
where $\{v_i\}_i$ is an orthonormal basis (ONB) for $H$, 
and $v^\dagger=\bra{v}=\braket{v|\cdot} \in H^*$ for $v=\ket{v} \in H$.
Note that we can naturally identify $H \otimes H^*$ with $B(H)$ by 
$H \otimes H^* \ni \ket{v} \otimes \bra{w} \mapsto \ket{v}\bra{w} \in B(H)$.
Then (\ref{B(H)-unit}) is identified with the unit map $\C \ni 1 \to \id_H \in B(H)$
and the canonical trace $\Tr: B(H) \to \C$.

The operators (\ref{B(H)-unit}) satisfies the following equalities, the so-called snake equation in \cite[section 2.2, (5)]{Musto2018compositional}:
\begin{align}
\begin{sd}
	\draw[-<-] (0,0) arc(0:90:\hsep/3) coordinate(m)
	 arc(90:180:\hsep/3)--++(0,-\hsep/2) ;
	\draw[->-] (0,0) arc(-180:-90:\hsep/3) coordinate(m*) 
	 arc(-90:0:\hsep/3)--++(0,\hsep/2) ;
\end{sd}
=
\begin{sd}
\draw[->-] (0,0) --(0,1);
\end{sd}
=
\begin{sd}
	\draw[->-] (0,0) arc(0:-90:\hsep/3) coordinate(m)
	 arc(-90:-180:\hsep/3)--++(0,\hsep/2) ;
	\draw[-<-] (0,0) arc(180:90:\hsep/3) coordinate(m*) 
	 arc(90:0:\hsep/3)--++(0,-\hsep/2) ;
\end{sd}, 
\quad
\begin{sd}
	\draw[->-] (0,0) arc(0:90:\hsep/3) coordinate(m)
	 arc(90:180:\hsep/3)--++(0,-\hsep/2) ;
	\draw[-<-] (0,0) arc(-180:-90:\hsep/3) coordinate(m*) 
	 arc(-90:0:\hsep/3)--++(0,\hsep/2) ;
\end{sd}
=
\begin{sd}
\draw[-<-] (0,0) --(0,1);
\end{sd}
=
\begin{sd}
	\draw[-<-] (0,0) arc(0:-90:\hsep/3) coordinate(m)
	 arc(-90:-180:\hsep/3)--++(0,\hsep/2) ;
	\draw[->-] (0,0) arc(180:90:\hsep/3) coordinate(m*) 
	 arc(90:0:\hsep/3)--++(0,-\hsep/2) ;
\end{sd}. \label{Hsnake}
\end{align}

The canonical operators associated with $(B,\psi)$ are denoted by
\[
1=\begin{sd}
	\node (C1) {$\C$};
	\draw  (0,1/2) arc(90:450:0.1) --(0,1) node(B1)[above]{$B$};
\end{sd}, \quad
m=\begin{sd}  
	\node (B1) at (-0.5,0) {$B$};
	\node (B2) at (0.5,0) {$B$};
	\draw (B1.north) arc(180:0:0.5) coordinate[midway] (m1) -- (B2) ;
	\draw (m1) --++(0,0.5) node(B3) [above]{$B$};
\end{sd}, \quad
\psi=1^\dagger=\begin{sd}
	\node (*1) {$\C$};
	\draw  (0,-1/2) arc(-90:270:0.1) --(0,-1) node(B1)[below]{$B$};
\end{sd}, \quad
m^\dagger=\begin{sd}  
	\node (B1) at (-0.5,0) {$B$};
	\node (B2) at (0.5,0) {$B$};
	\draw (B1.south) arc(-180:0:0.5) coordinate[midway] (m1) -- (B2) ;
	\draw (m1) --++(0,-0.5) node(B3) [below]{$B$};
\end{sd}.
\]
For simplicity we denote $\psi m$ and $m^\dagger 1$ without the vertical segment and node as follows:
\[
\psi m 
=
\begin{sd}
	\node (B1) at (-0.5,0) {$B$};
	\node (B2) at (0.5,0) {$B$};
	\draw (B1.north) arc(180:0:0.5) coordinate[midway] (m1) ;
	\draw (m1)--++(0,0.1) arc(-90:270:0.1) ;
	\node[above=0.5 of m1]{$\C$};
\end{sd}
=
\begin{sd}
	\node (B1) at (-0.5,0) {$B$};
	\node (B2) at (0.5,0) {$B$};
	\draw (B1.north) arc(180:0:0.5) coordinate[midway] (m1);
	\node[above=0.5 of m1]{$\C$};
\end{sd}, 
\quad
m^\dagger 1
=
\begin{sd}  
	\node (B1) at (-0.5,0) {$B$};
	\node (B2) at (0.5,0) {$B$};
	\draw (B1.south) arc(-180:0:0.5) coordinate[midway] (m1) ;
	\draw (m1)--++(0,-0.1) arc(90:450:0.1) ;
	\node[below=0.5 of m1]{$\C$};
\end{sd}
=
\begin{sd}  
	\node (B1) at (-0.5,0) {$B$};
	\node (B2) at (0.5,0) {$B$};
	\draw (B1.south)  arc(-180:0:0.5) coordinate[midway] (m1) ;
	\node[below=0.5 of m1]{$\C$};
\end{sd}.
\]
The linear extension of a flip map $x\otimes y \mapsto y \otimes x$ is denoted by a crossing of the strings
$\begin{sd}
	\draw (0,0)to[out=90,in=-90](1/2,1/2)
	 (0,1/2) to[out=-90,in=90](1/2,0) ;
\end{sd}$.

The algebra and coalgebra structure of $(B,m,1,m^\dagger,\psi)$ is depicted as follows:
\begin{align*}
	\begin{array}{c}
	\textrm{associative}
	\\
	\begin{sd}
		\draw (0,0) arc(0:180:\hsep/2) coordinate[pos=1/2](m1) ;
		\draw (m1) arc(180:90:\hsep/2)coordinate(m2) 
		to[out=0,in=90](\hsep,0);
		\draw (m2)--++(0,1/4);
	\end{sd}
	=
	\begin{sd}
		\draw (0,0) arc(180:0:\hsep/2) coordinate[pos=1/2](m1) ;
		\draw (m1) arc(0:90:\hsep/2)coordinate(m2) 
		to[out=180,in=90](-\hsep,0);
		\draw (m2)--++(0,1/4);
	\end{sd}
	\\
	m(m \otimes \id)=m(\id \otimes m)
	\end{array}
&\quad &
	\begin{array}{c}
	\textrm{unital} 
	\\
	\begin{sd}
		\draw (0,0) arc(0:90:\hsep/2) coordinate(m1) arc(90:135:\hsep/2) 
		coordinate (u) ;
		\draw[fill=white] (u) circle (\unode);
		\draw (m1)--++(0,1/2);
	\end{sd}
	=
	\begin{sd} 
	\draw (0,0)--++(0,1); 
	\end{sd}
	=
	\begin{sd}
		\draw (0,0) arc(180:90:\hsep/2) coordinate(m1) arc(90:45:\hsep/2) 
		coordinate (u) ;
		\draw[fill=white] (u) circle (\unode);
		\draw (m1)--++(0,1/2);
	\end{sd}
	\\
	m(1 \otimes \id)=\id=m(\id \otimes 1)
	\end{array}
\end{align*}
\begin{align*}
	\begin{array}{c}
	\textrm{coassociative}
	\\
	\begin{sd}
		\draw (0,0) arc(0:-180:\hsep/2) coordinate[pos=1/2](m1) ;
		\draw (m1) arc(-180:-90:\hsep/2)coordinate(m2) 
		to[out=0,in=-90](\hsep,0);
		\draw (m2)--++(0,-1/4);
	\end{sd}
	=
	\begin{sd}
		\draw (0,0) arc(-180:0:\hsep/2) coordinate[pos=1/2](m1) ;
		\draw (m1) arc(0:-90:\hsep/2)coordinate(m2) 
		to[out=-180,in=-90](-\hsep,0);
		\draw (m2)--++(0,-1/4);
	\end{sd}
	\\
	(m^\dagger \otimes \id)m^\dagger=(\id \otimes m^\dagger)m^\dagger
	\end{array}
	&\quad &
	\begin{array}{c}
	\textrm{counital} 
\\
	\begin{sd}
		\draw (0,0) arc(0:-90:\hsep/2) coordinate(m1) 
		arc(-90:-135:\hsep/2) coordinate (u) ;
		\draw[fill=white] (u) circle (\unode);
		\draw (m1)--++(0,-1/2);
	\end{sd}
	=
	\begin{sd} 
	\draw (0,0)--++(0,-1); 
	\end{sd}
	=
	\begin{sd}
		\draw (0,0) arc(-180:-90:\hsep/2) coordinate(m1) 
		arc(-90:-45:\hsep/2) coordinate (u) ;
		\draw[fill=white] (u) circle (\unode);
		\draw (m1)--++(0,-1/2);
	\end{sd}
	\\
	(\psi \otimes \id)m^\dagger=\id=(\id \otimes \psi)m^\dagger
	\end{array}
\end{align*}
The quintuple $(B,m,1,m^\dagger,\psi)$ forms a Frobenius algebra: 

\begin{dfn}[{cf. Vicary \cite[Definition 3.2]{Vicary2011categorical}}]
An algebra with coalgebra structure is called a Frobenius algebra if the multiplication and comultiplication satisfy the \emph{Frobenius equation}:
\begin{align}\label{Frobeniuseq}
\begin{array}{c}
\begin{sd}
	\draw (0,0) arc(0:90:\hsep/3) coordinate(m)
	 arc(90:180:\hsep/3)--++(0,-\hsep*2/3) ;
	\draw (0,0) arc(-180:-90:\hsep/3) coordinate(m*) 
	 arc(-90:0:\hsep/3)--++(0,\hsep*2/3) ;
	\draw (m)--++(0,\hsep/3) ;
	\draw (m*)--++(0,-\hsep/3) ;
\end{sd}
=
\begin{sd}
	\draw (0,0) arc(0:180:\hsep/2) coordinate[pos=1/2](m) ;
	\draw (m)--++(0,\hsep/2)coordinate (m*);
	\draw (m*) arc(-90:0:\hsep/2) 
	(m*) arc(-90:-180:\hsep/2);
\end{sd}
=
\begin{sd}
	\draw (0,0) arc(0:-90:\hsep/3) coordinate(m)
	 arc(-90:-180:\hsep/3)--++(0,\hsep*2/3) ;
	\draw (0,0) arc(180:90:\hsep/3) coordinate(m*) 
	 arc(90:0:\hsep/3)--++(0,-\hsep*2/3) ;
	\draw (m)--++(0,-\hsep/3) ;
	\draw (m*)--++(0,\hsep/3) ;
\end{sd}.
\\
(m \otimes \id)(\id \otimes m^\dagger)
=m^\dagger m
=(\id \otimes m) (m^\dagger \otimes \id) \
\end{array}
\end{align}

\end{dfn}

By composing the unit and the counit, we also have the following snake equation:
\begin{align}
\begin{array}{c}
\begin{sd}
	\draw (0,0) arc(0:90:\hsep/3) coordinate(m)
	 arc(90:180:\hsep/3)--++(0,-\hsep*2/3) ;
	\draw (0,0) arc(-180:-90:\hsep/3) coordinate(m*) 
	 arc(-90:0:\hsep/3)--++(0,\hsep*2/3) ;
\end{sd}
=
\begin{sd}
\draw (0,0) --(0,1.5);
\end{sd}
=
\begin{sd}
	\draw (0,0) arc(0:-90:\hsep/3) coordinate(m)
	 arc(-90:-180:\hsep/3)--++(0,\hsep*2/3) ;
	\draw (0,0) arc(180:90:\hsep/3) coordinate(m*) 
	 arc(90:0:\hsep/3)--++(0,-\hsep*2/3) ;
\end{sd}.
\\
(\psi m \otimes \id)(\id \otimes m^\dagger 1)
=\id 
=(\id \otimes \psi m) (m^\dagger 1 \otimes \id) \
\end{array}		\label{snakeeq}
\end{align}
Note that we may compute string diagrams by topological deformation via Frobenius equality, snake equality, associativity, and coassociativity.

\begin{dfn}[{Banica \cite[section 1]{Banica2002quantum}, Musto, Reutter, Verdon \cite[Terminology 3.1]{Musto2018compositional},
Brannan, et al. \cite[Definition 3.1]{Brannan2019bigalois}}]
Let $\psi$ be a faithful state on a finite dimensional $C^*$-algebra $B$ as above, and $\delta>0$.
The state $\psi$ is called a \emph{$\delta$-form} on $B$ if the following equality 
(so-called \emph{special} in Vicary \cite{Vicary2011categorical}) is satisfied:
\begin{align}
\begin{sd}
	\coordinate (m);
	\draw (m) arc(90:450:\hsep/2) coordinate[pos=1/2] (m*);
	\draw (m)--++(0,1/2) (m*)--++(0,-1/2);
\end{sd}
=\delta^2 
\begin{sd} 
\draw (0,0)--++(0,2); 
\end{sd}
\textrm{ , i.e.,  }
mm^\dagger=\delta^2 \id_B.			\label{special}
\end{align}
And then we call $(B,\psi)$ a quantum set.

A quantum set $(B,\psi)$ is said to be \emph{commutative} or \emph{symmetric} (tracial) if $B$ is commutative or $\psi$ is tracial respectively, which are formulated in diagrams as below.
\begin{align*}
	\begin{array}{cc}
	\textrm{commutative} & \textrm{symmetric (tracial)}
	\\
	\begin{sd}
		\coordinate (B1) ;
		\coordinate[right=\hsep of B1] (B2);
		\coordinate[above =1 of $(B1)!0.5!(B2)$] (m1);
		\draw (B1) to[out=90,in=0] (m1) to[out=-180,in=90] (B2) ;
		\draw (m1) --++(0,0.5) ;
	\end{sd}
	=
	\begin{sd}
		\draw (0,0) arc(180:0:\hsep/2) coordinate[midway] (m1) ;
		\draw (m1) --++(0,0.5) ;
	\end{sd}
	&
	\begin{sd}
		\coordinate (B1) ;
		\coordinate[right=\hsep of B1] (B2);
		\coordinate[above =1 of $(B1)!0.5!(B2)$] (m1);
		\draw (B1) to[out=90,in=0] (m1) to[out=-180,in=90] (B2) ;
	\end{sd}
	=
	\begin{sd}
		\draw (0,0) arc(180:0:\hsep/2) coordinate[midway] (m1) ;
	\end{sd}
	\\
	yx=xy
	&
	\psi(yx)=\psi(xy) \mathrlap{\quad \forall x,y \in B}
	\end{array}
\end{align*}
We often use $\tau$ instead of $\psi$ in the tracial case.
\end{dfn}

\begin{rmk}
The notion of $\delta$-form was introduced by Banica \cite{Banica2002quantum}, and Musto et al. \cite{Musto2018compositional} defined quantum sets in the case where $\psi$ is a trace. Finally, Brannan, et al. \cite{Brannan2019bigalois} defined quantum sets as above.
The definition in \cite{Musto2018compositional} is $mm^\dagger=\id_B$, 
which does not have $\delta^2$. This is because the counit is normalized as $\psi(1)=\delta^2=\abs{B}$ in \cite{Musto2018compositional}, whence $m^\dagger$ in \cite{Musto2018compositional} is our $m^\dagger/\delta^2$. Thus these formulations are equivalent.
\end{rmk}

\begin{lem}[{}]\label{lem:classicalset}
A finite set with the uniform probability measure corresponds to a commutative quantum set via Gelfand duality. In particular $\tau=\Tr/n$ is a $\delta=\sqrt{n}$-form on $\C^n$.
\end{lem}

\begin{proof}
Let $X=\{1,...,n\}$ be an $n$-element set with the uniform probability measure $\mu$. 
The pair $(X,\mu)$ corresponds to the commutative $C^*$-algebra $(C(X), \int \cdot d\mu)$ of (continuous) functions on $X$ with a tracial state $\int \cdot d\mu$ via Gelfand duality.
Moreover $(C(X), \int \cdot d\mu)$ is isomorphic to the $n \times n$ diagonal matrix algebra $(\C^n, \tau=\Tr/n)$ with normalized trace 
via $C(X)\ni \delta_i \mapsto e_i \in \C^n$ where $\delta_i$ is the indicator function of $\{i\} \subseteq X$ and $e_i$ is the matrix unit of $(i,i)$ entry.
Since the multiplication $m$ is given by $e_i \otimes e_j \mapsto \delta_{ij} e_i$,
the comultiplication $m^\dagger$ is given by $e_i \mapsto n e_i \otimes e_i$ because
\begin{align*}
\braket{e_j \otimes e_k | m^\dagger e_i}_{\tau \otimes \tau}
&=\braket{m(e_j \otimes e_k) | e_i}_{\tau}
= \tau((e_j e_k )^* e_i)= \frac{1}{n} \delta_{ji} \delta_{ki} 
\\
&= n \braket{e_j|e_i}_\tau \braket{e_k | e_i}_\tau
= \braket{e_j \otimes e_k | n e_i \otimes e_i}_{\tau \otimes \tau}.
\end{align*}
Thus $mm^\dagger e_i = m(n e_i \otimes e_i)=n e_i$, i.e., 
$mm^\dagger = n \ \id_{\C^n}$. 
Therefore $\tau=\Tr/n$ is a $\delta=\sqrt{n}$-form on $\C^n$.
\end{proof}

Although a general quantum set $(B,\psi)$ is not symmetric, 
it satisfies the following equality, 
so-called \emph{balanced symmetric} in Vicary \cite[Definition 3.10]{Vicary2011categorical}:
\begin{align}
	\begin{sd}
		\coordinate (B1) ;
		\coordinate[left=\hsep of B1] (B2) ; 
		\coordinate[above =1/2 of B2] (mB) ;
		\draw (B1)--++(0,1/2) to[out=90,in=90]  (mB)
		to[out=-90,in=-90] ($(B2)!1/2!(mB)+(-\loopdiam,0)$)
		to[out=90,in=90] (B2);
	\end{sd}
	=
	\begin{sd}
		\coordinate (B1) ;
		\coordinate[right=\hsep of B1] (B2);
		\coordinate[above =1 of $(B1)!0.5!(B2)$] (m1);
		\draw (B1) to[out=90,in=0] (m1) to[out=-180,in=90] (B2) ;
	\end{sd}
	=
	\begin{sd}
		\coordinate (B1) ;
		\coordinate[right=\hsep of B1] (B2) ; 
		\coordinate[above =1/2 of B2] (mB) ;
		\draw (B1)--++(0,1/2) to[out=90,in=90]  (mB)
		to[out=-90,in=-90] ($(B2)!1/2!(mB)+(\loopdiam,0)$)
		to[out=90,in=90] (B2);
	\end{sd}.
	\label{balancedsym}
\end{align}
This directly follows from the snake equation (\ref{snakeeq}) as
\[
	\begin{sd}
		\coordinate (B1) ;
		\coordinate[left=\hsep of B1] (B2) ; 
		\coordinate[above =1/2 of B2] (mB) ;
		\draw (B1)--++(0,1/2) to[out=90,in=90]  (mB)
		to[out=-90,in=-90] ($(B2)!1/2!(mB)+(-\loopdiam,0)$)
		to[out=90,in=90] (B2);
	\end{sd}
	=
	\begin{sd}
		\coordinate (B1) ;
		\coordinate[left=\hsep of B1] (B2) ; 
		\coordinate[above =1/2 of B1] (mB1) ;
		\coordinate[above =1/2 of B2] (mB2) ;
		\draw (B1)--(mB1)   (B2)--(mB2) arc(0:180:\loopdiam/2) coordinate (mB3);
		\draw (mB3) to[out=-90,in=-90] ($(mB3)!1/2!(mB1)$) 
				to[out=90,in=90] (mB1);
	\end{sd}
	\overset{\textrm{(\ref{snakeeq})}}{=}
	\begin{sd}
		\coordinate (B1) ;
		\coordinate[right=\hsep of B1] (B2);
		\coordinate[above =1 of $(B1)!0.5!(B2)$] (m1);
		\draw (B1) to[out=90,in=0] (m1) to[out=-180,in=90] (B2) ;
	\end{sd}.
\]
Thus topological deformations through a cusp are not allowed in nontracial cases, while they are allowed in the tracial case.

Put $B=\bigoplus_s M_{n_s}$ and 
$\psi = \Tr(Q \ \cdot)=\bigoplus_s \Tr_s(Q_s \ \cdot)$,
where $\Tr=\bigoplus_s \Tr_s$ is the canonical unnormalized trace given by the sum of diagonal entries or eigenvalues, and $Q=\bigoplus_s Q_s \in B$.
Note that $Q$ is positive definite and $\Tr(Q)=\sum_s \Tr_s(Q_s) =1$ if and only if $\psi$ is a faithful state. 
Since positive matrices are unitarily diagonalizable, we may assume that $Q$ is diagonal.

Let $e_{ij,s}$ be the matrix unit of $(i,j)$ entry of $s$-th direct summand $M_{n_s}\subseteq B$,
i.e., the matrix with entries $0$ except for $(i,j)$ entry $1$ of $s$-th direct summand.

\begin{lem}\label{lem:ONBinB}
$\{\tilde{e_{ij,s}} \coloneqq e_{ij,s}Q_s^{-1/2} \mid i,j \leq n_s, s  \}$ forms an ONB for $L^2(B,\psi)$.
\end{lem}

\begin{proof}
Since $\{e_{ij,s} \mid i,j \leq n_s, s  \}$ forms an ONB for $L^2(B,\Tr)$, 
we have
\[
\braket{e_{kl,r}Q_r^{-1/2} | e_{ij,s}Q_s^{-1/2}}_\psi
= \Tr(Q(e_{kl,r}Q_r^{-1/2})^* e_{ij,s}Q_s^{-1/2})
= \Tr(e_{kl,r}^* e_{ij,s})=\delta_{ij,s}^{kl,r}
\]
where $\delta_{ij,s}^{kl,r}\coloneqq \begin{cases}1 &\textrm{if }(i,j,s)=(k,l,r) \\
0 &\textrm{otherwise} \end{cases}$.
\end{proof}

We sometimes describe operators with respect to the basis 
$\{\tilde{e_{ij,s}} = e_{ij,s}Q_s^{-1/2} \}_{ijs}$ as indicated below. 

\begin{lem}\label{lem:1psimm*} 
\begin{itemize}
\item 
$1=\sum_{ijs} (Q_s^{1/2})_{ij}\tilde{e_{ij,s}}$,

\item
$\psi: \tilde{e_{ij,s}} \mapsto (Q_s^{1/2})_{ji}$.

\item
$m: \tilde{e_{ij,s}} \otimes \tilde{e_{kl,r}} \mapsto 
 \delta_{rs} (Q_s^{-1/2})_{jk} \tilde{e_{il,r}}$.

\item
$m^\dagger: \tilde{e_{ij,s}} \mapsto 
\sum_{u,v\leq n_s} (Q_s^{-1/2})_{vu} \tilde{e_{iu,s}} \otimes \tilde{e_{vj,s}}$.
\end{itemize}
\end{lem}

\begin{proof}
Simple computations show $\braket{\tilde{e_{ij,s}} | 1}=\Tr(Q_s^{1/2}e_{ji,s})=(Q_s^{1/2})_{ij}$,
$\psi(\tilde{e_{ij,s}})=\Tr(Q_s^{1/2}e_{ij,s})=(Q_s^{1/2})_{ji}$,
$\tilde{e_{ij,s}} \tilde{e_{kl,r}} = e_{ij,s}Q_s^{-1/2} e_{kl,r}Q_r^{-1/2}
=\delta_{rs} (Q_s^{-1/2})_{jk} \tilde{e_{il,r}}$,
and
\begin{align*}
	\braket{\tilde{e_{ku,r}} \otimes \tilde{e_{vl,r}}|m^\dagger \tilde{e_{ij,s}}}_{\psi \otimes \psi}
	&=\braket{\tilde{e_{ku,r}}\tilde{e_{vl,r}} | \tilde{e_{ij,s}}}_\psi
	=\braket{(Q_s^{-1/2})_{uv} \tilde{e_{kl,r}} | \tilde{e_{ij,s}}}_\psi
	\\
	&=\ol{(Q_s^{-1/2})_{uv}} \delta_{ij,s}^{kl,r} 
	= (Q_s^{-1/2})_{vu} \delta_{ij,s}^{kl,r}.
\end{align*}
\end{proof}

\begin{rmk}
Brannan et al. \cite[Lemma 3.2]{Brannan2020quantum} uses another unnormalized orthogonal basis 
$\{f_{ij,s} \coloneqq Q_s^{-1/2}e_{ij,s}Q_s^{-1/2}\}$
for diagonal $Q$ in order to simplify the expression of $m^\dagger$ and prevent the square root $Q^{1/2}$ from appearing in the coefficients above.
In this paper, we choose $\{\tilde{e_{ij,s}}\}$ because we later use matrix expressions of operators with respect to this ONB to compute quantum automorphism groups.
\end{rmk}

\begin{prop}[{Banica \cite[section 1]{Banica2002quantum}}]\label{prop:deltaform}
In this terminology, $\psi$ is a $\delta$-form on $B$ 
if and only if $\Tr_s(Q_s^{-1})=\delta^2$ holds for all indices $s$.
\end{prop}

\begin{proof}
By Lemma \ref{lem:1psimm*}, $\psi$ is a $\delta$-form 
if and only if it holds for all $i,j,s$ that
\begin{align*}
\delta^2 \tilde{e_{ijs}}&=m m^\dagger \tilde{e_{ijs}}
=m \sum_{u,v\leq n_s} (Q_s^{-1/2})_{vu} \tilde{e_{iu,s}} \otimes \tilde{e_{vj,s}}
\\
&=\sum_{u,v\leq n_s} (Q_s^{-1/2})_{vu} (Q_s^{-1/2})_{uv} \tilde{e_{ij,s}}
\\
&=\sum_{v\leq n_s} (Q_s^{-1})_{vv} \tilde{e_{ij,s}}
=\Tr_s(Q_s^{-1})\tilde{e_{ij,s}},
\end{align*}
i.e., $\Tr_s(Q_s^{-1})=\delta^2$ for all $s$.
\end{proof}

\begin{lem}\label{balanceloop}
We have $\begin{sd}
	\coordinate (B1) ;
	\coordinate[above =1/2 of B1] (B1a) ;
	\draw (B1) to[out=90,in=90] ($(B1)!1/2!(B1a)+(-\loopdiam,0)$)
	to[out=-90,in=-90] (B1a);
\end{sd}
=Q^{-1} (\cdot) Q$
and
$\begin{sd}
	\coordinate (B1) ;
	\coordinate[above =1/2 of B1] (B1a) ;
	\draw (B1) to[out=90,in=90] ($(B1)!1/2!(B1a)+(\loopdiam,0)$)
	to[out=-90,in=-90] (B1a);
\end{sd}
=Q (\cdot) Q^{-1}$.
\end{lem}

\begin{proof}
It holds for $x,y \in B$ that
\begin{align*}
\psi(yx)&=\Tr(Qyx) = \Tr(xQy) = \psi(Q^{-1}xQy) = \psi(xQyQ^{-1})
\\
	i.e., &\begin{sd}
		\node[circle,draw] (B1) {$x$};
		\node[circle,draw,right=\hsep/2 of B1] (B2) {$y$};
		\coordinate[above =1 of $(B1)!0.5!(B2)$] (m1);
		\draw (B1) to[out=90,in=0] (m1) to[out=-180,in=90] (B2) ;
	\end{sd}
	=
	\begin{sd}
		\node[circle,draw] (B1) {$y$};
		\node[draw,left=\hsep/2 of B1] (B2) {$Q^{-1}xQ$}; 
		\draw (B1) to[out=90,in=90] (B2);
	\end{sd}
	=
	\begin{sd}
		\node[draw] (B1) {$QyQ^{-1}$};
		\node[circle,draw,left=\hsep/2 of B1] (B2) {$x$}; 
		\draw (B1) to[out=90,in=90] (B2);
	\end{sd}.
\end{align*}
Comparing above with (\ref{balancedsym}), we obtain
$\begin{sd}
	\coordinate (B1) ;
	\coordinate[above =1/2 of B1] (B1a) ;
	\draw (B1) to[out=90,in=90] ($(B1)!1/2!(B1a)+(-\loopdiam,0)$)
	to[out=-90,in=-90] (B1a);
\end{sd}
=Q^{-1} (\cdot) Q$
and
$\begin{sd}
	\coordinate (B1) ;
	\coordinate[above =1/2 of B1] (B1a) ;
	\draw (B1) to[out=90,in=90] ($(B1)!1/2!(B1a)+(\loopdiam,0)$)
	to[out=-90,in=-90] (B1a);
\end{sd}
=Q (\cdot) Q^{-1}$
by the faithfulness of $\psi$.
\end{proof}

\begin{lem}
A $\delta$-form $\psi$ on $B$ satisfies $\delta^2 \geq \abs{B}=\dim B$,
with equality if and only if $\psi$ is tracial. 
\end{lem}

\begin{proof}
For $\psi= \bigoplus_s \Tr_s(Q_s \ \cdot)$ as above, 
Cauchy-Schwarz inequality with respect to $\Tr_s$ gives us
\[
n_s^2 = (\Tr_s(Q_s^{-1/2} Q_s^{1/2}))^2
\leq \Tr_s(Q_s^{-1}) \Tr_s(Q_s)
\overset{(\textrm{Proposition \ref{prop:deltaform}})}{=} \delta^2 \Tr_s(Q_s).
\]
Hence $1=\Tr(Q) \geq \sum_s n_s^2 /\delta^2=\abs{B}/\delta^2$ shows 
$\delta^2 \geq \abs{B}$, with equality if and only if 
\[
Q_s^{1/2}=q_s Q_s^{-1/2}
\iff Q_s=q_s 1_s
\]
for some constant $q_s$ for every $s$, i.e., $\psi$ is tracial.
\end{proof}

\begin{prop}[{Banica \cite[Proposition 2.1]{Banica1999symmetries}}]
There exists a unique tracial $\delta$-form $\tau$ on $B$,
and then $\delta^2=\abs{B}$. 
$\tau$ is explicitly given by $\tau=\bigoplus_s \frac{n_s}{\abs{B}}\Tr_s$
with $Q_s=\frac{n_s}{\abs{B}}1_s$. 
Moreover $\tau$ is the so-called \emph{Plancherel trace}, the restriction of the unique tracial state of $B(L^2(B,\psi))$ via left regular representation $B \hookrightarrow B(L^2(B,\psi))$.
\end{prop}

\begin{proof}
Let $\tau=\bigoplus_s \Tr_s(Q_s \ \cdot)$ be 
a tracial $\delta$-form on $B=\bigoplus_s M_{n_s}$. 
Traciality implies $Q_s=q_s 1_s$ for some $q_s>0$ for each $s$, 
and hence $\delta^2=\Tr_s(Q_s^{-1})=q_s^{-1} n_s$ by Proposition \ref{prop:deltaform}.
Then 
\[
1=\tau(1)=\sum_s \Tr_s(q_s 1_s)=\sum_s q_s n_s=\sum_s n_s^2 / \delta^2 
=\frac{\abs{B}}{\delta^2},
\]
therefore we have $\delta^2=\abs{B}$ and $Q_s=\frac{n_s}{\abs{B}}1_s$.
\cite[Proposition 2.1]{Banica1999symmetries} states that 
a tracial state $\tau$ satisfies $m m^\dagger=\delta^2 \id$ 
if and only if $\tau$ is the ristriction of the unique tracial state of $B(L^2(B,\psi))$.
\end{proof}

\begin{rmk}
Since commutativity $xy=yx$ implies traciality $\tau(xy)=\tau(yx)$, a commutative quantum set is the pair $(\C^n,\tau)$
of an $n\times n$ diagonal matrix algebra $\C^n$ and its normalized trace $\tau=\Tr/n$, which corresponds to the pair of an $n$-element set and the uniform probability measure as in Lemma \ref{lem:classicalset}.
\end{rmk}

In string diagram notation, involution and adjoint are related via twisted wires.
The equality $x^\dagger=\bra{x}=\psi(x^* \cdot)=\psi m(x^* \otimes \id_B )$ shows the identity
\begin{align}
\begin{sd}
\node[circle,draw] (x`) {$x^\dagger$};
\draw (x`)--++(0,-1);
\end{sd}
=
\begin{sd}
\node[circle,draw] (x*) {$x^*$};
\draw (x`.north) arc(180:0:\hsep/3)--++(0,-1);
\end{sd}
,\textrm{ hence }
\begin{sd}
\node[circle,draw] (x`) {$x^\dagger$};
\draw (x`.south) arc(-180:0:\hsep/3)--++(0,1);
\end{sd}
=
\begin{sd}
\node[circle,draw] (x*) {$x^*$};
\draw (x`) --++(0,1);
\end{sd}. \label{dagger-*}
\end{align}
This gives a characterization of $*$-preserving (also called real) operators in terms of string diagrams. 

\begin{lem}\label{lem:*-pres}
Let $(B,\psi)$ be a quantum set. Then an operator $f:B \to B$ is $*$-preserving
if and only if the following equality holds:
\begin{align}
\begin{sd}
\node[circle,draw] (f) {$f^\dagger$};
\draw (f.south) arc(-180:0:\hsep/3) --++(0,1)
		(f.north) arc(0:180:\hsep/3) --++(0,-1);
\end{sd}
=
\begin{sd}
\node[circle,draw] (f) {$f$};
\draw (f) --++(0,-1/2)
		(f)  --++(0,1/2);
\end{sd}.			\label{eqn:f*-pres}
\end{align}
\end{lem}

\begin{proof}
For $x \in B$, $f(x^*)^*$ is formulated in string diagrams as
\begin{align*}
f(x^*)^*
=
\left(
\begin{sd}
\node[circle,draw] (f) {$f$};
\node[circle,draw,below=1/4 of f] (x) {$x^*$};
\draw (f) --(x)
		(f)  --++(0,1/2) ;
\end{sd}
\right)^*
\overset{\textrm{(\ref{dagger-*})}}{=}
\left(
\begin{sd}
\node[circle,draw] (f) {$f$};
\node[circle,draw,left=1/4 of f] (x) {$x^\dagger$};
\draw (f) to[out=-90,in=-90] (x)
		(f)  --++(0,1/2) ;
\end{sd}
\right)^*
\overset{\textrm{(\ref{dagger-*})}}{=}
\begin{sd}
\node[circle,draw] (f) {$f^\dagger$};
\draw (f.south) arc(-180:0:\hsep/3) --++(0,1)
		(f.north) arc(0:180:\hsep/3) --++(0,-1/2) 
		node[circle,draw,fill=white] (x) {$x$};
\end{sd}.
\end{align*}
Therefore $f(x^*)^*=f(x) \ \forall x \in B$ is exactly equal to the desired equality.
\end{proof}

\begin{rmk}
Note that bending strings in the other direction can result in different operators, like
\begin{align*}
\begin{sd}
\node[circle,draw] (x*) {$x^\dagger$};
\draw (x*.south) arc(0:-180:\hsep/4)--++(0,2/3);
\end{sd}
&\overset{\textrm{(\ref{dagger-*})}}{=}
\begin{sd}
\node[circle,draw] (x`) {$x^*$};
\draw (x`.north)to[out=90,in=90]++(\loopdiam,1/2) to[out=-90,in=-90]++(-\loopdiam,1/2);
\end{sd}
\overset{\textrm{Lemma \ref{balanceloop}}}{=}
Qx^*Q^{-1}
\neq x^*,
\\
\begin{sd}
\node[circle,draw] (f) {$f^\dagger$};
\draw (f.south) arc(0:-180:\hsep/3) --++(0,1)
		(f.north) arc(180:0:\hsep/3) --++(0,-1/2) 
		node[circle,draw,fill=white] (x) {$x$};
\end{sd}
&\overset{\text{(flip)}}{=}
\begin{sd}
\node[circle,draw] (f) {$f^\dagger$};
\draw (f.south) to[out=-90,in=0]++(\hsep/3,-\loopdiam*1.3)
		to[out=-180,in=-90]++(\hsep/3,\loopdiam*1.3) --++(0,1)
		(f.north) to[out=90,in=180]++(-\hsep/3,\loopdiam*1.3)
		to[out=0,in=90]++(-\hsep/3,-\loopdiam*1.3) --++(0,-1/2) 
		node[circle,draw,fill=white] (x) {$x$};
\end{sd}
\overset{\text{Lemma \ref{balanceloop}}}{=}
Q f(Qx^*Q^{-1})^* Q^{-1}
\neq f(x^*)^*.
\end{align*}
In particular, $f^\dagger$ is not necessarily $*$-preserving even if $f:B \to B$ is $*$-preserving. Indeed $f^\dagger$ is $*$-preserving if and only if
\begin{align}
\begin{sd}
\node[circle,draw] (f) {$f$};
\draw (f.south) arc(-180:0:\hsep/3) --++(0,1)
		(f.north) arc(0:180:\hsep/3) --++(0,-1);
\end{sd}
=f^\dagger
, \text{ i.e., }
f=
\begin{sd}
\node[circle,draw] (f) {$f^\dagger$};
\draw (f.south) arc(0:-180:\hsep/3) --++(0,1)
		(f.north) arc(180:0:\hsep/3) --++(0,-1) ;
\end{sd}	\label{eqn:f'*-pres}
\end{align}
is satisfied, but the RHS is not necessarily equal to $*$-preserving $f=f((\cdot)^*)^*$ as above.
\end{rmk}

\begin{prop}
Given $*$-preserving operator $f:B \to B$, $f^\dagger$ is also $*$-preserving if and only if $f$ commutes with 
$\begin{sd}
	\coordinate (B1) ;
	\coordinate[above =1/2 of B1] (B1a) ;
	\draw (B1) to[out=90,in=90] ($(B1)!1/2!(B1a)+(-\loopdiam,0)$)
	to[out=-90,in=-90] (B1a);
\end{sd}
=Q^{-1} (\cdot) Q$.
\end{prop}

\begin{proof}
If $f^\dagger$ is also $*$-preserving, then (\ref{eqn:f'*-pres}) and the adjoint of (\ref{eqn:f*-pres}) shows
\[
\begin{sd}
\node[circle,draw] (f) {$f$};
\draw (f.south) arc(-180:0:\hsep/3) --++(0,1)
		(f.north) arc(0:180:\hsep/3) --++(0,-1);
\end{sd}
\overset{\textrm{(\ref{eqn:f'*-pres})}}{=} f^\dagger 
\overset{\textrm{(\ref{eqn:f*-pres})}}{=}
\begin{sd}
\node[circle,draw] (f) {$f$};
\draw (f.south) arc(0:-180:\hsep/3) --++(0,1)
		(f.north) arc(180:0:\hsep/3) --++(0,-1) ;
\end{sd}.
\]
By bending the bottom string counterclockwise and the top string clockwise, we obtain
\[
\begin{sd}
\node[circle,draw] (f) {$f$};
\draw (f) --++(0,-1/2)
		(f.north) to[out=90,in=90]++(-\loopdiam,1/4) 
		to[out=-90,in=-90]++(\loopdiam,1/4);
\end{sd}
=
\begin{sd}
\node[circle,draw] (f) {$f$};
\draw (f) --++(0,1/2)
		(f.south) to[out=-90,in=-90]++(-\loopdiam,-1/4) 
		to[out=90,in=90]++(\loopdiam,-1/4);
\end{sd}.
\]
Conversely if $f$ commutes with 
$\begin{sd}
	\coordinate (B1) ;
	\coordinate[above =1/2 of B1] (B1a) ;
	\draw (B1) to[out=90,in=90] ($(B1)!1/2!(B1a)+(-\loopdiam,0)$)
	to[out=-90,in=-90] (B1a);
\end{sd}$, then we can go back to
\[
\begin{sd}
\node[circle,draw] (f) {$f$};
\draw (f.south) arc(-180:0:\hsep/3) --++(0,1)
		(f.north) arc(0:180:\hsep/3) --++(0,-1);
\end{sd}
=
\begin{sd}
\node[circle,draw] (f) {$f$};
\draw (f.south) arc(0:-180:\hsep/3) --++(0,1)
		(f.north) arc(180:0:\hsep/3) --++(0,-1) ;
\end{sd} 
\overset{\textrm{(\ref{eqn:f*-pres})}}{=}
f^\dagger.
\] 
\end{proof}

In the case of $B=B(H)$ for a finite dimensional Hilbert space $H$, 
operators in $B(H) \cong H \otimes H^*$ can be expressed by strings of $H$ and $H^*$ under the identification
$H \otimes H^* \ni \ket{v} \otimes \bra{w} \leftrightarrow \ket{v}\bra{w} \in B(H)$.
This identification is formulated in string diagrams as 
\[
B(H) \ni T
\leftrightarrow
\begin{sd}
\coordinate (H1);
\coordinate[right=\hsep*2/3] (H2);
\node[circle,draw,below=1/2 of H1] (T) {$T$};
\draw[->-] (T) -- (H1);
\draw[-<-] (T.south) arc(-180:0:\hsep/3) -- (H2);
\end{sd}
\in H \otimes H^* .
\]
Recall that the strings of $H$ are oriented from bottom to top and those of $H^*$ from top to bottom.

\begin{prop}[{Musto, Reutter, Verdon \cite[Definition 2.5]{Musto2018compositional}}]\label{prop:HH*=B(H)}
By the identification above, the canonical operators of $(B(H),\tau=\tau_{B(H)})$ is formulated in string diagrams of $H \otimes H^*$ as follows:
\[
\id_H=\begin{sd}
	\draw[-<-]  (0,0) arc(-180:0:\hsep/4) ;
\end{sd}, \
m=\begin{sd}
	\coordinate (H1) at (-\hsep/2,0);  
	\coordinate (H2) at (\hsep/2,0)  ;
	\coordinate (H3) at (0,\hsep);
	\draw[->-] ([xshift=-0.2cm]H1) to[out=90,in=-90] ([xshift=-0.2cm]H3) ;
	\draw[->-] ([xshift=-0.2cm]H2) to[out=90,in=90] ([xshift=0.2cm]H1) ;
	\draw[-<-] ([xshift=0.2cm]H2) to[out=90,in=-90] ([xshift=0.2cm]H3) ;
\end{sd}, \
\tau=\frac{\Tr}{\abs{H}}=\frac{1}{\abs{H}}\begin{sd}
	\draw[->-]  (0,0) arc(180:0:\hsep/4) ;
\end{sd}, \
m^\dagger=\abs{H}\begin{sd}  
	\coordinate (H1) at (-\hsep/2,0);  
	\coordinate (H2) at (\hsep/2,0)  ;
	\coordinate (H3) at (0,-\hsep);
	\draw[-<-] ([xshift=-0.2cm]H1) to[out=-90,in=90] ([xshift=-0.2cm]H3) ;
	\draw[-<-] ([xshift=-0.2cm]H2) to[out=-90,in=-90] ([xshift=0.2cm]H1) ;
	\draw[->-] ([xshift=0.2cm]H2) to[out=-90,in=90] ([xshift=0.2cm]H3) ;
\end{sd}.
\]
\end{prop}

\begin{proof}
The equality about $\id_H$ is by the identification.
Since the multiplication in $B(H)$ is the composition, the equality about $m$ directly follows from the snake equation (\ref{Hsnake}).
Let $\{v_i \}_i$ be an ONB for $H$. Then
\[
\begin{sd}
	\draw[->-]  (0,0) arc(180:0:\hsep/2) ;
	\node[circle,draw,fill=white](H1) at (0,0) {$v_i$} ;
	\node[circle,draw,fill=white](H2) at (\hsep,0) {$v_j^\dagger$} ;
\end{sd}
= \braket{v_j|v_i}=\delta_{ij}
\]
shows the equality about $\tau$. 
Note that $H \otimes H^*$ is equipped with the inner product
\begin{align*}
\braket{v_1 \otimes w_1^\dagger | v_0 \otimes w_0^\dagger}_{H \otimes H^*}
&=\braket{v_1| v_0}_H  \braket{w_1^\dagger | w_0^\dagger}_{H^*}
=\braket{v_1| v_0}_H  \braket{w_0| w_1}_{H}
\\
&=\Tr \left( \ket{w_1} \bra{v_1} \ket{v_0} \bra{w_0} \right)
=\Tr \left( (\ket{v_1} \bra{w_1})^\dagger (\ket{v_0} \bra{w_0}) \right)
\\
&=\left\langle \ (\ket{v_1} \bra{w_1}) \ \middle| \ (\ket{v_0} \bra{w_0}) \ \right\rangle_{\Tr},
\end{align*}
hence the diagram $\begin{sd}  
	\coordinate (H1) at (-\hsep/2,0);  
	\coordinate (H2) at (\hsep/2,0)  ;
	\coordinate (H3) at (0,-\hsep);
	\draw[-<-] ([xshift=-0.2cm]H1) to[out=-90,in=90] ([xshift=-0.2cm]H3) ;
	\draw[-<-] ([xshift=-0.2cm]H2) to[out=-90,in=-90] ([xshift=0.2cm]H1) ;
	\draw[->-] ([xshift=0.2cm]H2) to[out=-90,in=90] ([xshift=0.2cm]H3) ;
\end{sd}$ 
is the adjoint of $m$ with respect to $\Tr$. Therefore the adjoint of $m$ with respect to $\tau=\Tr/\abs{H}$ is as stated.
\end{proof}

Replacing $H$ with $L^2(B,\psi)\cong B$, we have the same result for strings of $B$. 

\begin{cor}\label{cor:B(B)}
By an identification  
\[
B(L^2(B,\psi)) \ni T
\leftrightarrow
\begin{sd}
\coordinate (H1);
\coordinate[right=\hsep*2/3] (H2);
\node[circle,draw,below=1/3 of H1] (T) {$T$};
\draw (T) -- (H1);
\draw (T.south) arc(-180:0:\hsep/3) -- (H2);
\end{sd}
\in B \otimes B,
\]
the canonical operators of $(B(L^2(B,\psi)),\tau_{B(B)}=\Tr_{B(B)}/\abs{B})$ is formulated in string diagrams of $B \otimes B$ as follows:
\[
\id_B=
\begin{sd}
	\draw  (0,0) arc(-180:0:\hsep/4) ;
\end{sd}, \
m_{B(B)}=
\begin{sd}
	\coordinate (H1) at (-\hsep/2,0);  
	\coordinate (H2) at (\hsep/2,0)  ;
	\coordinate (H3) at (0,\hsep);
	\draw ([xshift=-0.2cm]H1) to[out=90,in=-90] ([xshift=-0.2cm]H3) ;
	\draw ([xshift=-0.2cm]H2) to[out=90,in=90] ([xshift=0.2cm]H1) ;
	\draw ([xshift=0.2cm]H2) to[out=90,in=-90] ([xshift=0.2cm]H3) ;
\end{sd}, \
\tau_{B(B)}=\frac{1}{\abs{B}}
\begin{sd}
	\draw  (-\hsep/4,0) to[out=60,in=0] (0,\loopdiam*1.5)
		 to[out=-180,in=120] (\hsep/4,0);
\end{sd}, \
m_{B(B)}^\dagger
=\abs{B}
\begin{sd}  
	\coordinate (H1) at (-\hsep/2,0);  
	\coordinate (H2) at (\hsep/2,0)  ;
	\coordinate (H3) at (0,-\hsep);
	\draw ([xshift=-0.2cm]H1) to[out=-90,in=90] ([xshift=-0.2cm]H3) ;
	\draw ([xshift=-0.2cm]H2)  to[out=-120,in=180] 
			($([xshift=-0.2cm]H2)!0.5!([xshift=0.2cm]H1)-(0,\loopdiam*1.5)$)
		 to[out=0,in=-60]  ([xshift=0.2cm]H1) ;
	\draw ([xshift=0.2cm]H2) to[out=-90,in=90] ([xshift=0.2cm]H3) ;
\end{sd}.
\]
\end{cor}

\begin{proof}
The statement directly follows from the previous proposition by the identification 
\begin{align}
B=L^2(B,\psi)\ni y=\ket{y} \leftrightarrow (y^*)^\dagger=\bra{y^*} \in L^2(B,\psi)^*		\label{eqn:L2*=L2}
\end{align}
because $\Tr(\ket{x} \bra{y^*}) =\braket{y^*|x}_\psi=\psi(yx)=
\begin{sd}
	\draw (-0.3,0)node[circle,draw](x){$x$} (0.3,0)node[circle,draw](y){$y$};
	\draw  (x) to[out=45,in=0] (0,1/2)
		 to[out=-180,in=135] (y);
\end{sd}$ for any $x,y  \in B$.
Since 
\[
\begin{sd}
	\draw[->-]  (0,0) arc(-180:0:\hsep/4) ;
\end{sd}
=
\sum_i b_i^\dagger \otimes b_i
\overset{\textrm{(\ref{eqn:L2*=L2})}}{\leftrightarrow} \sum_i b_i^* \otimes b_i
=\sum_i
\begin{sd}
	\draw (-\hsep/2,0) node[draw] (x){$b_i^\dagger$} 
			(\hsep/3,0) node[draw] (y){$b_i$};
	\draw (x)to[out=-60,in=-90](0,1/3) (y)--(\hsep/3,1/3);
\end{sd}
=
\begin{sd}
	\draw  (-\hsep/4,0) to[out=-60,in=0] (0,-\loopdiam*1.5)
		 to[out=-180,in=-120] (\hsep/4,0);
\end{sd}
\]
holds for an ONB $\{ b_i\}$ for $L^2(B,\psi)$, 
$m^\dagger$ with respect to $\tau_{B(B)}$ is as stated.
\end{proof}

The balancing loop in the trace is caused by the discrepancy between the inner products $\langle\cdot|\cdot\rangle_{\Tr_{B(B)}}$ and $\langle\cdot|\cdot\rangle_{\psi\otimes\psi}$ on $B(L^2(B,\psi))=B \otimes B$.
If we replace $\begin{sd}
	\draw  (-\hsep/4,0) to[out=60,in=0] (0,\loopdiam*1.5)
		 to[out=-180,in=120] (\hsep/4,0);
\end{sd}$
 with
$\begin{sd}
	\draw  (0,0) arc(180:0:\hsep/4) ;
\end{sd}$,
then we obtain a nontracial  $B(L^2(B,\psi))$.

\begin{cor}
By the same identification $B(L^2(B,\psi))=B \otimes B$ as in Corollary \ref{cor:B(B)}, 
$\tilde{\psi} \coloneqq \delta^{-2}\begin{sd}
	\draw  (0,0) arc(180:0:\hsep/4) ;
\end{sd}$ is a $\delta^2$-form on $B(L^2(B,\psi))$ with cannonical operators
\[
\id_B=\begin{sd}
	\draw  (0,0) arc(-180:0:\hsep/4) ;
\end{sd}, \
m_{B(B)}=\begin{sd}
	\coordinate (H1) at (-\hsep/2,0);  
	\coordinate (H2) at (\hsep/2,0)  ;
	\coordinate (H3) at (0,\hsep);
	\draw ([xshift=-0.2cm]H1) to[out=90,in=-90] ([xshift=-0.2cm]H3) ;
	\draw ([xshift=-0.2cm]H2) to[out=90,in=90] ([xshift=0.2cm]H1) ;
	\draw ([xshift=0.2cm]H2) to[out=90,in=-90] ([xshift=0.2cm]H3) ;
\end{sd}, \
\tilde{\psi}=\delta^{-2}\begin{sd}
	\draw  (0,0) arc(180:0:\hsep/4) ;
\end{sd}, \
m_{B(B)}^\dagger=\delta^2 \begin{sd}  
	\coordinate (H1) at (-\hsep/2,0);  
	\coordinate (H2) at (\hsep/2,0)  ;
	\coordinate (H3) at (0,-\hsep);
	\draw ([xshift=-0.2cm]H1) to[out=-90,in=90] ([xshift=-0.2cm]H3) ;
	\draw ([xshift=-0.2cm]H2) to[out=-90,in=-90] ([xshift=0.2cm]H1) ;
	\draw ([xshift=0.2cm]H2) to[out=-90,in=90] ([xshift=0.2cm]H3) ;
\end{sd}.
\]
\end{cor}

\begin{proof}
The unit and the multiplication are those in Corollary \ref{cor:B(B)}. 
In the same way as Proposition \ref{prop:HH*=B(H)}, 
we have $\langle\cdot|\cdot\rangle_{\delta^2 \tilde{\psi}}=\langle\cdot|\cdot\rangle_{\psi\otimes\psi}$, and hence $\tilde{\psi}$ is faithful and $m_{B(B)}^\dagger$ with respect to $\tilde{\psi}$ is as stated.
Since $\begin{sd} 
	\draw (0,0)arc(0:360:\hsep/4);
\end{sd}=\delta^2$, $\tilde{\psi}$ is a state and
we have $m_{B(B)}m_{B(B)}^\dagger=(\delta^2)^2 \id_{B(B)}$.
\end{proof}

Note that $m_B^\dagger:B \to B\otimes B$ with $m_{B(B)}$ as above is a $*$-homomorphism that corresponds to the left regular representation of $B$ (cf. Vicary \cite[Lemma 3.19, 3.20]{Vicary2011categorical}).
If we identify $L^2(B,\psi)^*$ with the left tensorand $B$ instead of the right one, then $m_B^\dagger$ corresponds to the right regular representation.
The Frobenius equality (\ref{Frobeniuseq}) means that the left and right regular representations are $*$-homomorphisms.

\subsection{Quantum graphs}

Recall that a simple (i.e., without multiple edges) finite classical graph is a pair $(V,E)$ of finite vertex set $V$ and edge set $E\subseteq V \times V$, and the adjacency operator of $(V,E)$ is the matrix $A\in M_V(\C)\cong B(C(V))$ such that the $(v,w)$ entry $A_{v,w}=1$ if $(v,w)\in E$ and $A_{v,w}=0$ otherwise,
where $M_V(\C)$ denotes the matrix algebra whose raws and columns are indexed by $V$.
This notion of adjacency operator is generalized on quantum sets as follows.

\begin{dfn}[{Musto, Reutter, Verdon \cite[Definition 5.1]{Musto2018compositional}, Brannan et al. \cite[Definition 3.4]{Brannan2019bigalois}}]
We define a \emph{quantum adjacency operator} on a quantum set $(B,\psi)$ as an operator $A: B \to B$ satisfying \emph{Schur idempotence}
\begin{align}
	\begin{sd}
	\node[circle,draw] (G1) {$A$};
	\node[circle,draw,right=\hsep/2 of G1)] (G2) {$A$};
	\draw (G1) to[out=90,in=90] coordinate[pos=1/2] (m) (G2);
	\draw (G1) to[out=-90,in=-90] coordinate[pos=1/2] (m*) (G2);
	\draw (m)--++(0,\hsep/2);
	\draw (m*)--++(0,-\hsep/2);
	\end{sd}
	=
	\delta^2
	\begin{sd}
	\node[circle,draw] (G1) {$A$};
	\draw (G1)--++(0,\hsep);
	\draw (G1)--++(0,-\hsep);
	\end{sd},
	\textrm{ i.e., } 
	m(A \otimes A)m^\dagger = \delta^2 A,	\label{Schuridemp}
\end{align}
and then we call $\mathcal{G}=(B,\psi,A)$, or simply $A$, a quantum graph on $(B,\psi)$.

We say that a quantum graph $(B,\psi,A)$ is
\begin{itemize}
\item \emph{self-adjoint} if $A$ is self-adjoint $A^\dagger=A$;

\item \emph{self-transpose} if 
$\begin{sd}
	\node[circle,draw] (G) {$A$};
	\draw (G.north) to[out=90,in=90] ([xshift=-\hsep/2]G.north) --++(0,-3/4);
	\draw (G.south) to[out=-90,in=-90] ([xshift=\hsep/2]G.south) --++(0,3/4);
\end{sd}=\begin{sd}
	\node[circle,draw] (G1) {$A$};
	\draw (G1)--++(0,\hsep/2);
	\draw (G1)--++(0,-\hsep/2);
\end{sd}
\ (\overset{\textrm{(\ref{snakeeq})}}{\iff}
\begin{sd}
	\node[circle,draw] (G1) {$A$};
	\draw (G1)--++(0,\hsep/2);
	\draw (G1)--++(0,-\hsep/2);
\end{sd}
=
\begin{sd}
	\node[circle,draw] (G) {$A$};
	\draw (G.north) to[out=90,in=90] ([xshift=\hsep/2]G.north) --++(0,-3/4);
	\draw (G.south) to[out=-90,in=-90] ([xshift=-\hsep/2]G.south) --++(0,3/4);
\end{sd})$;

\item \emph{real} (\emph{$*$-preserving}) if $Ab^*=(Ab)^* \ \forall b\in B$
$\ (\overset{\textrm{Lemma \ref{lem:*-pres}}}{\iff} \begin{sd}
	\node[circle,draw] (G) {$A^\dagger$};
	\draw (G.north) to[out=90,in=90] ([xshift=-\hsep/2]G.north) --++(0,-3/4);
	\draw (G.south) to[out=-90,in=-90] ([xshift=\hsep/2]G.south) --++(0,3/4);
\end{sd}=\begin{sd}
	\node[circle,draw] (G1) {$A$};
	\draw (G1)--++(0,\hsep/2);
	\draw (G1)--++(0,-\hsep/2);
\end{sd})$; 

\item \emph{undirected} if $A$ is real and self-adjoint.

\item \emph{reflexive} if $\begin{sd}
	\node[circle,draw] (G1) {$A$};
	\draw ([xshift=\hsep/2]G1.south)--([xshift=\hsep/2]G1.north);
	\draw (G1) to[out=90,in=90] coordinate[pos=1/2] (m) 
([xshift=\hsep/2]G1.north);
	\draw (G1) to[out=-90,in=-90] coordinate[pos=1/2] (m*) ([xshift=\hsep/2]G1.south);
	\draw (m)--++(0,\hsep/3);
	\draw (m*)--++(0,-\hsep/3);
\end{sd}
=
\delta^2
\begin{sd}
	\draw (0,0)--++(0,1.5);
\end{sd}$, 
and \emph{irreflexive} if 
$\begin{sd}
	\node[circle,draw] (G1) {$A$};
	\draw ([xshift=\hsep/2]G1.south)--([xshift=\hsep/2]G1.north);
	\draw (G1) to[out=90,in=90] coordinate[pos=1/2] (m) ([xshift=\hsep/2]G1.north);
	\draw (G1) to[out=-90,in=-90] coordinate[pos=1/2] (m*) ([xshift=\hsep/2]G1.south);
	\draw (m)--++(0,\hsep/3);
	\draw (m*)--++(0,-\hsep/3);
\end{sd}
=0$.
\end{itemize}
\end{dfn}

\begin{rmk}
In the classical case $(\C^n,\tau)$, 
\emph{Schur product} $f \bullet g$ of operators $f,g \in M_n \cong B(\C^n)$ is defined as the entrywise product. In fact it is realized as 
\[
f \bullet g =m(f \otimes g)m^\dagger/\delta^2 .
\]
 That is why the condition (\ref{Schuridemp}) is called Schur idempotence. Since a matrix is Schur idempotent if and only if it is $\{0,1\}$-valued, quantum graphs $A$ on $(\C^n,\tau)$ are exactly equal to the adjacency operators of classical multiplicity-free graphs on $n$ vertices. 
Then $A$ is always real because $A$ is real-valued and $*$ is just a complex conjugate, and $A$ is undirected if and only if the graph $(V, E)$ is undirected, i.e., any edge $(v,w)\in E$ has its opposite $(w,v)\in E$. 
It seems natural to call the graph symmetric instead of self-transpose, but it is confusing with symmetricity (traciality) of a quantum set, so we use the term self-transpose. 
A classical graph is called reflexive if it has all self-loops $(v,v)$ for $v\in V$, and irreflexive if it has no self-loop. 
Thus the reflexivity is characterized by the Schur product of $A$ and $\id_{\C^n}$,
which outputs the diagonal entries $A_{v,v}$.
\end{rmk}

\begin{rmk}
The notion of quantum adjacency operator is first introduced by  
Musto, Reutter, Verdon \cite[Definition 5.1]{Musto2018compositional}, 
who defined undirected quantum graphs on tracial quantum sets. 
Following \cite{Musto2018compositional},  Brannan et al. \cite[Definition 3.4]{Brannan2019bigalois} defined
undirected quantum graphs on general quantum sets.
The weakest definition assigning only Schur idempotence appears in Brannan, Eifler, Voigt, Weber \cite[Definition 3.3]{Brannan2020quantum}.
\end{rmk}

\begin{lem}\label{lem:sa-undi-*pres}
Let $(B,\psi,A)$ be a quantum graph.
Every couple of the following three conditions imply the other.
Equivalently, all couples of the following are equivalent to each other.
\begin{description}
\item[$(1)$]
$A$ is self-adjoint;
\item[$(2)$]
$A$ is self-transpose;
\item[$(3)$]
$A$ is real.
\end{description}
In particular, $A$ is undirected if and only if (1), (2), and (3) hold. 
\end{lem}

\begin{proof}
$(1) (2)\implies (3)$ We have
\[
\begin{sd}
	\node[circle,draw] (A) {$A^\dagger$};
	\draw (A.north) to[out=90,in=90] ([xshift=-\hsep/2]A.north) --++(0,-2/3);
	\draw (A.south) to[out=-90,in=-90] ([xshift=\hsep/2]A.south) --++(0,2/3);
\end{sd}
\overset{(1)}{=}
\begin{sd}
	\node[circle,draw] (A) {$A$};
	\draw (A.north) to[out=90,in=90] ([xshift=-\hsep/2]A.north) --++(0,-2/3);
	\draw (A.south) to[out=-90,in=-90] ([xshift=\hsep/2]A.south) --++(0,2/3);
\end{sd}
\overset{(2)}{=}
\begin{sd}
	\node[circle,draw] (A1) {$A$};
	\draw (A1)--++(0,\hsep/2);
	\draw (A1)--++(0,-\hsep/2);
\end{sd}.
\]
Thus $A$ is real.

\noindent $(2) (3)\implies (1)$
By using the Hilbert adjoint of real condition, we have 
\[
\begin{sd}
	\node[circle,draw] (A1) {$A^\dagger$};
	\draw (A1)--++(0,\hsep/2);
	\draw (A1)--++(0,-\hsep/2);
\end{sd}
\overset{(3)}{=}
\begin{sd}
	\node[circle,draw] (A) {$A$};
	\draw (A.north) to[out=90,in=90] ([xshift=\hsep/2]A.north) --++(0,-2/3);
	\draw (A.south) to[out=-90,in=-90] ([xshift=-\hsep/2]A.south) --++(0,2/3);
\end{sd}
\overset{(2)}{=}
\begin{sd}
	\node[circle,draw] (A1) {$A$};
	\draw (A1)--++(0,\hsep/2);
	\draw (A1)--++(0,-\hsep/2);
\end{sd}.
\]
Thus $A$ is self-adjoint.

\noindent $(3) (1)\implies (2)$
We have
\[
\begin{sd}
	\node[circle,draw] (A) {$A$};
	\draw (A.north) to[out=90,in=90] ([xshift=-\hsep/2]A.north) --++(0,-2/3);
	\draw (A.south) to[out=-90,in=-90] ([xshift=\hsep/2]A.south) --++(0,2/3);
\end{sd}
\overset{(1)}{=}
\begin{sd}
	\node[circle,draw] (A) {$A^\dagger$};
	\draw (A.north) to[out=90,in=90] ([xshift=-\hsep/2]A.north) --++(0,-2/3);
	\draw (A.south) to[out=-90,in=-90] ([xshift=\hsep/2]A.south) --++(0,2/3);
\end{sd}
\overset{(3)}{=}
\begin{sd}
	\node[circle,draw] (A1) {$A$};
	\draw (A1)--++(0,\hsep/2);
	\draw (A1)--++(0,-\hsep/2);
\end{sd}.
\]
Thus $A$ is self-transpose.
\end{proof}

Recall that an operator $A:B \to B'$ between $C^*$-algebras is called positive if it preserves the positive cone consisting of positive semidefinite elements $A(B^+)\subset {B'}^+$, and called completely positive if its amplification $A\otimes \id_{M_n}:B\otimes M_n \cong M_n(B) \to B'\otimes M_n \cong M_n(B')$ is positive for arbitrary positive integer $n$.
We can deduce the following equivalence from Schur idempotence.

\begin{prop}
Let $\mathcal{G}=(B,\psi,A)$ be a quantum graph. TFAE:
\begin{description}
\item[$(1)$] $A$ is real;
\item[$(2)$] $A$ is positive;
\item[$(3)$] $A$ is completely positive.
\end{description}
\end{prop}

\begin{proof}
$(3)\implies(2)$ Obvious by definition. $(2)\implies(1)$ Since the positive cone $B^+$ of $B$ spans the subspace $B^{sa}$ of self-adjoint operators, $A(B^+)\subset B^+$ implies $A(B^{sa})\subset B^{sa}$, i.e., 
$A$ is real.

$(1)\implies(3)$ 
Assume that $A$ is a real quantum graph on $(B,\psi)$.
Then $A \otimes \id_{M_n}$ is also a quantum graph on $(B\otimes M_n, \psi \otimes \tau_{M_n})$ for arbitrary $n$, and it is real since the involution is tensorandwise $(b\otimes x)^*=b^* \otimes x^*$ in $B\otimes M_n$.
Replacing $(B\otimes M_n, \psi \otimes \tau_{M_n}, A\otimes \id_{M_n})$ by $(B,\psi,A)$, it suffices to show that $A$ is positive. We take an arbitrary $x \in B$ and check that $A(x^* x)$ is positive semidefinite:
\begin{align*}
A(x^* x)
=
\begin{sd}  
	\coordinate (H3) at (0,-\hsep);  
	\coordinate (H1) at (0,0)  ;
	\path (H3) node[draw,below] (x){$x^*x$};
	\draw (H3) --(H1)node[midway,circle,draw,fill=white](A){$A$} ;
\end{sd}
=
\begin{sd}
	\coordinate (H1) at (-\hsep/2,0);  
	\coordinate (H2) at (\hsep/2,0)  ;
	\coordinate (H3) at (0,\hsep);
	\path (H1)node[draw,below](x'){$x^*$} (H2)node[draw,below](x){$x$} ;
	\draw (H2) to[out=90,in=90] coordinate[midway] (m) (H1) ;
	\draw (m)--(H3)node[midway,circle,draw,fill=white](A){$A$};
\end{sd}
=
\delta^{-2}
\begin{sd}
	\coordinate (H1) at (-\hsep/2,0);  
	\coordinate (H2) at (\hsep/2,0)  ;
	\coordinate (H3) at (0,8/5*\hsep);
	\path (H1)node[draw,below](x'){$x^*$} (H2)node[draw,below](x){$x$} 
		(-\hsep/3,\hsep)node[circle,draw](A1){$A$} 
		(\hsep/3,\hsep)node[circle,draw](A2){$A$};
	\draw (H2) to[out=90,in=90] coordinate[midway] (m1) (H1) ;
	\draw (A2.south) to[out=-90,in=-90] coordinate[midway] (m2) (A1.south) ;
	\draw (A2.north) to[out=90,in=90] coordinate[midway] (m3) (A1.north) ;
	\draw (m1)--(m2) (m3)--(H3);
\end{sd}
=
\delta^{-2}
\begin{sd}
	\coordinate (H1) at (-\hsep/2,0);  
	\coordinate (H2) at (\hsep/2,0)  ;
	\coordinate (H3) at (0,8/5*\hsep);
	\path (H1)node[draw,below](x'){$x^*$} (H2)node[draw,below](x){$x$} 
		(-\hsep/2,\hsep)node[circle,draw](A1){$A$} 
		(\hsep/2,\hsep)node[circle,draw](A2){$A$}
		($(x')!0.5!(A2)$)coordinate(id);
	\draw (H1)--++(0,\hsep/4) to[out=90,in=90] coordinate[midway] (m1) (id) ;
	\draw (A2.south)--++(0,-\hsep/4) to[out=-90,in=-90] coordinate[midway] (m2) (id) ;
	\draw (A2) to[out=100,in=80] coordinate[midway] (m3) (A1) ;
	\draw (m1)--(A1) (m2)--(x) (m3)--(H3);
\end{sd}.
\end{align*}
Decomposing the identity string in the middle of the diagram into 
$\id_B=\sum_i \ket{b_i}\bra{b_i}$ by an ONB $\{b_i\}_{i=1}^\abs{B}$ for $L^2(B,\psi)$, we obtain
\begin{align*}
A(x^*x)
&=
\delta^{-2} \sum_i
\begin{sd}
	\coordinate (H1) at (-\hsep*2/3,0);  
	\coordinate (H2) at (\hsep*2/3,0)  ;
	\coordinate (H3) at (0,8/5*\hsep);
	\path (H1)node[draw,below](x'){$x^*$}
		(H2)node[draw,below](x){$x$} 
		(-\hsep*2/3,\hsep)node[circle,draw](A1){$A$} 
		(\hsep*2/3,\hsep)node[circle,draw](A2){$A$}
		($(x')!0.5!(A2)+(0,\hsep*0.4)$)node[draw](b'){$b_i^\dagger$}
		($(x')!0.5!(A2)-(0,\hsep*0.4)$)node[draw](b){$b_i$};
	\draw (H1)--++(0,\hsep/4) to[out=90,in=90]
			 coordinate[pos=0.4](m1)(b) ;
	\draw (A2.south)--++(0,-\hsep/4) to[out=-90,in=-90]
			  coordinate[pos=0.4] (m2)(b') ;
	\draw (A2) to[out=120,in=60] coordinate[midway] (m3) (A1) ;
	\draw (m1)--(A1) (m2)--(x) (m3)--(H3);
\end{sd}
=
\delta^{-2} \sum_i
\begin{sd}
	\coordinate (H1) at (-\hsep/2,0);  
	\coordinate (H2) at (\hsep/2,0)  ;
	\coordinate (H3) at (0,\hsep*4/3);
	\path (H1)node[draw,below](x'){$x^* b_i$}
		(H2)node[draw,below](x){$b_i^* x$} 
		(-\hsep/2,\hsep/2)node[circle,draw](A1){$A$} 
		(\hsep/2,\hsep/2)node[circle,draw](A2){$A$};
	\draw (A2) to[out=90,in=90] coordinate[midway] (m3) (A1) ;
	\draw (x')--(A1) (A2)--(x) (m3)--(H3);
\end{sd}
\\
&=
\delta^{-2} \sum_i A(x^* b_i) A(b_i^* x)
=
\delta^{-2} \sum_i A(b_i^* x)^* A(b_i^* x) \geq0.
\end{align*}
Therefore $A$ is positive.
\end{proof}

Recall that the \emph{indegree} (resp, \emph{outdegree}) of a vertex of a classical directed graph is the number of edges into (resp. out of) the vertex, and the graph is called \emph{$d$-regular} if the indegrees and outdegrees of all vertices are equal to $d$. Note that indegree and outdegree coincide and are called the degree if the graph is undirected. Recall also that a classical graph is $d$-regular if and only if the adjacency operator $A$ and $A^\dagger$ have the constant function $1$ as an eigenvector of eigenvalue $d$. This notion can be generalized for quantum graphs:

\begin{dfn}
A quantum graph $\mathcal{G}=(B,\psi,A)$ is \emph{$d$-regular} if
\[
A1=d 1 \text{   and   } \psi A=d\psi
\]
are satisfied, and then $d$ is called the \emph{degree} of $\mathcal{G}$.
\end{dfn}

Here we exhibit typical examples of quantum graphs.
\begin{ex}[{cf. Brannan et al. \cite[Remark 3.6]{Brannan2019bigalois}}]
\begin{itemize}
\item
Let $(V,E)$ be a simple classical graph, and $A: C(V)\to C(V)$ the adjacency matrix. Then $(C(V), \tau,A)$ is a quantum graph. 
\item
Let $(B,\psi)$ be a quantum set. The (reflexive) \emph{complete graph} on $(B,\psi)$ is given by $A=\delta^2 \psi(\cdot)1$, which is an undirected reflexive $\delta^2$-regular quantum graph. Indeed the definition of the unit and counit shows
\begin{align*}
	\begin{sd}
	\node (G1) {$\phantom{A}$};
	\node[right=0cm of G1] (G2) {$\phantom{A}$};
	\draw (G1.north) arc(90:-270:\unode) to[out=90,in=90] 
			coordinate[pos=1/2] (m) (G2.north) arc(90:-270:\unode);
	\draw (G1.south) arc(-90:270:\unode) to[out=-90,in=-90] 
			coordinate[pos=1/2] (m*) (G2.south) arc(-90:270:\unode);
	\draw (m)--++(0,\hsep/4);
	\draw (m*)--++(0,-\hsep/4);
	\end{sd}
	&=
	\begin{sd}
	\node (G1) {$\phantom{A}$};
	\draw (G1.north) arc(90:-270:\unode)--++(0,\hsep/4);
	\draw (G1.south) arc(-90:270:\unode)--++(0,-\hsep/4);
	\end{sd};
	&
	\begin{sd}
	\node (G1) {$\phantom{A}$};
	\draw (G1.north) arc(90:-270:\unode) to[out=90,in=90] 
		([xshift=-\hsep/2]G1.north)--++(0,-\hsep/2);
	\draw (G1.south) arc(-90:270:\unode) to[out=-90,in=-90] 
		([xshift=\hsep/2]G1.south)--++(0,\hsep/2);
	\end{sd}
	&=
	\begin{sd}
	\node (G1) {$\phantom{A}$};
	\draw (G1.north) arc(90:-270:\unode)--++(0,\hsep/4);
	\draw (G1.south) arc(-90:270:\unode)--++(0,-\hsep/4);
	\end{sd};
	&
	\left(\begin{sd}
	\node (G1) {$\phantom{A}$};
	\draw (G1.north) arc(90:-270:\unode)--++(0,\hsep/4);
	\draw (G1.south) arc(-90:270:\unode)--++(0,-\hsep/4);
	\end{sd}\right)^\dagger
	&=
	\begin{sd}
	\node (G1) {$\phantom{A}$};
	\draw (G1.north) arc(90:-270:\unode)--++(0,\hsep/4);
	\draw (G1.south) arc(-90:270:\unode)--++(0,-\hsep/4);
	\end{sd};
	&
	\begin{sd}
	\node (G1) {$\phantom{A}$};
	\node[right=0cm of G1] (G2) {$\phantom{A}$};
	\draw (G1.north) arc(90:-270:\unode) to[out=90,in=90] 
			coordinate[pos=1/2] (m) (G2.north) --(G2.south);
	\draw (G1.south) arc(-90:270:\unode) to[out=-90,in=-90] 
			coordinate[pos=1/2] (m*) (G2.south) ;
	\draw (m)--++(0,\hsep/4);
	\draw (m*)--++(0,-\hsep/4);
	\end{sd}
	&=
	\begin{sd}
	\draw (0,0)--++(0,\hsep);
	\end{sd}.
\end{align*}
In classical caces $\abs{B}\tau(\cdot)1=\Tr(\cdot)1$ is the matrix with all entries one, which is the reflexive complete graph.
By Proposition \ref{prop:ref-irref}, the irreflexive quantum complete graph is 
$A=\delta^2 \psi(\cdot)1-\id_B$.

\item
Let $(B,\psi)$ be a quantum set. The \emph{trivial graph} on $(B,\psi)$ is given by $A=\id_B$, which is an undirected reflexive $1$-regular quantum graph.
This follows from the specialty (\ref{special}) and snake equation (\ref{snakeeq}).
The trivial graph is the reflexive complement (defined in Proposition \ref{prop:reflexcomple}) of the complete graph.
In classical cases, the trivial graph is the graph with only the self-loops.
Its irreflexive counterpart as in Proposition \ref{prop:ref-irref} is $A=0$.

\end{itemize}
\end{ex}

\begin{prop}\label{prop:ref-irref}
Irreflexive real quantum graphs $A_\mathrm{irref}$ on $(B,\psi)$ have one to one correspondence with reflexive real quantum graphs $A_\mathrm{ref}$ via
\[
A_\mathrm{irref}+\id_B=A_\mathrm{ref}.
\] 
Thus their spectra satisfy 
$\Spec(A_\mathrm{irref})+1=\Spec(A_\mathrm{ref})$.
\end{prop}

\begin{rmk}
The correspondence between reflexive and irreflexive quantum graphs also holds for self-transpose quantum graphs. Its proof is the same except the replacement of real condition by self-transpose condition.
\end{rmk}

\begin{proof}
Let $A=A_\mathrm{irref}$ be an irreflexive real quantum graph on $(B,\psi)$. 
We show that $A_\mathrm{ref}=A+\id_B$ is a reflexive real quantum graph. 
The reality of $A$ and $\id_B$ shows that $A_\mathrm{ref}$ is also real by linearity.
The reflexivity follows by
\[
\begin{sd}
	\node[draw] (A1) {$A_\mathrm{ref}$};
	\draw ([xshift=\hsep]A1.south)--([xshift=\hsep]A1.north);
	\draw (A1) to[out=90,in=90] coordinate[pos=1/2] (m) ([xshift=\hsep]A1.north);
	\draw (A1) to[out=-90,in=-90] coordinate[pos=1/2] (m*) ([xshift=\hsep]A1.south);
	\draw (m)--++(0,\hsep/2);
	\draw (m*)--++(0,-\hsep/2);
\end{sd}
=
\begin{sd}
	\node[circle,draw] (A1) {$A$};
	\draw ([xshift=\hsep]A1.south)--([xshift=\hsep]A1.north);
	\draw (A1) to[out=90,in=90] coordinate[pos=1/2] (m) ([xshift=\hsep]A1.north);
	\draw (A1) to[out=-90,in=-90] coordinate[pos=1/2] (m*) ([xshift=\hsep]A1.south);
	\draw (m)--++(0,\hsep/2);
	\draw (m*)--++(0,-\hsep/2);
\end{sd}
+
\begin{sd}
	\node[circle] (A1) {$\phantom{A}$};
	\draw  (A1.south)--(A1.north)
	 ([xshift=\hsep]A1.south)--([xshift=\hsep]A1.north);
	\draw (A1) to[out=90,in=90] coordinate[pos=1/2] (m) ([xshift=\hsep]A1.north);
	\draw (A1) to[out=-90,in=-90] coordinate[pos=1/2] (m*) 
		([xshift=\hsep]A1.south);
	\draw (m)--++(0,\hsep/2);
	\draw (m*)--++(0,-\hsep/2);
\end{sd}
\overset{\textrm{(irreflexive)}}{=}
0+
\delta^2
\begin{sd}
	\draw (0,0)--++(0,2);
\end{sd}.
\]
 We have by the irreflexivity that
\begin{align*}
\begin{sd}
	\node[draw] (A1) {$A_\mathrm{ref}$};
	\node[draw,right=\hsep/2 of A1)] (A2) {$A_\mathrm{ref}$};
	\draw (A1) to[out=90,in=90] coordinate[pos=1/2] (m) (A2);
	\draw (A1) to[out=-90,in=-90] coordinate[pos=1/2] (m*) (A2);
	\draw (m)--++(0,\hsep/2);
	\draw (m*)--++(0,-\hsep/2);
\end{sd}
&=
\begin{sd}
	\node[circle,draw] (A1) {$A$};
	\node[circle,draw,right=\hsep/2 of A1)] (A2) {$A$};
	\draw (A1) to[out=90,in=90] coordinate[pos=1/2] (m) (A2);
	\draw (A1) to[out=-90,in=-90] coordinate[pos=1/2] (m*) (A2);
	\draw (m)--++(0,\hsep/2);
	\draw (m*)--++(0,-\hsep/2);
\end{sd}
+
\begin{sd}
	\node[circle,draw] (A1) {$A$};
	\draw ([xshift=\hsep]A1.south)--([xshift=\hsep]A1.north);
	\draw (A1) to[out=90,in=90] coordinate[pos=1/2] (m) ([xshift=\hsep]A1.north);
	\draw (A1) to[out=-90,in=-90] coordinate[pos=1/2] (m*) ([xshift=\hsep]A1.south);
	\draw (m)--++(0,\hsep/2);
	\draw (m*)--++(0,-\hsep/2);
\end{sd}
+
\begin{sd}
	\node[circle,draw] (A1) {$A$};
	\draw ([xshift=-\hsep]A1.south)--([xshift=-\hsep]A1.north);
	\draw (A1) to[out=90,in=90] coordinate[pos=1/2] (m) ([xshift=-\hsep]A1.north);
	\draw (A1) to[out=-90,in=-90] coordinate[pos=1/2] (m*) 
		([xshift=-\hsep]A1.south);
	\draw (m)--++(0,\hsep/2);
	\draw (m*)--++(0,-\hsep/2);
\end{sd}
+
\begin{sd}
	\node[circle] (A1) {$\phantom{A}$};
	\draw  (A1.south)--(A1.north)
	 ([xshift=\hsep]A1.south)--([xshift=\hsep]A1.north);
	\draw (A1) to[out=90,in=90] coordinate[pos=1/2] (m) ([xshift=\hsep]A1.north);
	\draw (A1) to[out=-90,in=-90] coordinate[pos=1/2] (m*) 
		([xshift=\hsep]A1.south);
	\draw (m)--++(0,\hsep/2);
	\draw (m*)--++(0,-\hsep/2);
\end{sd}
\\
&=
\delta^2
\begin{sd}
	\node[circle,draw] (A1) {$A$};
	\draw (A1)--++(0,\hsep);
	\draw (A1)--++(0,-\hsep);
\end{sd}
+
\begin{sd}
	\node[circle,draw] (A1) {$A$};
	\draw ([xshift=-\hsep]A1.south)--([xshift=-\hsep]A1.north);
	\draw (A1) to[out=90,in=90] coordinate[pos=1/2] (m) ([xshift=-\hsep]A1.north);
	\draw (A1) to[out=-90,in=-90] coordinate[pos=1/2] (m*) 
		([xshift=-\hsep]A1.south);
	\draw (m)--++(0,\hsep/2);
	\draw (m*)--++(0,-\hsep/2);
\end{sd}
+
\delta^2
\begin{sd}
	\draw (0,0)--++(0,2);
\end{sd}
=
\delta^2
\begin{sd}
	\node[draw] (A1) {$A_\mathrm{ref}$};
	\draw (A1)--++(0,\hsep);
	\draw (A1)--++(0,-\hsep);
\end{sd}
+
\begin{sd}
	\node[circle,draw] (A1) {$A$};
	\draw ([xshift=-\hsep]A1.south)--([xshift=-\hsep]A1.north);
	\draw (A1) to[out=90,in=90] coordinate[pos=1/2] (m) ([xshift=-\hsep]A1.north);
	\draw (A1) to[out=-90,in=-90] coordinate[pos=1/2] (m*) 
		([xshift=-\hsep]A1.south);
	\draw (m)--++(0,\hsep/2);
	\draw (m*)--++(0,-\hsep/2);
\end{sd}.
\end{align*}
Hence it suffices to show that the final term $m(\id_B \otimes A)m^\dagger$ is zero.
Indeed reality and irreflexivity implies that
\begin{align*}
\begin{sd}
	\node[circle,draw] (A1) {$A$};
	\draw ([xshift=-\hsep]A1.south)--([xshift=-\hsep]A1.north);
	\draw (A1) to[out=90,in=90] coordinate[pos=1/2] (m) ([xshift=-\hsep]A1.north);
	\draw (A1) to[out=-90,in=-90] coordinate[pos=1/2] (m*) 
		([xshift=-\hsep]A1.south);
	\draw (m)--++(0,\hsep/2);
	\draw (m*)--++(0,-\hsep/2);
\end{sd}
\overset{(\textrm{real})}{=}
\begin{sd}
	\node[circle,draw] (A1) {$A^\dagger$};
	\draw ([xshift=-\hsep]A1.south)--([xshift=-\hsep]A1.north);
	\draw ([xshift=-\hsep/2]A1.south)--
		([xshift=-\hsep/2]A1.north) to[out=90,in=90] 
		(A1)  to[out=-90,in=-90]
		([xshift=\hsep/2]A1.south)--([xshift=\hsep/2]A1.north);
	\draw	([xshift=\hsep/2]A1.north) to[out=90,in=90]  coordinate[pos=1/2] (m) 
		 ([xshift=-\hsep]A1.north);
	\draw ([xshift=-\hsep]A1.south) to[out=-90,in=-90] coordinate[pos=1/2] (m*) 
		([xshift=-\hsep/2]A1.south);
	\draw (m)--++(0,\hsep/2);
	\draw (m*)--++(0,-\hsep/2);
\end{sd}
=
\begin{sd}
	\node[circle,draw] (A1) {$A^\dagger$};
	\draw ([xshift=\hsep/2]A1.south)--([xshift=\hsep/2]A1.north);
	\draw	([xshift=\hsep/2]A1.north) to[out=90,in=90]  coordinate[pos=1/2] (m) 
		 (A1.north);
	\draw (A1.south) to[out=-90,in=-90] coordinate[pos=1/2] (m*) 
		([xshift=\hsep/2]A1.south);
	\draw (m)to[out=90,in=90]++(-\hsep*3/4,0)--++(0,-\hsep);
	\draw (m*)to[out=-90,in=-90]++(\hsep*3/4,0)--++(0,\hsep);
\end{sd}
\overset{(\textrm{irreflexive})}{=} 0,
\end{align*}
where the third equality follows from topological calculation using the coassociativity, associativity, snake equation (\ref{snakeeq}) and Frobenius equation (\ref{Frobeniuseq}).
Therefore $A_\mathrm{ref}=A_\mathrm{irref}+\id_B$ is a reflexive quantum graph. 
Similarly given reflexive quantum graph $A_\mathrm{ref}$, it follows that $A_\mathrm{irref}=A_\mathrm{ref}-\id_B$ is an irreflexive quantum graph. 
The equality of their spectra follows from
\[
\lambda \id_B - A_\mathrm{irref} = (\lambda+1) \id_B - A_\mathrm{ref} 
\quad \forall \lambda \in \C.
\]
\end{proof}

\begin{prop}\label{prop:reflexcomple}
Let $(B,\psi,A)$ be a real reflexive quantum graph. Then
\[
A^c\coloneqq \id_B + \delta^2 \psi(\cdot)1 -A
\]
is also a real reflexive quantum graph on $(B,\psi)$, the so-called reflexive complement of $A$.
\end{prop}

\begin{proof}
Since $\id_B$, $\delta^2 \psi(\cdot)1$, and $A$ are real reflexive quantum graphs  on $(B,\psi)$, linearity shows that $A^c$ is also real and reflexive.
We have by distributing the unit and the counit that
\begin{align*}
\begin{sd}
	\node[draw,circle] (A1) {$A^c$};
	\node[draw,circle,right=\hsep/4 of A1)] (A2) {$A^c$};
	\draw (A1) to[out=90,in=90] coordinate[pos=1/2] (m) (A2);
	\draw (A1) to[out=-90,in=-90] coordinate[pos=1/2] (m*) (A2);
	\draw (m)--++(0,\hsep/3);
	\draw (m*)--++(0,-\hsep/3);
\end{sd}
&=
\begin{sd}
	\node[draw] (A1) {$\id-A$};
	\node[draw,right=\hsep/8 of A1)] (A2) {$\id-A$};
	\draw (A1) to[out=90,in=90] coordinate[pos=1/2] (m) (A2);
	\draw (A1) to[out=-90,in=-90] coordinate[pos=1/2] (m*) (A2);
	\draw (m)--++(0,\hsep/3);
	\draw (m*)--++(0,-\hsep/3);
\end{sd}
+
2\delta^2
\begin{sd}
	\node[draw] (A1) {$\id-A$};
	\draw (A1)--++(0,\hsep*2/3);
	\draw (A1)--++(0,-\hsep*2/3);
\end{sd}
+
\delta^4
\begin{sd}
	\node (A1) {$\phantom{A}$};
	\draw (A1.north) arc(90:-270:\unode)--++(0,\hsep/3);
	\draw (A1.south) arc(-90:270:\unode)--++(0,-\hsep/3);
\end{sd}.
\end{align*}
Since $A-\id_B$ is an irreflexive real quantum graph by Proposition \ref{prop:ref-irref}, we obtain 
\begin{align*}
&=
\delta^2
\begin{sd}
	\node[draw] (A1) {$\id-A$};
	\draw (A1)--++(0,\hsep/2);
	\draw (A1)--++(0,-\hsep/2);
\end{sd}
+
\delta^4
\begin{sd}
	\node (A1) {$\phantom{A}$};
	\draw (A1.north) arc(90:-270:\unode)--++(0,\hsep/4);
	\draw (A1.south) arc(-90:270:\unode)--++(0,-\hsep/4);
\end{sd}
=
\delta^2
\begin{sd}
	\node[draw,circle] (A1) {$A^c$};
	\draw (A1)--++(0,\hsep/2);
	\draw (A1)--++(0,-\hsep/2);
\end{sd}.
\end{align*}
\end{proof}

\subsection{Quantum isomorphisms}

\begin{dfn}[{Musto, Reutter, Verdon \cite[Definition 3.11, 4.3]{Musto2018compositional}}]\label{dfn:Qfcn}
A \emph{quantum function} $(H,P) : (B',\psi') \to (B,\psi)$ between quantum sets $(B,\psi)$ and $(B',\psi')$ is a pair 
$(H,P)$ of a finite dimensional Hilbert space $H$ and a linear operator 
$P: H \otimes B \to B' \otimes H$ denoted in string diagrams by
\[
P=\begin{sd}
\node (B) {$B$};
\node[right=\hsep of B] (H) {$H$};
\node[above=\hsep of B] (Ha) {$H$};
\node[right=\hsep of Ha] (B') {$B'$};
\draw (B) to[out=90,in=-90]node[circle,draw,fill=white,midway](P) {$P$} (B') ;
\draw[->-] (H) to[out=90,in=-20] (P);
\draw[->-] (P) to[out=160,in=-90] (Ha);
\end{sd}
\]
satisfying 
\begin{align}
\begin{sd}
\coordinate (B) ;
\coordinate[right=\hsep of B] (H) ;
\coordinate[above=\hsep of B] (Ha) ;
\coordinate[right=\hsep of Ha] (B') ;
\node[circle,draw,fill=white](P) at ($(B)!0.5!(B')$) {$P$};
\draw (B') to (P) ;
\draw (P) --++(-1/4,-1/4) arc(45:-315:\unode) ;
\draw[->-] (H) to (P);
\draw[->-] (P) to (Ha);
\end{sd}
&=
\begin{sd}
\coordinate (B) ;
\coordinate[right=\hsep of B] (H) ;
\coordinate[above=\hsep of B] (Ha) ;
\coordinate[right=\hsep of Ha] (B') ;
\draw (B')--++(0,-1/4) arc(90:-270:\unode);
\draw[->-] (H)--(Ha); 
\end{sd}
&
\begin{sd}
\coordinate (B2) ;
\coordinate[right=\hsep of B2] (B1) ;
\coordinate[right=\hsep of B1] (H) ;
\coordinate[above=\hsep of B2] (Ha) ;
\coordinate[right=\hsep*3/2 of Ha] (B') ;
\node[circle,draw,fill=white](P1) at ($(H)!1/3!(Ha)$) {$P$};
\node[circle,draw,fill=white](P2) at ($(H)!2/3!(Ha)$) {$P$};
\draw (B2) to (P2) (B1) to (P1);
\draw (P1) to[out=45,in=45]coordinate[midway](m) (P2);
\draw (m)--(B');
\draw[->-] (H) to (P1); \draw[->-] (P1) to (P2); \draw[->-] (P2) to (Ha) ;
\end{sd}
&=
\begin{sd}
\coordinate (B2) ;
\coordinate[right=\hsep of B2] (B1) ;
\coordinate[right=\hsep of B1] (H) ;
\coordinate[above=\hsep of B2] (Ha) ;
\coordinate[right=\hsep*3/2 of Ha] (B') ;
\node[circle,draw,fill=white](P1) at ($(H)!1/2!(Ha)$) {$P$};
\draw (B1) to[out=90,in=45]coordinate[midway](m) (B2);
\draw (m)--(P1) to (B');
\draw[->-] (H) to (P1); \draw[->-] (P1) to (Ha) ;
\end{sd}
&
\begin{sd}
\coordinate (B) ;
\coordinate[right=\hsep of B] (H) ;
\coordinate[above=\hsep of B] (Ha) ;
\coordinate[right=\hsep of Ha] (B') ;
\node[circle,draw,fill=white](P) at ($(B)!0.5!(B')$) {$P^\dagger$};
\draw (B) to[out=90,in=135] (P) to[out=-45,in=-90] (B');
\draw[->-] ($(B)!0.2!(H)$) to (P);
\draw[->-] (P) to ($(B')!0.2!(Ha)$);
\end{sd}
&=
\begin{sd}
\coordinate (B) ;
\coordinate[right=\hsep of B] (H) ;
\coordinate[above=\hsep of B] (Ha) ;
\coordinate[right=\hsep of Ha] (B') ;
\node[circle,draw,fill=white](P) at ($(B)!0.5!(B')$) {$P$};
\draw (B) to (P) to (B');
\draw[->-] (H) to (P);
\draw[->-] (P) to (Ha);
\end{sd},					\label{qfcn}
\end{align}
which respectively means that $P$ preserves the unit, multiplication, and involution.
A quantum function $(H,P)$ is called a \emph{quantum bijection} if it also satisfies
\begin{align}
\begin{sd}
\coordinate (B) ;
\coordinate[right=\hsep of B] (H) ;
\coordinate[above=\hsep of B] (Ha) ;
\coordinate[right=\hsep of Ha] (B') ;
\node[circle,draw,fill=white](P) at ($(B)!0.5!(B')$) {$P$};
\draw (B) to (P) ;
\draw (P) --++(1/4,1/4) arc(-135:225:\unode) ;
\draw[->-] (H) to (P);
\draw[->-] (P) to (Ha);
\end{sd}
&=
\begin{sd}
\coordinate (B) ;
\coordinate[right=\hsep of B] (H) ;
\coordinate[above=\hsep of B] (Ha) ;
\coordinate[right=\hsep of Ha] (B') ;
\draw (B)--++(0,1/4) arc(-90:270:\unode);
\draw[->-] (H)--(Ha); 
\end{sd}
&
\begin{sd}
\coordinate (B0) ;
\coordinate[right=\hsep*1/2 of B0] (B1) ;
\coordinate[right=\hsep*3/2 of B1] (H) ;
\coordinate[above=\hsep of B0] (Ha) ;
\coordinate[right=\hsep of Ha] (B'2) ;
\coordinate[right=\hsep of B'2] (B'1) ;
\node[circle,draw,fill=white](P1) at ($(H)!1/3!(Ha)$) {$P$};
\node[circle,draw,fill=white](P2) at ($(H)!2/3!(Ha)$) {$P$};
\draw (B'2) to (P2) (B'1) to (P1);
\draw (P1) to[out=-135,in=-135]coordinate[midway](m) (P2);
\draw (m)--(B1);
\draw[->-] (H) to (P1); \draw[->-] (P1) to (P2); \draw[->-] (P2) to (Ha) ;
\end{sd}
&=
\begin{sd}
\coordinate (B0) ;
\coordinate[right=\hsep*1/2 of B0] (B1) ;
\coordinate[right=\hsep*3/2 of B1] (H) ;
\coordinate[above=\hsep of B0] (Ha) ;
\coordinate[right=\hsep of Ha] (B'2) ;
\coordinate[right=\hsep of B'2] (B'1) ;
\node[circle,draw,fill=white](P1) at ($(H)!1/2!(Ha)$) {$P$};
\draw (B'2) to[out=-90,in=-135]coordinate[midway](m) (B'1);
\draw (m)--(P1) to (B1);
\draw[->-] (H) to (P1); \draw[->-] (P1) to (Ha) ;
\end{sd},					\label{qbij}
\end{align}
which respectively means that $P$ preserves the counit and comultiplication.
If $\abs{H}=\dim H=1$, then a quantum function (resp. quantum bijection) $(H,P)$ is called a classical function (resp. classical bijection).
\end{dfn}

\begin{rmk}
In the case of $\abs{H}=1$, we may forget the oriented strings of $H$.
Then (\ref{qfcn}) exactly says that
\begin{align*}
P(1)&=1,
& P(x) P(y) &= P(xy),
& P(x^*)^*&=P(x)  \quad \forall x,y \in B,
\end{align*} 
i.e., $P:B \to B'$ is a $*$-homomorphism. Similarly (\ref{qbij}) says that $P:B \to B'$ is a cohomomorphism. This is why $(H,P)$ is called classical if $\abs{H}=1$.
\end{rmk}

Note that the quantum function $(H,P):(B',\psi')\to (B,\psi)$ and `homomorphism' $P:B\otimes H\to H\otimes B'$ have opposite direction. This is based on the Gelfand duality, where a set function $f:X \to Y$ corresponds to a unital $*$-homomorphism $\cdot \circ f : C(Y) \to C(X)$.

\begin{rmk}\label{rmk:qisom/*hom}
Alternatively we may consider
\[
\tilde{P}=
\begin{sd}
\node[draw] (P) {$\tilde{P}$};
\draw (P.north)--++(0,0.5) (P.south)--++(0,-0.5);
\draw[->-] ([xshift=-0.2cm]P.north)--++(0,0.5) ;
\draw[-<-] ([xshift=0.2cm]P.north)--++(0,0.5);
\end{sd}
\coloneqq
\begin{sd}
\coordinate (B) ;
\coordinate[right=\hsep of B] (H) ;
\coordinate[above=\hsep*3/2 of B] (Ha) ;
\coordinate[right=\hsep of Ha] (B') ;
\node[circle,draw,fill=white](P) at ($(B)!0.5!(B')$) {$P$};
\draw (B) to (P) to (B');
\draw[->-] ([xshift=\hsep/3]B') to[out=-90,in=-45] (P);
\draw[->-] (P) to (Ha);
\end{sd}
: B \to H \otimes B' \otimes H^* \cong B' \otimes B(H)
\]
(cf. \cite[proof of Theorem 3.28]{Musto2018compositional}). Then $(H,P)$ is a quantum function if and only if 
$\tilde{P}: B \to  B' \otimes B(H)$ is a $*$-homomorphism.
Note that $H \otimes B' \otimes H^* \cong B' \otimes B(H)$ is equipped with the following operators by Proposition \ref{prop:HH*=B(H)}: 
\[
1'\otimes \id_H =\begin{sd}
	\draw (0,0)--++(0,1/4); \draw[fill=white] (0,0) circle (\unode) ;
	\draw[-<-]  (-\hsep/4,1/4)--(-\hsep/4,0) arc(-180:0:\hsep/4) 
		--(\hsep/4,1/4);
\end{sd}, \quad
m=\begin{sd}
	\coordinate (H1) at (-\hsep/2,0);  
	\coordinate (H2) at (\hsep/2,0)  ;
	\coordinate (H3) at (0,\hsep);
	\draw[->-] ([xshift=-0.2cm]H1) to[out=90,in=-90] ([xshift=-0.2cm]H3) ;
	\draw[->-] ([xshift=-0.2cm]H2) to[out=90,in=90] ([xshift=0.2cm]H1) ;
	\draw (H2) arc(0:180:\hsep/2) coordinate[midway] (m)  ;
	\draw (m)--(H3);
	\draw[-<-] ([xshift=0.2cm]H2) to[out=90,in=-90] ([xshift=0.2cm]H3) ;
\end{sd}, \quad
\psi' \otimes \frac{\Tr}{\abs{H}}=\frac{1}{\abs{H}}\begin{sd}	
	\draw (0,0)--++(0,-1/4); \draw[fill=white] (0,0) circle (\unode) ;
	\draw[->-]  (-\hsep/4,-1/4)--(-\hsep/4,0) arc(180:0:\hsep/4) 
		--(\hsep/4,-1/4);
\end{sd}, \quad
m^\dagger=\abs{H}\begin{sd}  
	\coordinate (H1) at (-\hsep/2,0);  
	\coordinate (H2) at (\hsep/2,0)  ;
	\coordinate (H3) at (0,-\hsep);
	\draw[-<-] ([xshift=-0.2cm]H1) to[out=-90,in=90] ([xshift=-0.2cm]H3) ;
	\draw[-<-] ([xshift=-0.2cm]H2) to[out=-90,in=-90] ([xshift=0.2cm]H1) ;
	\draw (H2) arc(0:-180:\hsep/2) coordinate[midway] (m)  (H1) ;
	\draw (m)--(H3);
	\draw[->-] ([xshift=0.2cm]H2) to[out=-90,in=90] ([xshift=0.2cm]H3) ;
\end{sd}.
\]
Thus indeed (\ref{qfcn}) formulates that $\tilde{P}$ is a $*$-homomorphism.
Although string diagrams like 
$\begin{sd}
	\draw[->-]  (0,0) arc(180:0:\hsep/4) ;
\end{sd}$
 do not work well for infinite dimensional $H$, 
the formulation in terms of $\tilde{P}$ is valid.
The formulation of quantum isomorphisms by Brannan et al. 
\cite[section 4]{Brannan2019bigalois} is derived from this viewpoint.
\end{rmk}

\begin{rmk}
By the snake equations (\ref{snakeeq}), 
the $*$-preserving condition in (\ref{qfcn}) has an equivalent formulation:
\begin{align}
\begin{sd}
\coordinate (B) ;
\coordinate[right=\hsep of B] (H) ;
\coordinate[above=\hsep of B] (Ha) ;
\coordinate[right=\hsep of Ha] (B') ;
\node[circle,draw,fill=white](P) at ($(B)!0.5!(B')$) {$P^\dagger$};
\draw (B) to[out=90,in=135] (P) to[out=-45,in=-90] (B');
\draw[->-] ($(B)!0.2!(H)$) to (P);
\draw[->-] (P) to ($(B')!0.2!(Ha)$);
\end{sd}
=
\begin{sd}
\coordinate (B) ;
\coordinate[right=\hsep of B] (H) ;
\coordinate[above=\hsep of B] (Ha) ;
\coordinate[right=\hsep of Ha] (B') ;
\node[circle,draw,fill=white](P) at ($(B)!0.5!(B')$) {$P$};
\draw (B) to (P) to (B');
\draw[->-] (H) to (P);
\draw[->-] (P) to (Ha);
\end{sd}
\iff
\begin{sd}
\coordinate (B) ;
\coordinate[left=\hsep of B] (H) ;
\coordinate[above=\hsep of B] (Ha) ;
\coordinate[left=\hsep of Ha] (B') ;
\node[circle,draw,fill=white](P) at ($(B)!0.5!(B')$) {$P^\dagger$};
\draw (B) to (P) to (B');
\draw[->-] (H) to (P);
\draw[->-] (P) to (Ha);
\end{sd}
=
\begin{sd}
\coordinate (B) ;
\coordinate[left=\hsep of B] (H) ;
\coordinate[above=\hsep of B] (Ha) ;
\coordinate[left=\hsep of Ha] (B') ;
\node[circle,draw,fill=white](P) at ($(B)!0.5!(B')$) {$P$};
\draw (B) to[out=90,in=45] (P) to[out=-135,in=-90] (B');
\draw[->-] ($(B)!0.2!(H)$) to (P);
\draw[->-] (P) to ($(B')!0.2!(Ha)$);
\end{sd}.				\label{*presqfcn}
\end{align}
\end{rmk}

\begin{dfn}[{Musto, Reutter, Verdon \cite[Definition 3.18]{Musto2018compositional}}]
Let $(H,P),(H',P') : (B',\psi') \to (B,\psi)$ be quantum functions. 
An \emph{intertwiner} $f:(H,P)\to(H',P')$ is an operator $f: H\to H'$ satisfying
\[
\begin{sd}
\coordinate (B) ;
\coordinate[right=\hsep*3/2 of B] (H) ;
\coordinate[above=\hsep*3/2 of B] (Ha) ;
\coordinate[right=\hsep*3/2 of Ha] (B') ;
\node[circle,draw,fill=white](P) at ($(B)!0.5!(B')$) {$P$};
\draw (B) to  (P); 
\draw (P) to (B');
\draw[->-] (H) to  (P);
\draw[->-] (P) to coordinate[midway] (G) (Ha);
\node[circle,draw,fill=white] at (G) {$f$};
\end{sd}
=
\begin{sd}
\coordinate (B) ;
\coordinate[right=\hsep*3/2 of B] (H) ;
\coordinate[above=\hsep*3/2 of B] (Ha) ;
\coordinate[right=\hsep*3/2 of Ha] (B') ;
\node[circle,draw,fill=white](P) at ($(B)!0.5!(B')$) {$P'$};
\draw (B) to (P); 
\draw (P) to (B');
\draw[->-] (H) to coordinate[midway] (G)  (P);
\draw[->-] (P) to  (Ha);
\node[circle,draw,fill=white] at (G) {$f$};
\end{sd}.
\]
The category $\QSet$ of quantum sets is defined as a $2$-category that
consists of
\begin{itemize}
\item \textbf{Objects:} quantum sets $(B,\psi)$;
\item \textbf{1-morphisms:} quantum functions $(H,P) : (B',\psi') \to (B,\psi)$;
\item \textbf{2-morphisms:} intertwiners $f:(H,P)\to(H',P')$.
\end{itemize}
Given quantum sets $(B,\psi), (B',\psi')$, we define the category $\QBij((B',\psi'), (B,\psi))$ as a category consisting of
\begin{itemize}
\item \textbf{Objects:} quantum bijections $(H,P) : (B',\psi') \to (B,\psi)$;
\item \textbf{Morphisms:} intertwiners $f:(H,P)\to(H',P')$.
\end{itemize}

\end{dfn}

\begin{lem}[{Tracial case by Musto, Reutter, Verdon \cite[Theorem 4.8]{Musto2018compositional}}]\label{lem:qbijunitary}
For a quantum function $(H,P) : (B',\psi') \to (B,\psi)$, TFAE: 
\begin{description}
\item[$(1)$]
$(H,P)$ is a quantum bijection;
\item[$(2)$]
$P$ is a unitary operator.
\end{description}
\end{lem}

\begin{proof}
$(1)\implies(2)$
By the involution and multiplication preserving conditions in (\ref{qfcn})
and the counit preserving condition in (\ref{qbij}), we have
\begin{align*}
P^\dagger P 
=
\begin{sd}
\coordinate (B1) ;
\coordinate[right=\hsep of B1] (H1) ;
\coordinate[above=\hsep of B1] (H2) ;
\coordinate[right=\hsep of H2] (B2) ;
\coordinate[above=\hsep of H2] (B3);
\coordinate[right=\hsep of B3] (H3) ;
\node[circle,draw,fill=white](P1) at ($(B1)!0.5!(B2)$) {$P$};
\node[circle,draw,fill=white](P2) at ($(B3)!0.5!(B2)$) {$P^\dagger$};
\draw (B1) to (P1) ;
\draw (P1) to[out=45,in=-45] (P2) ;
\draw (P2) to (B3) ;
\draw[->-] (H1) to (P1);
\draw[->-] (P1) to[out=135,in=-135] (P2) ;
\draw[->-] (P2) to (H3) ;
\end{sd}
\overset{\textrm{(\ref{*presqfcn})}}{=}
\begin{sd}
\coordinate (B1) ;
\coordinate[right=\hsep of B1] (H1) ;
\coordinate[above=\hsep of B1] (H2) ;
\coordinate[right=\hsep of H2] (B2) ;
\coordinate[above=\hsep of H2] (B3);
\coordinate[right=\hsep of B3] (H3) ;
\node[circle,draw,fill=white](P1) at ($(H1)!1/3!(B3)$) {$P$};
\node[circle,draw,fill=white](P2) at ($(H1)!2/3!(B3)$) {$P$};
\draw (B1) to (P1) ;
\draw (P1) to[out=45,in=45] (P2) ;
\draw (P2) to[out=-135,in=-90] ([xshift=-0.3cm]B3) ;
\draw[->-] (H1) to (P1);
\draw[->-] (P1) to (P2) ;
\draw[->-] (P2) to (B3) ;
\end{sd}
\overset{\textrm{(\ref{qfcn})}}{=}
\begin{sd}
\coordinate (B1) ;
\coordinate[right=\hsep of B1] (H1) ;
\coordinate[above=\hsep*2 of B1] (H2);
\coordinate[right=\hsep of H2] (B2) ;
\node[circle,draw,fill=white](P1) at ($(H1)!1/2!(H2)$) {$P$};
\draw (P1)--++(-0.3,-0.3) coordinate(m) ;
\draw (P1)--++(1/4,1/4) arc(-135:225:\unode) ;
\draw (B1) to[out=45,in=0] (m) to[out=180,in=0]++(-1/6,-1/10)
	to[out=180,in=-90] ([xshift=-0.3cm,yshift=-\hsep]H2)-- ([xshift=-0.3cm]H2);
\draw[->-] (H1) to (P1);
\draw[->-] (P1) to (H2) ;
\end{sd}
\overset{\textrm{(\ref{qbij})}}{=}
\begin{sd}
\coordinate (B1) ;
\coordinate[right=\hsep of B1] (H1) ;
\coordinate[above=\hsep*2 of B1] (H2);
\coordinate[right=\hsep of H2] (B2) ;
\coordinate(P1) at ($(H1)!1/2!(H2)$);
\draw (P1)++(-0.3,-0.3) coordinate(m) ;
\draw (B1) to[out=45,in=0] (m) to[out=180,in=0]++(-1/6,-1/10)
	to[out=180,in=-90] ([xshift=-0.3cm,yshift=-\hsep]H2)-- ([xshift=-0.3cm]H2);
\draw[->-] (H1) to (H2) ;
\end{sd}
=
\begin{sd}
\draw (0,0)--++(0,\hsep);
\draw[->-] (\hsep/2,0)--++(0,\hsep);
\end{sd}
=\id_{B\otimes H}.
\end{align*}
Similarly by the involution and unit preserving conditions in (\ref{qfcn})
and the comultiplication preserving condition in (\ref{qbij}), we have
\begin{align*}
P P^\dagger 
=
\begin{sd}
\coordinate (B1) ;
\coordinate[right=\hsep of B1] (H1) ;
\coordinate[above=\hsep of B1] (H2) ;
\coordinate[right=\hsep of H2] (B2) ;
\coordinate[above=\hsep of H2] (B3);
\coordinate[right=\hsep of B3] (H3) ;
\node[circle,draw,fill=white](P1) at ($(B1)!0.5!(B2)$) {$P^\dagger$};
\node[circle,draw,fill=white](P2) at ($(B3)!0.5!(B2)$) {$P$};
\draw[->-] (B1) to (P1) ;
\draw[->-] (P1) to[out=45,in=-45] (P2) ;
\draw[->-] (P2) to (B3) ;
\draw (H1) to (P1);
\draw (P1) to[out=135,in=-135] (P2) ;
\draw (P2) to (H3) ;
\end{sd}
\overset{\textrm{(\ref{*presqfcn})}}{=}
\begin{sd}
\coordinate (B1) ;
\coordinate[right=\hsep of B1] (H1) ;
\coordinate[above=\hsep of B1] (H2) ;
\coordinate[right=\hsep of H2] (B2) ;
\coordinate[above=\hsep of H2] (B3);
\coordinate[right=\hsep of B3] (H3) ;
\node[circle,draw,fill=white](P1) at ($(H1)!1/3!(B3)$) {$P$};
\node[circle,draw,fill=white](P2) at ($(H1)!2/3!(B3)$) {$P$};
\draw (H3) to (P2) ;
\draw (P2) to[out=-135,in=-135] (P1) ;
\draw (P1) to[out=45,in=90] ([xshift=0.3cm]H1) ;
\draw[->-] (H1) to (P1);
\draw[->-] (P1) to (P2) ;
\draw[->-] (P2) to (B3) ;
\end{sd}
\overset{\textrm{(\ref{qbij})}}{=}
\begin{sd}
\coordinate (B1) ;
\coordinate[right=\hsep of B1] (H1) ;
\coordinate[above=\hsep*2 of B1] (H2);
\coordinate[right=\hsep of H2] (B2) ;
\node[circle,draw,fill=white](P1) at ($(H1)!1/2!(H2)$) {$P$};
\draw (P1)--++(0.3,0.3) coordinate(m) ;
\draw (P1)--++(-1/4,-1/4) arc(45:-315:\unode) ;
\draw (B2) to[out=-135,in=180] (m) to[out=0,in=180]++(1/6,1/10)
	to[out=0,in=90] ([xshift=0.3cm,yshift=\hsep]H1)-- ([xshift=0.3cm]H1);
\draw[->-] (H1) to (P1);
\draw[->-] (P1) to (H2) ;
\end{sd}
\overset{\textrm{(\ref{qfcn})}}{=}
\begin{sd}
\coordinate (B1) ;
\coordinate[right=\hsep of B1] (H1) ;
\coordinate[above=\hsep*2 of B1] (H2);
\coordinate[right=\hsep of H2] (B2) ;
\coordinate(P1) at ($(H1)!1/2!(H2)$) ;
\draw (P1)++(0.3,0.3) coordinate(m) ;
\draw (B2) to[out=-135,in=180] (m) to[out=0,in=180]++(1/6,1/10)
	to[out=0,in=90] ([xshift=0.3cm,yshift=\hsep]H1)-- ([xshift=0.3cm]H1);
\draw[->-] (H1) to (H2) ;
\end{sd}
=
\begin{sd}
\draw (0,0)--++(0,\hsep);
\draw[->-] (-\hsep/2,0)--++(0,\hsep);
\end{sd}
=\id_{H\otimes B'}.
\end{align*}
Therefore $P$ is unitary.

\noindent $(2)\implies(1)$
Since $P$ is a unitary quantum function,  we have
\begin{align*}
\begin{sd}
\coordinate (B1) ;
\coordinate[left=\hsep of B1] (H1) ;
\coordinate[above=\hsep of B1] (H2) ;
\coordinate[left=\hsep of H2] (B2) ;
\coordinate[above=\hsep of H2] (B3);
\coordinate[left=\hsep of B3] (H3) ;
\coordinate (P1) at ($(B1)!0.5!(B2)$) ;
\node[circle,draw,fill=white](P2) at ($(B3)!0.5!(B2)$) {$P$};
\draw  (P2) --++ (-1/3,-1/3) arc(45:-3155:\unode);
\draw (P2) to (B3) ;
\draw[->-] (H1) to (P1) to[out=45,in=-45] (P2) ;
\draw[->-] (P2) to (H3) ;
\end{sd}
\overset{\textrm{(\ref{qfcn})}}{=}
\begin{sd}
\draw (0,\hsep/2)arc(90:-270:\unode)--(0,\hsep);
\draw[->-] (-\hsep/2,0)--++(0,\hsep);
\end{sd}
\overset{(\textrm{unitary})}{=}
\begin{sd}
\coordinate (B1) ;
\coordinate[left=\hsep of B1] (H1) ;
\coordinate[above=\hsep of B1] (H2) ;
\coordinate[left=\hsep of H2] (B2) ;
\coordinate[above=\hsep of H2] (B3);
\coordinate[left=\hsep of B3] (H3) ;
\node[circle,draw,fill=white](P1) at ($(B1)!0.5!(B2)$) {$P^\dagger$};
\node[circle,draw,fill=white](P2) at ($(B3)!0.5!(B2)$) {$P$};
\draw (P1) --++ (1/3,-1/3) arc(135:-225:\unode);
\draw (P1) to[out=135,in=-135] (P2) ;
\draw (P2) to (B3) ;
\draw[->-] (H1) to (P1);
\draw[->-] (P1) to[out=45,in=-45] (P2) ;
\draw[->-] (P2) to (H3) ;
\end{sd}.
\end{align*}
By postcomposing $P^\dagger$ and taking the adjoint, we obtain  (\ref{qbij}):
\[
\begin{sd}
\coordinate (B) ;
\coordinate[right=\hsep of B] (H) ;
\coordinate[above=\hsep of B] (Ha) ;
\coordinate[right=\hsep of Ha] (B') ;
\draw (B)--++(0,1/4) arc(-90:270:\unode);
\draw[->-] (H)--(Ha); 
\end{sd}
=
\begin{sd}
\coordinate (B) ;
\coordinate[right=\hsep of B] (H) ;
\coordinate[above=\hsep of B] (Ha) ;
\coordinate[right=\hsep of Ha] (B') ;
\node[circle,draw,fill=white](P) at ($(B)!0.5!(B')$) {$P$};
\draw (B) to (P) ;
\draw (P) --++(1/4,1/4) arc(-135:225:\unode) ;
\draw[->-] (H) to (P);
\draw[->-] (P) to (Ha);
\end{sd}.
\] 
Next, we show the comultiplication preserving condition in (\ref{qbij}).
Considering the composition of $P$ and the adjoint of (\ref{qbij}), we have 
\[
\begin{sd}
\coordinate (B0) ;
\coordinate[right=\hsep*1/2 of B0] (B1) ;
\coordinate[right=\hsep*3/2 of B1] (H) ;
\coordinate[below=\hsep of B0] (Ha) ;
\coordinate[right=\hsep of Ha] (B'2) ;
\coordinate[right=\hsep of B'2] (B'1) ;
\node[circle,draw,fill=white](P1) at ($(H)!1/3!(Ha)$) {$P^\dagger$};
\node[circle,draw,fill=white](P2) at ($(H)!2/3!(Ha)$) {$P^\dagger$};
\coordinate[above=\hsep of B0] (Hb) ;
\coordinate[right=\hsep*3/2 of Hb] (B') ;
\node[circle,draw,fill=white](P3) at ($(H)!1/2!(Hb)$) {$P$};
\draw[->-] (P1) to[out=30,in=-30] (P3); \draw[->-] (P3) to (Hb) ;
\draw (B'2) to (P2) (B'1) to (P1) (P3) to (B');
\draw (P1) to[out=135,in=135]coordinate[midway](m) (P2);
\draw (m)to[out=135,in=-135](P3);
\draw[-<-] (P1) to (P2); \draw[-<-] (P2) to (Ha); 
\end{sd}
\overset{\textrm{(\ref{qfcn})}}{=}
\begin{sd}
\coordinate (B0) ;
\coordinate[right=\hsep*1/2 of B0] (B1) ;
\coordinate[right=\hsep*3/2 of B1] (H) ;
\coordinate[below=\hsep of B0] (Ha) ;
\coordinate[right=\hsep of Ha] (B'2) ;
\coordinate[right=\hsep of B'2] (B'1) ;
\node[circle,draw,fill=white](P1) at ($(H)!1/3!(Ha)$) {$P^\dagger$};
\node[circle,draw,fill=white](P2) at ($(H)!2/3!(Ha)$) {$P^\dagger$};
\node[circle,draw,fill=white](P3) at ($(H)!1/3!(Hb)$) {$P$};
\node[circle,draw,fill=white](P4) at ($(H)!2/3!(Hb)$) {$P$};
\draw (P2) to[out=135,in=-135] (P4) (P1) to[out=135,in=-135] (P3);
\draw (P3) to[out=45,in=45]coordinate[midway](m1) (P4);
\draw (m1)--(B');
\draw[->-] (P3) to (P4); \draw[->-] (P4) to (Hb) ;
\draw[->-] (P1) to[out=30,in=-30] (P3);
\draw (B'2) to (P2) (B'1) to (P1);
\draw[-<-] (P1) to (P2); \draw[-<-] (P2) to (Ha); 
\end{sd}
\overset{\textrm{(unitary)}}{=}
\begin{sd}
\coordinate (B1) ;
\coordinate[right=\hsep of B1] (H1) ;
\coordinate[above=\hsep of B1] (H2) ;
\coordinate[right=\hsep of H2] (B2) ;
\coordinate[above=\hsep of H2] (B3);
\coordinate[right=\hsep of B3] (H3) ;
\node[circle,draw,fill=white](P1) at ($(B1)!0.5!(B2)$) {$P^\dagger$};
\node[circle,draw,fill=white](P2) at ($(B3)!0.5!(B2)$) {$P$};
\draw[->-] (B1) to (P1) ;
\draw[->-] (P1) to[out=45,in=-45] (P2) ;
\draw[->-] (P2) to (B3) ;
\draw (H1) to (P1) (H3)--++(0,-0.2) coordinate (m);
\draw (P1) to[out=135,in=-135] (P2) ;
\draw (P2) to[out=45,in=180] (m) to[out=0,in=90] ([xshift=\hsep/2]H1) ;
\end{sd}
\overset{\textrm{(unitary)}}{=}
\begin{sd}
\draw (0,0)arc(180:0:\hsep/2) coordinate[midway](m);
\draw (m)--++(0,\hsep/2);
\draw[->-] (-\hsep/2,0)--++(0,\hsep);
\end{sd}.
\]
By postcomposing $P^\dagger$ and taking the adjoint again, we obtain (\ref{qbij}):
\[
\begin{sd}
\coordinate (B0) ;
\coordinate[right=\hsep*1/2 of B0] (B1) ;
\coordinate[right=\hsep*3/2 of B1] (H) ;
\coordinate[above=\hsep of B0] (Ha) ;
\coordinate[right=\hsep of Ha] (B'2) ;
\coordinate[right=\hsep of B'2] (B'1) ;
\node[circle,draw,fill=white](P1) at ($(H)!1/3!(Ha)$) {$P$};
\node[circle,draw,fill=white](P2) at ($(H)!2/3!(Ha)$) {$P$};
\draw (B'2) to (P2) (B'1) to (P1);
\draw (P1) to[out=-135,in=-135]coordinate[midway](m) (P2);
\draw (m)--(B1);
\draw[->-] (H) to (P1); \draw[->-] (P1) to (P2); \draw[->-] (P2) to (Ha) ;
\end{sd}
=
\begin{sd}
\coordinate (B0) ;
\coordinate[right=\hsep*1/2 of B0] (B1) ;
\coordinate[right=\hsep*3/2 of B1] (H) ;
\coordinate[above=\hsep of B0] (Ha) ;
\coordinate[right=\hsep of Ha] (B'2) ;
\coordinate[right=\hsep of B'2] (B'1) ;
\node[circle,draw,fill=white](P1) at ($(H)!1/2!(Ha)$) {$P$};
\draw (B'2) to[out=-90,in=-135]coordinate[midway](m) (B'1);
\draw (m)--(P1) to (B1);
\draw[->-] (H) to (P1); \draw[->-] (P1) to (Ha) ;
\end{sd}.
\]
Therefore $(H,P)$ is a quantum bijection.
\end{proof}

\begin{dfn}[{Musto, Reutter, Verdon \cite[Definition 5.11]{Musto2018compositional}}]\label{dfn:Qisom}
Let $\mathcal{G}=(B,\psi,A)$ and $\mathcal{G}'=(B',\psi',A')$ be quantum graphs. 
A \emph{quantum} (resp. \emph{classical}) \emph{isomorphism} 
$(H,P): \mathcal{G}' \to \mathcal{G}$ is a quantum (resp. classical) bijection $(H,P) : (B',\psi') \to (B,\psi)$ satisfying 
\begin{align}
\begin{sd}
\coordinate (B) ;
\coordinate[right=\hsep*3/2 of B] (H) ;
\coordinate[above=\hsep*3/2 of B] (Ha) ;
\coordinate[right=\hsep*3/2 of Ha] (B') ;
\node[circle,draw,fill=white](P) at ($(B)!0.5!(B')$) {$P$};
\draw (B) to coordinate[midway] (G) (P); 
\draw (P) to (B');
\draw[->-] (H) to  (P);
\node[circle,draw,fill=white] at (G) {$A$};
\draw[->-] (P) to (Ha);
\end{sd}
=
\begin{sd}
\coordinate (B) ;
\coordinate[right=\hsep*3/2 of B] (H) ;
\coordinate[above=\hsep*3/2 of B] (Ha) ;
\coordinate[right=\hsep*3/2 of Ha] (B') ;
\node[circle,draw,fill=white](P) at ($(B)!0.5!(B')$) {$P$};
\draw (B) to (P); 
\draw (P) to coordinate[midway] (G) (B');
\draw[->-] (H) to  (P);
\node[circle,draw,fill=white] at (G) {$A'$};
\draw[->-] (P) to (Ha);
\end{sd}.					\label{qgraphiso}
\end{align}
Quantum graphs $\mathcal{G}, \mathcal{G}'$ are said to be quantum (resp. classical) isomorphic if there is a nonzero quantum (resp. classical) isomorphism $(H,P): \mathcal{G}' \to \mathcal{G}$.
\end{dfn}

\begin{rmk}
Quantum isomorphism is denoted by $\cong_{q}$.
Recall that we assume $H$ to be finite-dimensional.
If quantum graphs are quantum isomorphic via possibly infinite dimensional $H$, then they are said to be \emph{$C^*$-algebraically quantum isomorphic} ($\cong_{C^*}$) in Brannan et al. \cite[Definition 4.4]{Brannan2019bigalois}. The authors of \cite{Brannan2019bigalois} also defined \emph{quantum commuting isomorphism} ($\cong_{qc}$), and \emph{algebraic quantum isomorphism} ($\cong_{A^*}$). For quantum graphs, $\cong_{q}, \cong_{qc}\Rightarrow \cong_{C^*} \Leftrightarrow \cong_{A^*}$  (\cite[Corollary 4.8]{Brannan2019bigalois}). 

Since quantum bijections are unitary, finiteness of $\abs{H}$ implies $\abs{B}=\abs{B'}$ for $(B,\psi,A)\cong_q (B',\psi',A')$.
It is shown in \cite[Example 4.13]{Brannan2019bigalois} that there are $C^*$-quantum isomorphic quantum graphs with distinct dimensions, hence our $\cong_q$ is strictly stronger than $\cong_{C^*}$.
\end{rmk}

\begin{dfn}
Given quantum graphs $\mathcal{G},\mathcal{G}'$, the category 
$\QIso(\mathcal{G}',\mathcal{G})$ of quantum isomorphisms is a category that consists of
\begin{itemize}
\item \textbf{Objects:} quantum isomorphisms $(H,P) : \mathcal{G}' \to \mathcal{G}$;
\item \textbf{Morphisms:} intertwiners $f:(H,P)\to(H',P')$.
\end{itemize}
We denote $\QIso(\mathcal{G},\mathcal{G})$ by $\QAut(\mathcal{G})$.
\end{dfn}

\begin{rmk}
Since tensoring with zero annihilates everything, any couple of quantum graphs have a trivial quantum isomorphism $0=(P=0,H=0)$. 
\end{rmk}

\begin{rmk}\label{rmk:autconst}
Let $\mathcal{G}=(B,\psi,A), \mathcal{G}'=(B',\psi',A')$ be quantum graphs, 
and $\{e_i\}_{i=1}^m, \{e'_k\}_{k=1}^n$ be ONB's for $L^2(B,\psi),L^2(B',\psi')$ with $\abs{B}=m, \abs{B'}=n$. 
Note that a quantum isomorphism $(H,P): \mathcal{G}' \to \mathcal{G}$ can be described by operators $P_i^k\in B(H)$ as follows:
\[
P_i^k=\begin{sd}
\coordinate (B) ;
\coordinate[right=\hsep*3/2 of B] (H) ;
\coordinate[above=\hsep*3/2 of B] (Ha) ;
\coordinate[right=\hsep*3/2 of Ha] (B') ;
\node[circle,draw,fill=white](P) at ($(B)!0.5!(B')$) {$P$};
\draw (P) --++(1/2,1/2)node[circle,draw,fill=white](f){${e'_k}^\dagger$}; 
\draw (P) --++(-1/2,-1/2)node[circle,draw,fill=white](e){$e_i$};
\draw[->-] (H) to  (P);
\draw[->-] (P) to (Ha);
\end{sd};
\qquad
P=\sum_{ik}
\begin{sd}
\coordinate (B) ;
\coordinate[right=\hsep*3/2 of B] (H) ;
\coordinate[above=\hsep*3/2 of B] (Ha) ;
\coordinate[right=\hsep*3/2 of Ha] (B') ;
\node[circle,draw,fill=white](P) at ($(B)!0.5!(B')$) {$P_i^k$};
\draw (B') --++(0,-1/2)node[circle,draw,fill=white](f){$e'_k$}; 
\draw (B) --++(0,1/2)node[circle,draw,fill=white](e){$e_i^\dagger$};
\draw[->-] (H) to  (P);
\draw[->-] (P) to (Ha);
\end{sd}.
\]
Then $\tilde{P}: B \to  B' \otimes B(H)$ as in Remark \ref{rmk:qisom/*hom} is explicitly described as $\tilde{P} e_i=\sum_k e'_k \otimes P_i^k$.
In this setting $\tilde{P}$ is a unital $*$-homomorphism since $P$ is a quantum function, 
and the matrix $(P_i^k)_{k,i} \in M_{n,m}(B(H))$ is unitary since $P$ is a quantum bijection (hence unitary by Lemma \ref{lem:qbijunitary}),
and $\tilde{P}A=(A' \otimes \id_{B(H)})\tilde{P}$ since $P$ is a quantum isomorphism.
Note that $m,n$ need not be equal if we allow infinite-dimensional $H$.
By considering universal such $P_i^k$'s, we reach the notion of the quantum automorphism group of a quantum graph as below and the bigalois extention between two quantum graphs introduced in \cite[Definition 4.1]{Brannan2019bigalois}, which we later use in section 4.
\end{rmk}

\begin{dfn}[{Woronowicz \cite[Definition 1.1]{Woronowicz1987compact}, \cite[Definition 1.1]{Woronowicz1998compact}}]
A \emph{compact quantum group} (CQG) is a pair $(\mathcal{A},\Delta)$ of a separable unital $C^*$-algebra $\mathcal{A}$ and a $*$-homomorphism $\Delta:\mathcal{A}\to \mathcal{A} \otimes \mathcal{A}$, so-called comultiplication, satisfying
\begin{description}
\item[(coassociativity)]
$(\Delta \otimes \id_\mathcal{A})\Delta=(\id_\mathcal{A} \otimes \Delta)\Delta$;
\item[(cancellation property)]
$(\mathcal{A} \otimes 1)\Delta(\mathcal{A})$ \, 
 and 
$(1 \otimes \mathcal{A})\Delta(\mathcal{A})$ \, 
 are dense in $\mathcal{A} \otimes \mathcal{A}$.
\end{description}
\end{dfn}

\begin{dfn}[{Brannan et al. \cite[Definition 3.7]{Brannan2019bigalois}}]
Let $\mathcal{G}=(B,\psi,A)$ be a quantum graph and fix an ONB $\{e_i\}_i$ for $L^2(B,\psi)$. 
The quantum automorphism group of $\mathcal{G}$ is 
a CQG $\Qut(\mathcal{G})=(\A,\Delta)$ defined as follows:
\begin{itemize}
\item 
The group algebra $\A$ is the universal unital $C^*$-algebra generated by the coefficients $u_i^k$ of a unitary $u=(u_i^k)_{k,i} \in M_n(\A)$ that makes the operator
\[
\rho:B \ni e_i \mapsto \sum_k e_k \otimes u_i^k \in B \otimes \A
\]
a unital $*$-homomorphism satisfying $\rho \ A=(A \otimes \id)\rho$. This $\rho$ and $u$ are called the fundamental representation;
\item
The comultiplication $\Delta: \A \to \A \otimes \A$ is defined as a $*$-homomorphism satisfying
\[
\Delta u_i^k=\sum_j u_j^k \otimes u_i^j.
\]
\end{itemize}
We have additional operators associated to $\A$, a counit $\epsilon$ and antipode $S$ defined as a $*$-homomorphism $\epsilon:\A \to \C$ and a homomorphism $S: \A \to \A^{op}$
satisfying
\[
\epsilon u_i^k=\delta_{ik}; \qquad S u_i^k ={u_k^i}^*.
\] 
Then $(\A,\Delta,\epsilon,S)$ satisfy 
\begin{align*}
(\Delta \otimes \id)\Delta&=(\id \otimes \Delta)\Delta;
\\
(\epsilon \otimes \id)\Delta&=\id=(\id \otimes \epsilon)\Delta;
\\
m(\epsilon \otimes \id)\Delta&=\epsilon(\cdot)1=m(\id \otimes \epsilon)\Delta.
\end{align*}
Such a quadruple $(\A,\Delta,\epsilon,S)$ is called a Hopf $*$-algebra. 

\end{dfn}

The authors of \cite{Musto2018compositional} investigated the relationship between the category $\QAut(\mathcal{G})$ and the quantum automorphism group $\Qut(\mathcal{G})$ for classical graphs, but they did not introduce the quantum automorphism group of quantum graphs. Here we show the straightforward generalization of the following theorem.

\begin{thm}[{Musto, Reutter, Verdon, \cite[Proposition 5.19]{Musto2018compositional}}]
If $\mathcal{G}$ is a classical graph on $(\C^n,\tau)$, then we have an isomorphism of categories
\[
\Rep_\mathrm{fin}(\A)\cong\QAut(\mathcal{G})
\]
where $\Rep_\mathrm{fin}(\A)$ is the category of finite dimensional $*$-representations of the $C^*$-algebra $\A$, and $\QAut(\mathcal{G})$ is the category of quantum automorphisms on $\mathcal{G}$.
\end{thm}

\begin{thm}\label{thm:repcat=qaut}
If $G$ is a quantum graph on $(B,\psi)$, then we have an isomorphism of categories
\[
\Rep_\mathrm{fin}(\A)\cong\QAut(\mathcal{G})
\]
where $\Rep_\mathrm{fin}(\A)$ is the category of finite dimensional $*$-representations of the $C^*$-algebra $\A$, and $\QAut(\mathcal{G})$ is the category of quantum automorphisms on $\mathcal{G}$.
\end{thm}

\begin{proof}
As is explained in Remark \ref{rmk:autconst}, the CQG algebra $\A$ is generated by the universal coefficients of a unitary $u=(u_i^k)$ that satisfies exactly the same relation as the unitary $P=(P_i^k)$ of a quantum automorphism $(H,P)$ on $\mathcal{G}=(B,\psi,A)$. 
Therefore 
given a quantum isomorphism $(H,P)$ in $\QAut(\mathcal{G})$, the universality of $\A$ shows the existence of a $*$-representation $\pi_P:\A\ni u_i^k \mapsto P_i^k \in B(H)$. 
Conversely  a $*$-representation $\pi:\A\to B(H)$ defines operators
$P_i^k=\pi(u_i^k)$, which induces a quantum automorphism $P_\pi=\sum_j \ket{e_k} P_i^k \bra{e_i}$.
By construction, it is trivial that $P_{\pi_P}=P$ and $\pi_{P_\pi}=\pi$.
For quantum automorphisms $(H,P),(H',P')$, 
an operator $f:H \to H'$ is an intertwiner $(H,P)\to(H',P')$ in $\QAut(\mathcal{G})$
$\iff$ $(f\otimes\id_B)P=P'(\id_B \otimes f)$ 
$\iff$ $f \pi(u_i^k)=f P_i^k={P'}_i^k f=\pi'(u_i^k) f \ (\forall i,k)$
$\iff$ $f \pi(\cdot)=\pi'(\cdot) f$
$\iff$ $f$ is an intertwiner $\pi \to \pi'$ in $\Rep_\mathrm{fin}(\A)$.
Therefore the intertwiners also coincide. 
\end{proof}

Since finiteness of $\abs{H}$ is not used in the proof of Theorem \ref{thm:repcat=qaut}, if we allow `$\QAut(\mathcal{G})$' to include infinite dimensional quantum isomorphisms as in Remark \ref{rmk:autconst},
then $\Rep(\A)\cong\text{`}\QAut(\mathcal{G})\text{'}$ is obtained.

\section{Quantum graphs on $M_2$}

\subsection{Tracial quantum graphs}

Let $(B=\bigoplus_s M_{n_s}, \tau)$ be a quantum set with the unique tracial $\sqrt{\abs{B}}$-form $\tau=\frac{1}{\abs{B}}\bigoplus_s n_s \Tr_s$.
We always assume that \emph{quantum graphs} are undirected in this chapter.

Let $A=(A_{ij,s}^{kl,r})_{i,j\leq n_s,s}^{k,l\leq n_r,r}$ be a reflexive quantum graph on $(B, \tau)$ parametrized as 
\[
A_{ij,s}^{kl,r}= \braket{\tilde{e_{kl,r}} | A \tilde{e_{ij,s}}}
\textrm{ , i.e., }
A=\sum_{ijsklr} \ket{\tilde{e_{kl,r}}} A_{ij,s}^{kl,r} \bra{\tilde{e_{ij,s}}}
\]
where $\left\{ \widetilde{e_{ij,s}} = \sqrt{\frac{\abs{B}}{n_s}} e_{ij,s} \right\}$ is an ONB for $L^2(B,\tau)$.
Thus $A$ is a self-adjoint ($\ol{A_{ij,s}^{kl,r}}=A^{ij,s}_{kl,r}$) operator satisfying the following:
\begin{align}
	\text{Schur idempotent}
	&\iff \frac{1}{\sqrt{n_s n_r}}\sum_{u, v} A_{iu,s}^{kv,r} A_{uj,s}^{vl,r} = A_{ij,s}^{kl,r};
	\label{Bmat-idempotent}
	\\
	\text{reflexive}
	&\iff \frac{1}{n_s}\sum_u A_{iu,s}^{ku,s} = \delta_{ik} ;
	\label{Bmat-reflexive}
	\\
	\text{undirected (self-transpose)}
	&\iff A_{ij,s}^{kl,r} = A_{lk,r}^{ji,s} ,
	\label{Bmat-*pres}
\end{align}
where the RHS of these equivalences are quantified by $\forall i,j,s,k,l,r$. 

Note that these relations are independent for different pairs $(r,s)$ and $(r',s')$.

\subsection{Tracial quantum graphs on $M_2$}

Let $A=(A_{ij}^{kl})_{i,j=1,2}^{k,l=1,2}$ be a quantum adjacency matrix on $(M_2, \Tr/2)$ with respect to the orthonormal basis
$\left\{ \widetilde{e_{ij}} = \sqrt{2} e_{ij} \right\}$. 
Then
\begin{align}
\frac{1}{2}\left( A_{i1}^{k1} A_{1j}^{1l}+A_{i1}^{k2} A_{1j}^{2l}+A_{i2}^{k1} A_{2j}^{1l}+A_{i2}^{k2} A_{2j}^{2l} \right) 
&= A_{ij}^{kl} \quad\forall i,j,k,l=1,2;
\label{M2-relational}
\\
\frac{1}{2}\left( A_{i1}^{k1}+A_{i2}^{k2} \right) &= \delta_{ik} \quad\forall i,k=1,2;
\label{M2-reflexive}
\\
\ol{A^{ij}_{kl}}=A_{ij}^{kl} &= A_{lk}^{ji} \quad\forall i,j,k,l=1,2.  
\label{M2-*pres}
\end{align}

By the latter two conditions (\ref{M2-reflexive})(\ref{M2-*pres}), $A$ is of the following form where $x,p\in \R$ and $y,z \in\C$: 
\[
\begin{pmatrix}
A^{11}_{11} & A^{11}_{12} & A^{11}_{21} & A^{11}_{22} \\
A^{12}_{11} & A^{12}_{12} & A^{12}_{21} & A^{12}_{22} \\
A^{21}_{11} & A^{21}_{12} & A^{21}_{21} & A^{21}_{22} \\
A^{22}_{11} & A^{22}_{12} & A^{22}_{21} & A^{22}_{22} 
\end{pmatrix}
=
\begin{pmatrix}
p & \ol{y} & y & x \\
y & 2-p & z & -y \\
\ol{y} & \ol{z} & 2-p & -\ol{y} \\
x & -\ol{y} & -y & p 
\end{pmatrix}
\]

Regularity $A1=d1$ holds for some $d\in \R$ if and only if
\begin{align*}
& d1 = A(e_{11} + e_{22})
\\
&= (A_{11}^{11}+A_{22}^{11})e_{11} + (A_{11}^{12}+A_{22}^{12})e_{12} 
   + (A_{11}^{21}+A_{22}^{21})e_{21} + (A_{11}^{22}+A_{22}^{22})e_{22}
\\
&= (p+x)e_{11} + (\overline{y}-\ol{y})e_{12} + (y-y)e_{21} + (x+p)e_{22}
=(p+x)1,
\end{align*}
i.e., this is automatically $p+x=d$-regular.

If $y=0$, then we have $\Spec(A)=\{p\pm x, 2-p \pm \abs{z}\}$.

\begin{thm}\label{thm:QgraphM2}
A reflexive quantum graph $A$ on $(M_2,\tau)$ is classical (and quantum) isomorphic to exactly one of the following $d$-regular quantum graphs.
\begin{description}
\item[$d=1)$] Trivial graph $A_1=\id_B={
\begin{pmatrix}
1 &  &  &  \\
 & 1 &  &  \\
 &  & 1 &  \\
 &  &  & 1
\end{pmatrix}}$, $\Spec(A_1)=\{1,1,1,1\}$.

\item[$d=2)$] $A_2={
\begin{pmatrix}
2 &  &  &  \\
 & 0 &  &  \\
 &  & 0 &  \\
 &  &  & 2 
\end{pmatrix}}$, $\Spec(A_2)=\{2,2,0,0\}$.

\item[$d=3)$] $A_3={
\begin{pmatrix}
1 &  &  & 2 \\
 & 1 &  &  \\
 &  & 1 &  \\
2 &  &  & 1
\end{pmatrix}}$, $\Spec(A_3)=\{3,1,1,-1\}$.

\item[$d=4)$] Complete graph $A_4=4\tau(\cdot)1={
\begin{pmatrix}
2 &  &  & 2 \\
 & 0 &  &  \\
 &  & 0 &  \\
2 &  &  & 2 
\end{pmatrix}}$, $\Spec(A_4)=\{4,0,0,0\}$.

\end{description}
\end{thm}

\begin{proof}
By Schur idempotence (\ref{M2-relational}), we get the following equations:
\begin{align}
2p&=p^2+\abs{y}^2+(2-p)^2+\abs{y}^2        \label{M2-1}
\\
2(2-p)&=p(2-p)-\abs{y}^2+p(2-p)-\abs{y}^2
\iff (p-1)(2-p)=\abs{y}^2                           \label{M2-2}
\\
2x&=\abs{y}^2+x^2+\abs{y}^2+\abs{z}^2     \label{M2-3}
\\
2z&=y^2+xz+y^2+zx  \iff (1-x)z=y^2                              \label{M2-4}
\\
2y&=py+yx-y(2-p)+\ol{y}z                   \label{M2-5}
\end{align}
By (\ref{M2-2}) and (\ref{M2-4}), we get $p\in[1,2]$ and $(p-1)(2-p)=\abs{y}^2=\abs{1-x}\abs{z}$. Hence
\begin{align*}
(\ref{M2-1})
&\iff (p-1+2-p)^2=1^2=1 \quad\textrm{(automatic)} 
\\
(\ref{M2-3})
&\iff (\abs{1-x}+\abs{z})^2=1 \iff \abs{1-x}+\abs{z}=1
\end{align*}

\noindent\textbf{[0]}
 If $y=0$, (\ref{M2-5}) is automatic, $(\ref{M2-2})\iff p=1$ or $2$, 
and $(\ref{M2-4})\iff (1-x)z=0$.
\begin{itemize}
\item
If $x=1$, then $\abs{z}=1$ by (\ref{M2-3}).
\item
If $z=0$, then $\abs{1-x}=1$ by (\ref{M2-3}), hence $x=0,2$
\end{itemize}
Therefore we have the following table under $y=0$:
\[
	\begin{array}{|c|c|c|c|c|}
	p & x & z & d=p+x & \Spec(A)=\{p\pm x, 2-p \pm \abs{z}\} \\ \hline\hline
	1 & 0 & 0 & 1 & \{1,1,1,1\} \\ \hline
	1 & 1 & \T & 2 & \{2,0,2,0\} \\ \hline
	1 & 2 & 0 & 3 & \{3,-1,1,1\} \\ \hline
	2 & 0 & 0 & 2 & \{2,2,0,0\} \\ \hline
	2 & 1 & \T & 3 & \{3,1,1,-1\} \\ \hline
	2 & 2 & 0 & 4 & \{4,0,0,0\} 
	\end{array}
\]
where $\T=\{z\in\C \mid \abs{z}=1\}$.

\noindent\textbf{[1]}
If $y\neq0$,  (\ref{M2-4}) implies $z\neq0$ and hence
\begin{align}
	(\ref{M2-5})
	&\iff (x-2(2-p))y+\ol{y}z=0 \overset{y\neq0}{\iff} (x-2(2-p))y^2+\abs{y}^2 z=0
	\nonumber \\ 
	&\iff ((x-2(2-p))(1-x)+(p-1)(2-p))z=0
	\nonumber \\ 
	&\iff (x-(3-p))(x-(2-p))=0
	\iff x=3-p \textrm{ or } 2-p  \label{M2-5'}
\end{align}
By (\ref{M2-2}), we may put $y= \theta \sqrt{(p-1)(2-p)}$ 
for some $\theta\in\T$. 
Then (\ref{M2-4}) implies 
\[
	z=\frac{y^2}{1-x}= \theta^2 \frac{(p-1)(2-p)}{1-x}.
\]

If $x=2-p$ in (\ref{M2-5'}), then $d=p+x=2$ and 
\[
	z=\theta^2 \frac{(p-1)(2-p)}{p-1}= \theta^2 (2-p)= \theta^2 (4-d-p),
\]
which satisfies (\ref{M2-3}): $\abs{1-x}+\abs{z}= (p-1)+(2-p)=1$, hence all conditions are satisfied.

If $x=3-p$ in (\ref{M2-5'}), then $d=p+x=3$ and 
\[
	z=\theta^2 \frac{(p-1)(2-p)}{p-2}= \theta^2 (1-p) = \theta^2 (4-d-p),
\]
which satisfies (\ref{M2-3}): $\abs{1-x}+\abs{z}= (2-p)+(p-1)=1$, hence all conditions are satisfied.

Therefore we obtain two families of quantum graphs for each $d =2,3$ parametrized by 
 $(p,\theta) \in (1,2)\times\T$ under $\abs{y}=\sqrt{(p-1)(2-p)}\neq0$:  
\begin{align*}
	&d=2) & A^{(2)}_{p,\theta}&=
	\begin{pmatrix}
	p & \ol{\theta} \abs{y} & \theta \abs{y} & 2-p \\
	\theta \abs{y} & 2-p & \theta^2 (2-p) & -\theta \abs{y} \\
	\ol{\theta} \abs{y} & \ol{\theta}^2 (2-p) & 2-p & -\ol{\theta} \abs{y} \\
	2-p & -\ol{\theta} \abs{y} & -\theta \abs{y} & p 
	\end{pmatrix}
	\\
	&d=3) & A^{(3)}_{p,\theta}&=
	\begin{pmatrix}
	p & \ol{\theta} \abs{y} & \theta \abs{y} & 3-p \\
	\theta \abs{y} & 2-p & \theta^2 (1-p) & -\theta \abs{y} \\
	\ol{\theta} \abs{y} & \ol{\theta}^2 (1-p) & 2-p & -\ol{\theta} \abs{y} \\
	3-p & -\ol{\theta} \abs{y} & -\theta \abs{y} & p 
	\end{pmatrix}
\end{align*}
If we take the limits $p\to 1$ or $2$, these graphs converges to $y=0$ cases above.
Hence we may include them as $p \in [1,2]$.

Those graphs $\{A^{(d)}_{p,\theta} \mid \theta\in\T\}$ arising from the sign $\theta$ are mutually isomorphic via inner automorphism of $M_2$ by $u_\theta=\begin{pmatrix}1 & 0 \\ 0 & \theta \end{pmatrix}$:

\begin{lem}\label{lem:utheta}
It follows that
\begin{align}
A^{(d)}_{p,\theta}=\ad(u_\theta^*)A^{(d)}_{p,1}\ad(u_\theta).			\label{M2-utheta}
\end{align}
\end{lem}

\begin{proof}
The adjoint action 
$\ad(u_\theta) 
\begin{pmatrix}
a & b \\ c& d
\end{pmatrix}
= u_\theta \begin{pmatrix}
a & b \\ c& d
\end{pmatrix} u_\theta^*
=\begin{pmatrix}
a & \ol{\theta} b \\ \theta c& d
\end{pmatrix}$ 
has a diagonal unitary matrix expression
\[
	\ad(u_\theta)=
	\begin{pmatrix}
	1 &  &  & 0 \\
	 & \ol{\theta} &  &  \\
	 &  & \theta &  \\
	0 &  &  & 1 
	\end{pmatrix}
\]
with respect to the ONB 
$( \widetilde{e_{11}},\widetilde{e_{12}},\widetilde{e_{21}},\widetilde{e_{22}} )$.
Hence 
\[
	\ad(u_\theta^*) A \ad(u_\theta)=
	\begin{pmatrix}
	1 &  &  & 0 \\
	 & \theta &  &  \\
	 &  & \ol{\theta} &  \\
	0 &  &  & 1 
	\end{pmatrix}
	A
	\begin{pmatrix}
	1 &  &  & 0 \\
	 & \ol{\theta} &  &  \\
	 &  & \theta &  \\
	0 &  &  & 1 
	\end{pmatrix}
\]
 is the entrywise product of $A$ and 
\[
	\begin{pmatrix}
	1 & \ol{\theta} & \theta & 1 \\
	\theta & 1 & \theta^2 & \theta \\
	\ol{\theta} & \ol{\theta}^2 & 1 & \ol{\theta} \\
	1 & \ol{\theta} & \theta & 1 
	\end{pmatrix}.
\] 
Therefore we have 
$A^{(d)}_{p,\theta}=\ad(u_\theta^*)A^{(d)}_{p,1}\ad(u_\theta)$.
\end{proof}

Moreover those graphs $\{A^{(d)}_{p,1} \mid p\in[1,2]\}$ are also mutually isomorphic via inner automorphism of $M_2$ by 
$\displaystyle v_p=\frac{1}{\sqrt{2}}\begin{pmatrix}
\sqrt{1+\sqrt{2-p}}
& \sqrt{1-\sqrt{2-p}} 
\\ -\sqrt{1-\sqrt{2-p}}
& \sqrt{1+\sqrt{2-p}}
\end{pmatrix}$:

\begin{lem}
It follows that
\begin{align}
	\ad(v_p^*) A^{(3)}_{p,1} \ad(v_p) = A^{(3)}_{1,1} 
	\quad \textrm{  and  } \quad
	\ad(v_p) A^{(2)}_{3-p,1} \ad(v_p^*) = A^{(2)}_{2,1} . \label{M2_Ap}
\end{align}
\end{lem}

\begin{proof}
$\ad(v_p)$ has a unitary matrix expression
\[
	\ad(v_p)
	= \frac{1}{2}
	\begin{pmatrix}
	1+\sqrt{2-p} & \sqrt{p-1} & \sqrt{p-1} & 1-\sqrt{2-p} \\
	-\sqrt{p-1} & 1+\sqrt{2-p} & -(1-\sqrt{2-p}) & \sqrt{p-1} \\
	-\sqrt{p-1} & -(1-\sqrt{2-p}) & 1+\sqrt{2-p} & \sqrt{p-1} \\
	1-\sqrt{2-p} & -\sqrt{p-1} & -\sqrt{p-1} & 1+\sqrt{2-p} 
	\end{pmatrix}
\]
with respect to the ONB 
$( \widetilde{e_{11}},\widetilde{e_{12}},\widetilde{e_{21}},\widetilde{e_{22}} )$,
and we can directly compute (\ref{M2_Ap}). 

Abstractly $v_p$ is a unitary matrix such that $\ad(v_p)$ maps the eigenvector 
$\begin{pmatrix} -1 & 0 \\ 0 & 1 \end{pmatrix}$ 
for the eigenvalue $-1$ of $A^{(3)}_{1,1}$ to the eigenvector 
$\begin{pmatrix} -\sqrt{2-p} & \sqrt{p-1} \\ \sqrt{p-1} & \sqrt{2-p} \end{pmatrix}$ 
for the eigenvalue $-1$ of $A^{(3)}_{p,1}$. 
Since $\Spec(A^{(3)}_{p,1})=\{3,1,1,-1\}$ and $\ad(v_p)$ also preserves the eigenvector $1_{M_2}$ for the eigenvalue $3$, the orthogonality of eigenspaces implies 
\begin{align}
	\ad(v_p^*) A^{(3)}_{p,1} \ad(v_p) = A^{(3)}_{1,1}. \label{M2-vp3}
\end{align}

By the correspondence between a quantum graph $A$ and its reflexive complement 
as in Proposition \ref{prop:reflexcomple}
\[ 
	A^c \coloneqq \id_B +\abs{B}\tau(\cdot)1_B-A,
\] 
we do not need the latter equality in (\ref{M2_Ap}) for the proof of Theorem \ref{thm:QgraphM2} because the complement preserves isomorphism classes of quantum graphs, and the graphs of degree $2$ and $3$ are mutual complements. 
But here we show (\ref{M2_Ap}) explicitly. In this case
\begin{align}
	&(A^{(d)}_{p,\theta})^c
	=\id +4\tau(\cdot)1-A^{(d)}_{p,\theta}
	\nonumber \\
	&=
	\begin{pmatrix}
	1 &  &  & 0 \\
	 & 1 &  &  \\
	 &  & 1 &  \\
	0 &  &  & 1 
	\end{pmatrix}
	+
	\begin{pmatrix}
	2 &  &  & 2 \\
	 & 0 &  &  \\
	 &  & 0 &  \\
	2 &  &  & 2 
	\end{pmatrix}
	-
	\begin{pmatrix}
	p & \ol{\theta} \abs{y} & \theta \abs{y} & d-p \\
	\theta \abs{y} & 2-p & \theta^2 (4-d-p) & -\theta \abs{y} \\
	\ol{\theta} \abs{y} & \ol{\theta}^2 (4-d-p) & 2-p & -\ol{\theta} \abs{y} \\
	d-p & -\ol{\theta} \abs{y} & -\theta \abs{y} & p 
	\end{pmatrix}
	\nonumber \\
	&=
	\begin{pmatrix}
	3-p & -\ol{\theta} \abs{y} & -\theta \abs{y} & 2-d+p \\
	-\theta \abs{y} & p-1 & -\theta^2 (4-d-p) & \theta \abs{y} \\
	-\ol{\theta} \abs{y} & -\ol{\theta}^2 (4-d-p) & p-1 & \ol{\theta} \abs{y} \\
	2-d+p & \ol{\theta} \abs{y} & \theta \abs{y} & 3-p 
	\end{pmatrix}
	\nonumber \\
	&= A^{(5-d)}_{3-p,-\theta},			\label{M2-c}
\end{align}
where the last equality follows from $p-1=2-(3-p)$, $2-d+p=(5-d)-(3-p)$, and $-(4-d-p)=4-(5-d)-(3-p)$.
Note that the complement and the conjugation by $\ad(u)$ for unitary $u\in B$ commute as
\begin{align}
	\ad(u^*) A^c \ad(u) 
	&= \ad(u^*)\ad(u) +\abs{B}\tau(u\cdot u^*) u^*1 u- \ad(u^*) A \ad(u)
	\nonumber \\
	&= \id + \abs{B}\tau(\cdot)1 -\ad(u^*) A \ad(u)
	\nonumber \\
	&=\left( \ad(u^*) A \ad(u) \right)^c,		\label{M2-ad/c}
\end{align}
thereby
\begin{align*}
	 \ad(v_p) A^{(2)}_{3-p,1} \ad(v_p^*)
	&\overset{\textrm{(\ref{M2-c})}}{=} \ad(v_p) \left( A^{(3)}_{p,-1} \right)^c \ad(v_p^*) 
	\\
	&\overset{\textrm{(\ref{M2-ad/c})}}{=} \left( \ad(v_p) A^{(3)}_{p,-1} \ad(v_p^*) \right)^c
	\\&
	\overset{\textrm{(\ref{M2-utheta})}}{=} 
		\left( \ad(v_p u_{-1}^*) A^{(3)}_{p,1} \ad(u_{-1} v_p^*) \right)^c
\end{align*}
and since $u_{-1} v_p^*
=\frac{1}{\sqrt{2}}\begin{pmatrix}
\sqrt{1+\sqrt{2-p}}
& -\sqrt{1-\sqrt{2-p}} 
\\ -\sqrt{1-\sqrt{2-p}}
& -\sqrt{1+\sqrt{2-p}}
\end{pmatrix}
=v_p u_{-1}$, 
\begin{align*}
	\phantom{\ad(v_p) A^{(2)}_{3-p,1} \ad(v_p^*)}
	&\overset{\phantom{(5.15)}}{=}\left( \ad(u_{-1}^* v_p^*) A^{(3)}_{p,1} \ad(v_p u_{-1}) \right)^c
	\\&
	\overset{\textrm{(\ref{M2-vp3})}}{=} \left( \ad(u_{-1}^*) A^{(3)}_{1,1} \ad(u_{-1}) \right)^c
	\\
	&\overset{\textrm{(\ref{M2-utheta})}}{=} \left( A^{(3)}_{1,-1} \right)^c
	\overset{\textrm{(\ref{M2-c})}}{=} A^{(2)}_{2,1}.
\end{align*}
\end{proof}

Therefore up to inner automorphism, there is a unique reflexive quantum graph on $M_2$ for every degree $d\in \{1,2,3,4\}$.
Since an inner automorphism is a classical isomorphism (1-dimensional quantum isomorphism) and $\Spec(G)$ is invariant under quantum isomorphism,
the complete system of representatives for the classical and quantum isomorphism classes of quantum graphs on $M_2$ is given by the following.
\[
\begin{array}{|c|c|c|c|c|}
p & x & z=y & d=p+x & \Spec(G)=\{p\pm x, 2-p \pm \abs{z}\} \\ \hline\hline
1 & 0 & 0 & 1 & \{1,1,1,1\} \\ \hline
2 & 0 & 0 & 2 & \{2,2,0,0\} \\ \hline
1 & 2 & 0 & 3 & \{3,-1,1,1\} \\ \hline
2 & 2 & 0 & 4 & \{4,0,0,0\} 
\end{array}
\]
Recall that $G$ above are of the form 
\[
G=
\begin{pmatrix}
p & \ol{y} & y & x \\
y & 2-p & z & -y \\
\ol{y} & \ol{z} & 2-p & -\ol{y} \\
x & -\ol{y} & -y & p 
\end{pmatrix}
=
\begin{pmatrix}
p & 0 & 0 & x \\
0 & 2-p & 0 & 0 \\
0 & 0 & 2-p & 0 \\
x & 0 & 0 & p 
\end{pmatrix},
\]
therefore the table indicates the quantum graphs in the statement.
\end{proof}

Therefore reflexive quantum graph on $M_2$ can be $d$-regular for $d\in \{1,2,3,4\}$.
Hence irreflexive quantum graph on $M_2$ can be $d$-regular for $d\in \{0,1,2,3\}$.

\subsection{Nontracial quantum graphs on $M_2$}

By unitary diagonalization, a faithful state on $M_2$ is a unitary conjugate of one of the Powers states $\omega_q=\Tr(Q \, \cdot)$ 
where $q \in (0,1]$ and 
$\displaystyle 
Q=\frac{1}{1+q^2}\begin{pmatrix} 1 & 0 \\ 0 & q^2 \end{pmatrix}$.

Note that $\omega_1=\tau_{M_2}$, hence we may assume $q \in (0,1)$.

\begin{lem}
The Powers state $\omega_q$ is a $\delta=q+q^{-1}$-form on $M_2$. 
Hence $(M_2,\omega_q)$ is a quantum set.
\end{lem}

\begin{proof}
It follows from Proposition \ref{prop:deltaform} that
\[
\delta^2=\Tr(Q^{-1})=(1+q^2)(1+q^{-2})=(q+q^{-1})^2.
\]
\end{proof}

Note that we have $Q=\begin{pmatrix}
q\delta & 0 \\ 0 & q^{-1}\delta
\end{pmatrix}^{-1}$.

Let $e_{ij}$ be the $(i,j)$ matrix unit in $M_2$.
By Lemma \ref{lem:ONBinB}, 
$\{\tilde{e_{ij}} \coloneqq e_{ij}Q^{-1/2} 
\}_{ij}$ forms an ONB 
for $L^2(M_2,\omega_q)$. Explicitly these are
\begin{align*}
\tilde{e_{11}} &=\sqrt{1+q^2} e_{11} = \sqrt{q\delta} e_{11};
&
\tilde{e_{12}} &=\sqrt{1+q^{-2}} e_{12} = \sqrt{q^{-1}\delta} e_{12};
\\
\tilde{e_{21}} &=\sqrt{1+q^2} e_{21} = \sqrt{q\delta} e_{21};
&
\tilde{e_{22}} &=\sqrt{1+q^{-2}} e_{22} = \sqrt{q^{-1}\delta} e_{22}.
\end{align*}
Then we have
\begin{align*}
1 &= (q\delta)^{-1/2} \tilde{e_{11}} + (q^{-1}\delta)^{-1/2} \tilde{e_{22}};
&
m(\tilde{e_{ij}}\otimes \tilde{e_{kl}}) &= Q^{-1/2}_{jk} \tilde{e_{il}};
\\
m^\dagger \tilde{e_{ij}} 
&= (q\delta)^{1/2} \tilde{e_{i1}}\otimes \tilde{e_{1j}} 
	+ (q^{-1}\delta)^{1/2} \tilde{e_{i2}}\otimes \tilde{e_{2j}};
&
\omega_q(\tilde{e_{ij}}) &=Q^{1/2}_{ij}.
\end{align*}

For a quantum graph $(M_2,\omega_q,A)$, we put 
$A_{ij}^{kl}=\braket{\tilde{e_{kl}} | A \tilde{e_{ij}}}$ and
\[
A=\begin{pmatrix}
A^{11}_{11} & A^{11}_{12} & A^{11}_{21} & A^{11}_{22} \\
A^{12}_{11} & A^{12}_{12} & A^{12}_{21} & A^{12}_{22} \\
A^{21}_{11} & A^{21}_{12} & A^{21}_{21} & A^{21}_{22} \\
A^{22}_{11} & A^{22}_{12} & A^{22}_{21} & A^{22}_{22} 
\end{pmatrix}
.
\]
Rewriting the diagramatic definitions as equations of the coefficients,
this operator $A$ is:
\begin{description}
\item[a)]
self-adjoint if and only if $A_{ij}^{kl}=\ol{A_{kl}^{ij}}$.

\item[b)]
real if and only if $A_{ij}^{kl}
= \ol{A_{ji}^{lk}} Q_{ii}^{1/2} Q_{jj}^{-1/2} Q_{kk}^{-1/2} Q_{ll}^{1/2}$.

\item[c)]
Schur idempotent if and only if $\delta^2 A_{ij}^{kl} 
= \sum_{u,v} Q_{uu}^{-1/2}Q_{vv}^{-1/2} A_{iu}^{kv} A_{uj}^{vl}$

\phantom{Schur idempotent } 
$= q\delta A_{i1}^{k1} A_{1j}^{1l}
+ \delta A_{i2}^{k1} A_{2j}^{1l}
+ \delta A_{i1}^{k2} A_{1j}^{2l}
+ q^{-1}\delta A_{i2}^{k2} A_{2j}^{2l}$.

\item[d)]
reflexive if and only if $\delta^2 \delta_{i}^{k}
= \sum_{u} Q_{uu}^{-1} A_{iu}^{ku}= q\delta A_{i1}^{k1} + q^{-1}\delta A_{i2}^{k2}$.
\end{description}

\begin{thm}\label{thm:QgraphM2q}
An undirected reflexive quantum graph $(M_2,\omega_q,A)$ with $q\in (0,1),\delta=q+q^{-1}$ is exactly one of the following.
\begin{description}
\item[1)]
The trivial quantum graph $A_1=\id_B=\begin{pmatrix}
1 &  &  & 0 \\
 & 1 &  &  \\
 &  & 1 &  \\
0 &  &  & 1
\end{pmatrix}$, which is $1$-regular with 
$\Spec(A)=\{1,1,1,1\}$.

\item[2)]
$A_2=\begin{pmatrix}
q^{-1}\delta &  &  &  \\
 & 0 &  &  \\
 &  & 0 &  \\
 &  &  & q\delta
\end{pmatrix}$, which is irregular with 
$\Spec(A)=\{q^{-1}\delta,q\delta,0,0\}$.

\item[3)]
$A_3=\begin{pmatrix}
1 &  &  & \delta \\
 & 1 &  &  \\
 &  & 1 &  \\
\delta &  &  & 1
\end{pmatrix}$, which is irregular with
$\Spec(A)=\{1+\delta,1,1,1-\delta\}$.

\item[4)]
The complete quantum graph $A_4=\delta^2 \omega_q(\cdot)1=\begin{pmatrix}
q^{-1}\delta &  &  &  \delta \\
 & 0 &  &  \\
 &  & 0 &  \\
\delta &  &  & q\delta
\end{pmatrix}$, 
which is $\delta^2$-regular with 
$\Spec(A)=\{\delta^2,0,0,0\}$.
\end{description}
\end{thm}

\begin{proof}
By (a) and (b) we put
\begin{align*}
	p &= A_{11}^{11} \overset{(\text{a})}{=} \ol{A_{11}^{11}};
\quad	t = A_{12}^{12} \overset{(\text{a})}{=} \ol{A_{12}^{12}} \overset{(\text{b})}{=} A_{21}^{21};
\\	p' &= A_{22}^{22} \overset{(\text{a})}{=} \ol{A_{22}^{22}};
\quad	x = A_{22}^{11} \overset{(\text{b})}{=} \ol{A_{22}^{11}} \overset{(\text{a})}{=} {A_{11}^{22}};
\\	y &= A_{12}^{11} \overset{(\text{a})}{=} \ol{A_{11}^{12}} \overset{(\text{b})}{=} qA_{11}^{21}
	   \overset{(\text{a})}{=} q\ol{A_{21}^{11}} \overset{(\text{b})}{=} q^2A_{12}^{11};
\\	y' &= A_{21}^{22} \overset{(\text{a})}{=} \ol{A_{22}^{21}} \overset{(\text{b})}{=} q^{-1}A_{22}^{12}
	   \overset{(\text{a})}{=} q^{-1}\ol{A_{12}^{22}} \overset{(\text{b})}{=} q^{-2}A_{21}^{22};
\\	z &= A_{12}^{21} \overset{(\text{a})}{=} \ol{A_{21}^{12}} \overset{(\text{b})}{=} q^{2}A_{12}^{21}.
\end{align*}
Then $y=q^2 y, y'=q^{-2} y', z=q^2 z$ and $0<q<1$ imply $y=y'=z=0$.
Thus 
\[
A=\begin{pmatrix}
A^{11}_{11} & A^{11}_{12} & A^{11}_{21} & A^{11}_{22} \\
A^{12}_{11} & A^{12}_{12} & A^{12}_{21} & A^{12}_{22} \\
A^{21}_{11} & A^{21}_{12} & A^{21}_{21} & A^{21}_{22} \\
A^{22}_{11} & A^{22}_{12} & A^{22}_{21} & A^{22}_{22} 
\end{pmatrix}
=
\begin{pmatrix}
p & 0 & 0 &  x \\
0 & t & 0 & 0 \\
0 & 0 & t & 0 \\
x & 0 & 0 & p'
\end{pmatrix}
\]
where $p,t,p',x \in \R$.
By (d): $\delta \delta_{i}^{k} = q A_{i1}^{k1} + q^{-1} A_{i2}^{k2}$, 
we have
\begin{align}
\delta = q p + q^{-1} t; \quad \delta = q t + q^{-1} p'. \label{eqn:M2wreflex}
\end{align}
By (c$_{ij}^{kl}$): 
$\delta A_{ij}^{kl} 
= q A_{i1}^{k1} A_{1j}^{1l}
+  A_{i2}^{k1} A_{2j}^{1l}
+  A_{i1}^{k2} A_{1j}^{2l}
+ q^{-1} A_{i2}^{k2} A_{2j}^{2l}$,
we obtain
\begin{align*}
	(\text{c}_{11}^{11})&\quad  \delta p = q p^2 + q^{-1} t^2
&	(\text{c}_{12}^{12})&\quad  \delta t = q pt + q^{-1} tp'
\\	(\text{c}_{22}^{22})&\quad  \delta p' = q t^2 + q^{-1} {p'}^2
&	(\text{c}_{22}^{11})&\quad  \delta x = x^2
\end{align*}
Substituting (\ref{eqn:M2wreflex}) for $t$ in (c$_{11}^{11}$),
\begin{align*}
	\delta p = q p^2 + q (\delta - q p)^2
	=q(p^2+ \delta^2 - 2q\delta p +q^2 p^2)
\\	\delta^2 -(q^{-1}+ 2q)\delta p +(1+q^2) p^2=0.
\end{align*}
Since $1+q^2=q\delta$, division by $\delta$ deduces 
\[
\delta-(\delta+q)p+qp^2=(\delta-qp)(1-p)=0.
\] 
Thus (\ref{eqn:M2wreflex}) implies
\begin{align}
(p,t,p')=(1,1,1), (q^{-1}\delta, 0, q\delta). \label{eqn:M2wptp}
\end{align}
These solutions also satisfy (c$_{12}^{12}$) and (c$_{22}^{22}$).
Independently (c$_{22}^{11}$) shows $x=0,\delta$.
Therefore undirected quantum graphs are the four graphs in the statement:
\[
(p,t,p',x)=A_1(1,1,1,0), \ A_2(q^{-1}\delta, 0, q\delta,0), \
	A_3(1,1,1,\delta), \  A_4(q^{-1}\delta, 0, q\delta,\delta).
\]

Now $A$ is $d$-regular if and only if 
$(q\delta)^{1/2}1_{M_2}=q \tilde{e_{11}} + \tilde{e_{22}}$ is an eigenvector of eigenvalue $d$ for $A$:
\[
A\begin{pmatrix}q \\ 0 \\ 0 \\ 1\end{pmatrix}
=\begin{pmatrix} pq+x \\ 0 \\ 0 \\ xq+p'\end{pmatrix}
=d\begin{pmatrix}q \\ 0 \\ 0\\ 1\end{pmatrix},
\]
i.e., $d=p+q^{-1}x=xq+p'$. Thus $A_1$ is $1$-regular, $A_4$ is $\delta^2$-regular, and $A_2, A_3$ are irregular.
\end{proof}

\section{Quantum automorphism groups of quantum graphs on $M_2$}

\subsection{Quantum automorphism groups of tracial $(M_2,\tau,A)$}

For each $d=1,2,3,4$, let $\mathcal{G}_d=(M_2,\tau,A_d)$ be the $d$-regular quantum graph as in Theorem \ref{thm:QgraphM2}.

\begin{thm}\label{thm:QutM2d=14}
The quantum automorphism group of the $1$-regular (trivial) or $4$-regular (complete) quantum graphs on $(M_2,\tau)$ is the special orthogonal group $SO(3)$:
\[
\Qut(\mathcal{G}_1)=\Qut(\mathcal{G}_4) \cong SO(3).
\]
\end{thm}

\begin{proof}
Since the quantum automorphism group of trivial and complete graphs are the quantum symmetry group $\Qut(M_2,\tau)=(A_{aut}(M_2),\Delta)$ defined by Wang \cite[Proposition 4.2]{Wang1998quantum}, the statement follows directly from the computation by So{\l}tan \cite[Theorem 5.2]{Soltan2010quantum}.

Concretely \cite[Theorem 5.2]{Soltan2010quantum} shows
\begin{align*}
&C(\Qut(M_2,\tau)) \cong C(SO(3)) \\
&= C^*\left\langle S,T,R \text{ normal, commuting} \middle|
 ST=-R^2, \abs{S}+\abs{T}=1 \right\rangle
\end{align*}
with compatible comultiplication. This is the universal coefficient algebra with fundamental representation with respect to $(\tilde{e_{11}},\tilde{e_{12}},\tilde{e_{21}},\tilde{e_{22}})$:
\[
	u=(u_{ij}^{kl})=
	\begin{pmatrix}
	1-K & -R & -R^* & K \\
	C & S & T^* & -C \\
	C^* & T & S^* & -C^* \\
	K & R & R^* & 1-K
	\end{pmatrix}
\]
where $K=R^*R+T^*T, C=SR^*-RT$. The generators $S,T,R$ correspond to the coordinate functions $s,t,r$ of $SO(3)$ subgroup of $SU(3)$ as follows:
\[
	v^* SO(3) v =
	\left\{
	\begin{pmatrix}
	s & \ol{t} & \sqrt{2}(r \ol{t}-s\ol{r}) \\
	t & \ol{s} & \sqrt{2}(t \ol{r}-r\ol{s}) \\
	\sqrt{2}r & \sqrt{2}\ol{r} & \abs{s}^2-\abs{t}^2
	\end{pmatrix}
	\middle|
	\begin{matrix}st=-r^2 \\ \abs{s}+\abs{t}=1 \end{matrix}
	\right\}
	\subset SU(3)
\]
where $\displaystyle v=\frac{1}{\sqrt{2}}\begin{pmatrix}
	1 & 1 & 0 \\
	-i & i & 0 \\
	0 & 0 & \sqrt{2}
\end{pmatrix} \in U(3)$ and 
\[
	v
	\begin{pmatrix}
	s & \ol{t} & \sqrt{2}(r \ol{t}-s\ol{r}) \\
	t & \ol{s} & \sqrt{2}(t \ol{r}-r\ol{s}) \\
	\sqrt{2}r & \sqrt{2}\ol{r} & \abs{s}^2-\abs{t}^2
	\end{pmatrix}
	v^*
	=
	\begin{pmatrix}
	\Re(s+t) & -\Im(s+t) & 2\Re(r \ol{t}-s\ol{r}) \\
	\Im(s-t) & \Re(s-t) & 2\Im(r \ol{t}-s\ol{r}) \\
	2\Re r & 2\Im{r} & \abs{s}^2-\abs{t}^2
	\end{pmatrix}.
\]
Since $A_1=\id, A_4 = 4\tau(\cdot)1$ commute with $u$, we have 
$\Qut(\mathcal{G}_1)=\Qut(\mathcal{G}_4)=SO(3)$.
\end{proof}

\begin{thm}\label{thm:QutM2d=23}
The quantum automorphism group of the $2$-regular or $3$-regular quantum graphs on $(M_2,\tau)$ is a subgroup of $SO(3)$ that is isomorphic to the orthogonal group $O(2)$:
\begin{align*}
\Qut(\mathcal{G}_2) &= \Qut(\mathcal{G}_3)
\cong O(2)
\end{align*}
\end{thm}

\begin{proof}
Since $\mathcal{G}_3=\mathcal{G}_2^c$, we have $\Qut(\mathcal{G}_2)=\Qut(\mathcal{G}_3)$. 
It suffices to compute $\quotient{C(SO(3))}{\langle A_2 u=u A_2 \rangle}$.
Then $A_2 u=u A_2$ implies
\begin{align*}
	\begin{pmatrix}
	1-K & -R & -R^* & K \\
	0 & 0 & 0 & 0 \\
	0 & 0 & 0 & 0 \\
	K & R & R^* & 1-K
	\end{pmatrix}
	&=
	\begin{pmatrix}
	1-K & 0 & 0 & K \\
	C & 0 & 0 & -C \\
	C^* & 0 & 0 & -C^* \\
	K & 0 & 0 & 1-K
	\end{pmatrix},
\end{align*}
hence $C=R=0$. 
Since $C=SR^*-RT$, the additional relation is $R=0$.
Then $ST=-R^2=0$, and
 $\abs{S}+\abs{T}=1$ implies $\abs{S}=1$ or $\abs{T}=1$. 
Therefore it follows that
\begin{align*}
	\Qut(\mathcal{G}_2)
	&\cong
	\left\{
		\begin{pmatrix}
		t & 0 & 0 \\
		0 & \ol{t} & 0 \\
		0 & 0 & 1
		\end{pmatrix}, 
		\begin{pmatrix}
		0 & \ol{t} & 0 \\
		t & 0 & 0 \\
		0 & 0 & -1
		\end{pmatrix} 
	\middle| \abs{t}=1
	\right\}
\\
	&\overset{(\ad v)}{\cong} 
	\left\{
		\begin{pmatrix}
		x & 0 \\
		0 & \det x
		\end{pmatrix}
	\middle| x \in O(2)
	\right\}
	\cong O(2).
\end{align*}
where $v$ is as in the proof of Theorem \ref{thm:QutM2d=14}.
\end{proof}

\subsection{Quantum automorphism groups of nontracial $(M_2,\omega_q,A)$}

For each $d=1,2,3,4$, let $\mathcal{G}_d$ be the quantum graph on $(M_2,\omega_q)$ for $q \in (0,1)$ with adjacency operator $A_d$ as in Theorem \ref{thm:QgraphM2q}.

\begin{thm}\label{thm:QutM2q14}
The quantum automorphism groups of the trivial and complete graphs $\mathcal{G}_1, \mathcal{G}_4$ are the quantum special orthogonal group $SO_q(3)$:
\[
\Qut(\mathcal{G}_1)=\Qut(\mathcal{G}_4) \cong SO_q(3).
\]
\end{thm}

\begin{proof}
Similarly to Theorem \ref{thm:QutM2d=14}, the quantum automorphism group of trivial and complete graphs are the quantum symmetry group $\Qut(M_2,\omega_q)=(A_{aut}^Q(M_2),\Delta)$ defined by Wang \cite[Proposition 4.2]{Wang1998quantum}, and hence the statement follows from the computation by So{\l}tan \cite[Theorem 4.3]{Soltan2010quantum}.

Concretely \cite[Theorem 4.3]{Soltan2010quantum} shows that
\begin{align*}
C(\Qut(M_2,\omega_q)) \cong C(SO_q(3))
= C^*\left\langle A,G,L \right\rangle
\end{align*}
is generated by the universal coefficients of the fundamental representation with respect to $(\tilde{e_{11}},\tilde{e_{12}},\tilde{e_{21}},\tilde{e_{22}})$:
\[
	u=(u_{ij}^{kl})=
	\begin{pmatrix}
	1-q^2 K & -A & -qA^* & qK \\
	qC & L & -q^2 G^* & -C \\
	C^* & -G & L^* & -q^{-1}C^* \\
	qK & q^{-1}A & A^* & 1-K
	\end{pmatrix}
\]
where $K=A^*A+G^*G, C=q^{-1}LA^* + q^2 AG^*$. By Podle\'s \cite[Proposition 3.1]{Podles1995symmetries}, their defining relations are the following:
\begin{align*}
L^*L&=(1-K)(1-q^{-2}K) & LL^*&=(1-q^2 K)(1-q^4 K) 
& G^*G&=GG^*=K^2 \\
A^*A&=C^*C=K-K^2 & AA^*&=CC^*=q^2K-q^4 K^2 
& A^2&=q^{-1}LG \\
LG&=q^4GL & LA&=q^2AL & AG&=q^2GA \\
LG^*&=q^4 G^* L & A^*L&=q^{-1}(1-K)C & LK&=q^4 KL \\
GK&=KG & AK&=q^2 KA & CK&=q^2 KC \\
AC&=CA & & & &
\end{align*}

Since $A_1=\id, A_4 = \delta^2 \omega_q(\cdot)1$ commute with $u$, we have 
$\Qut(\mathcal{G}_1)=\Qut(\mathcal{G}_4)=SO_q(3)$.
\end{proof}

\begin{thm}\label{thm:QutM2q23}
The quantum automorphism groups of $\mathcal{G}_2,\mathcal{G}_3$ are the torus subgroup $\T$ of $SO_q(3)$:
\[
\Qut(\mathcal{G}_2)=\Qut(\mathcal{G}_3) \cong \T < SO_q(3).
\]
\end{thm}

\begin{proof}
Since $\mathcal{G}_3=\mathcal{G}_2^c$, we have $\Qut(\mathcal{G}_2)=\Qut(\mathcal{G}_3)$. 
It suffices to compute $\quotient{C(SO(3))}{\langle A_2 u=u A_2 \rangle}$.
Then $A_2 u=u A_2$ implies
{\small
\begin{align*}
	\begin{pmatrix}
	q^{-1}(1-q^2 K) & -q^{-1}A & -A^* & K \\
	0 & 0 & 0 & 0 \\
	0 & 0 & 0 & 0 \\
	q^2 K & A & qA^* & q(1-K)
	\end{pmatrix}
	&=
	\begin{pmatrix}
	q^{-1}(1-q^2 K) & 0 & 0 & q^2K \\
	C & 0 & 0 & -qC \\
	q^{-1}C^* & 0 & 0 & -C^* \\
	K & 0 & 0 & q(1-K)
	\end{pmatrix},
\end{align*}
}
hence $A=C=0, K=q^2K$. Then $K=0$ by $0<q<1$, and $G*G=K-A^*A=0$ implies $G=0$.
Therefore we have 
\[
C(\Qut(\mathcal{G}_2))=C^*\left\langle L \middle| L^*L=LL^*=1 \right\rangle \cong C(\T)
\]
where the last isomorphism is via $L \mapsto z= (\id_\T:\T \to \C)$.
Note that the comultiplication $\Delta$ of $\Qut(\mathcal{G}_2)$ is now characterized by
\[
\Delta(L)=
-qC\otimes A + L \otimes L + G \otimes q^2G^* -q^{-1}A\otimes C
=L \otimes L,
\]
which is isomorphic to the unitary torus 
$\T=(C(\T),\Delta:z \mapsto z \otimes z)$. 
Therefore $\Qut(\mathcal{G}_2)=\T$.
\end{proof}

\subsection{Quantum isomorphisms between quantum graphs on $M_2$ and $\C^4$}

Recall that a regular undirected reflexive classical graph on four vertices is isomorphic to one of the graphs $\mathcal{G}'_d=(\C^4,\tau_{\C^4},A'_d)$ of degree $d=1,2,3,4$ as in \tblref{table:graphC4}.
\begin{table}[hbtp]
	\caption{Regular reflexive graphs on four vertices up to permutation}
	\label{table:graphC4}
	\[
	\begin{array}{r||c|c|c|c}
	 & d=1&d=2&d=3&d=4
	\\
	\mathcal{G}'_d &
	\begin{sd}
	\path coordinate (A) at (0,0) coordinate(B) at (1,0) 
		coordinate(C) at (1,1)  coordinate(D) at (0,1);
	\fill (A)circle(2pt) (B)circle(2pt) (C)circle(2pt) (D)circle(2pt);
	\draw (A) arc(45:-315:0.2) (B) arc(135:-225:0.2)
		(C) arc(-135:225:0.2) (D) arc(-45:315:0.2);
	\end{sd}
	&
	\begin{sd}
	\path coordinate (A) at (0,0) coordinate(B) at (1,0) 
		coordinate(C) at (1,1)  coordinate(D) at (0,1);
	\fill (A)circle(2pt) (B)circle(2pt) (C)circle(2pt) (D)circle(2pt);
	\draw (A) arc(45:-315:0.2) (B) arc(135:-225:0.2)
		(C) arc(-135:225:0.2) (D) arc(-45:315:0.2)
		(A)--(B) (C)--(D);
	\end{sd}
	&
	\begin{sd}
	\path coordinate (A) at (0,0) coordinate(B) at (1,0) 
		coordinate(C) at (1,1)  coordinate(D) at (0,1);
	\fill (A)circle(2pt) (B)circle(2pt) (C)circle(2pt) (D)circle(2pt);
	\draw (A) arc(45:-315:0.2) (B) arc(135:-225:0.2)
		(C) arc(-135:225:0.2) (D) arc(-45:315:0.2)
		(A)--(C)--(B)--(D)--(A);
	\end{sd}
	&
	\begin{sd}
	\path coordinate (A) at (0,0) coordinate(B) at (1,0) 
		coordinate(C) at (1,1)  coordinate(D) at (0,1);
	\fill (A)circle(2pt) (B)circle(2pt) (C)circle(2pt) (D)circle(2pt);
	\draw (A) arc(45:-315:0.2) (B) arc(135:-225:0.2)
		(C) arc(-135:225:0.2) (D) arc(-45:315:0.2)
		(A)--(B)--(C)--(D)--(A)--(C) (B)--(D);
	\end{sd}
	\\
	A'_d &
	\begin{pmatrix}
	1 & 0 & 0 & 0 \\
	0 & 1 & 0 & 0 \\
	0 & 0 & 1 & 0 \\
	0 & 0 & 0 & 1
	\end{pmatrix}
	&
	\begin{pmatrix}
	1 & 1 & 0 & 0 \\
	1 & 1 & 0 & 0 \\
	0 & 0 & 1 & 1 \\
	0 & 0 & 1 & 1
	\end{pmatrix}
	&
	\begin{pmatrix}
	1 & 0 & 1 & 1 \\
	0 & 1 & 1 & 1 \\
	1 & 1 & 1 & 0 \\
	1 & 1 & 0 & 1
	\end{pmatrix}
	&
	\begin{pmatrix}
	1 & 1 & 1 & 1 \\
	1 & 1 & 1 & 1 \\
	1 & 1 & 1 & 1 \\
	1 & 1 & 1 & 1
	\end{pmatrix}
	\\
	\Spec
	& \{1,1,1,1\} & \{2,2,0,0\} & \{3,1,1,-1\} & \{4,0,0,0\}
	\end{array}
\]
\end{table}

By the identity of their spectra, we can expect the quantum isomorphism between $\mathcal{G}_d$ on $(M_2,\tau=\Tr/2)$ and 
$\mathcal{G}'_d$ on $(\C^4,\tau_{\C^4}=\Tr/4)$, and indeed this is the case.

Recall that $\{\tilde{e_{ij}}=\sqrt{2}e_{ij}\}_{i,j=1}^2$ is an ONB for $L^2(M_2,\tau)$ and 
$\{\tilde{e_{r}}=2e_{r}\}_{r=1}^4$ is an ONB for $L^2(\C^4,\tau_{\C^4})$
where $e_{ij},e_r$ are matrix units.

Before considering concrete quantum isomorphisms $(H,P):\mathcal{G}' \to \mathcal{G}$ for some $H$,
we compute the relations of the universal coefficients of quantum isomorphisms:

\begin{dfn}[{\cite[Definition 4.1]{Brannan2019bigalois}}]
Let $\mathcal{G}=(B,\psi,A), \mathcal{G}'=(B',\psi',A')$ be quantum graphs and $\{e_i\}, \{e'_k\}$ be ONB's for $L^2(B,\psi),L^2(B',\psi')$. 
The bigalois extention from $\mathcal{G}'$ to $\mathcal{G}$ is the $*$-algebra $\mathcal{O}(G^+(\mathcal{G}',\mathcal{G}))$ generated by the universal coefficients $(P_{i}^k)$
that make 
\[
P\coloneqq \sum_{ik} \ket{e'_{k}} P_{i}^k \bra{e_{i}} : B \to B' \otimes \mathcal{O}(G^+(\mathcal{G}',\mathcal{G}))
\]
a quantum isomorphism as in Remark \ref{rmk:autconst}, i.e., $P$ is a unital $*$-homomorphism, the matrix $(P_{i}^k)$ is unitary, and 
$PA=(A' \otimes \id_\mathcal{O})P$.
\end{dfn}
Recall that such $(P_{i}^k)$ is unitary if and only if $P$ satisfies the counit and comultiplication preserving conditions by Lemma \ref{lem:qbijunitary}.
If both $\mathcal{G}$ and $\mathcal{G}'$ are trivial $A^{(\prime)}=\id_{B^{(\prime)}}$, then the compatibility with adjacency operators $PA=(A' \otimes \id_\mathcal{O})P$ is trivial.
If both $\mathcal{G}$ and $\mathcal{G}'$ are complete 
$A^{(\prime)}={\delta^{(\prime)}}^2 \psi^{(\prime)}(\cdot)1_{B^{(\prime)}}$, then the compatibility with adjacency operators follows from the compatibility with unit and counit if $\delta=\delta'$.

\begin{lem}\label{lem:bigaloiscomple}
If both $\mathcal{G}$ and $\mathcal{G}'$ are real reflexive quantum graphs equipped with $\delta=\delta'$-forms $\psi,\psi'$, then 
$\mathcal{O}(G^+(\mathcal{G}',\mathcal{G}))=\mathcal{O}(G^+({\mathcal{G}'}^c,\mathcal{G}^c))$ holds for the reflexive complement $\mathcal{G}^c=(B,\psi,A^c=\id_B+{\delta}^2 \psi(\cdot)1_{B}-A)$.
\end{lem}

\begin{proof}
Since $\mathcal{G}$ and $\mathcal{G}'$ are real reflexive, $\mathcal{G}^c$ and ${\mathcal{G}'}^c$ are real reflexive quantum graphs by Proposition \ref{prop:reflexcomple}.
Since $\delta=\delta'$, we have
\[
PA^c-({A'}^c \otimes \id_\mathcal{O})P=-PA+(A' \otimes \id_\mathcal{O})P.
\]
Thus $\mathcal{O}(G^+(\mathcal{G}',\mathcal{G}))=\mathcal{O}(G^+({\mathcal{G}'}^c,\mathcal{G}^c))$.
\end{proof}

\begin{prop}
The bigalois extension $\mathcal{O}_d \coloneqq\mathcal{O}(G^+(\mathcal{G}'_d,\mathcal{G}_d))$ is given by
\begin{align*}
	\mathcal{O}_1=\mathcal{O}_4
	&=*\hyphen \left\langle 
		S_1,S_2,S_3,S_4 
	\middle| 
		\begin{array}{l}
		S_r S_r^* S_r=S_r, \	S_r S_r^* + S_r^* S_r=1, 
		\\
		\sum_{r=1}^4 S_r^* S_r=2, \ \sum_{r=1}^4 S_r=0, 
		\\
		S_s^* S_r = -S_s^* S_s S_r^* S_r \ \forall r \neq s
		\end{array}
	\right\rangle;
	\\
	\mathcal{O}_2=\mathcal{O}_3
	&=
	\quotient{\mathcal{O}_1}{\left\langle 
	S_1+S_2=S_3+S_4=0	
	\right\rangle} 
	=
	*\hyphen \left\langle
		S_1,S_3
	\middle| 
		\begin{array}{l}
		S_r S_r^* S_r=S_r, \\	
		S_r S_r^* + S_r^* S_r=1, \\
		S_1^* S_1+S_3^* S_3=1
		\end{array}
	\right\rangle,
\end{align*}
which arise as the universal coefficients of quantum isomorphism 
$P:M_2 \to \C^4 \otimes \mathcal{O}_d$ in the following way:
\begin{align}
	P\tilde{e_{11}}&=\sum_{r=1}^4 \tilde{e_r} \otimes \frac{S_r S_r^*}{\sqrt{2}} ;
	&P\tilde{e_{12}}&=\sum_{r=1}^4 \tilde{e_r} \otimes \frac{S_r}{\sqrt{2}} ;
	 \nonumber \\
	P\tilde{e_{21}}&=\sum_{r=1}^4 \tilde{e_r} \otimes \frac{S_r^*}{\sqrt{2}} ;
	&P\tilde{e_{22}}&=\sum_{r=1}^4 \tilde{e_r} \otimes \frac{S_r^* S_r}{\sqrt{2}}.
	\label{eqn:coefqisoM2C4}
\end{align}
\end{prop}
Note that the first two defining relations mean that each $S_r$ is a partial isometry where its source and range are mutual orthocomplements.

\begin{proof}
Let $(P_{ij}^r)_{i,j \leq 2}^{r\leq 4}$ be the generators of $\mathcal{O}=\mathcal{O}_d$ that make $P= \sum_{ijr} \ket{\tilde{e_{r}}} P_{ij}^r \bra{\tilde{e_{ij}}} : M_2 \to \C^4 \otimes \mathcal{O}$ a quantum isomorphism. 
The coefficients satisfy the following relations by (\ref{qfcn}), (\ref{qbij}):
\begin{description}
\item[(unit)] 
	$P1_{M_2}=1_{\C^4}\otimes 1_{\mathcal{O}}$, 
	so $1_{M_2}=\frac{1}{\sqrt{2}} (\tilde{e_{11}}+\tilde{e_{22}})$ and 
	$1_{\C^4}=\frac{1}{2}\sum_r \tilde{e_r}$ implies
	\begin{align}
	\sqrt{2}(P_{11}^r + P_{22}^r) = 1_\mathcal{O} \quad \forall r.
	\label{eqn:M2C4unit}
	\end{align}
\item[(multiplication)] 
	$P(\tilde{e_{ij}} \tilde{e_{kl}})=P(\tilde{e_{ij}})P(\tilde{e_{kl}})$, 
	so $\tilde{e_{ij}} \tilde{e_{kl}}=\sqrt{2} \delta_{jk}\tilde{e_{il}}$ 
	and $\bra{\tilde{e_r}}m_{\C^4} = 2 \bra{\tilde{e_r}} \otimes \bra{\tilde{e_r}}$
	implies
	\begin{align}
	\sqrt{2} \delta_{jk} P_{il}^r =2 P_{ij}^r P_{kl}^r \quad \forall i,j,k,l,r.
	\label{eqn:M2C4m}
	\end{align}
\item[(involution)] 
	$P(\tilde{e_{ij}}^*)^*=P(\tilde{e_{ij}})$ implies
	\begin{align}
	{P_{ji}^r}^*=P_{ij}^r \quad \forall i,j,r.
	\label{eqn:M2C4*}
	\end{align}
\item[(counit)] 
	$\tau_{\C^4} P = \tau_{M_2} \otimes 1_{\mathcal{O}}$, 
	so $\tau_{M_2}=\frac{1}{\sqrt{2}} (\bra{\tilde{e_{11}}}+\bra{\tilde{e_{22}}})$
	and $\tau_{\C^4}=\frac{1}{2}\sum_r \bra{\tilde{e_r}}$ imply
	\begin{align}
	\sum_r \sqrt{2} P_{11}^r = \sum_r \sqrt{2} P_{22}^r =2,
\quad	\sum_r \sqrt{2} P_{12}^r = \sum_r \sqrt{2} P_{21}^r =0.
	\label{eqn:M2C4tau}
	\end{align}
\item[(comultiplication)] 
	$m_{\C^4}^\dagger P=m_{\mathcal{O}}(P\otimes P)m_{M_2}^\dagger$, 
	so $m_{M_2}^\dagger \tilde{e_{ij}}
	=\sqrt{2} (\tilde{e_{i1}}\otimes\tilde{e_{1j}}+\tilde{e_{i2}}\otimes\tilde{e_{2j}})$ 
	and $(\bra{\tilde{e_s}} \otimes \bra{\tilde{e_r}})m_{\C^4}^\dagger 
	= 2 \delta_{rs} \bra{\tilde{e_r}}$
	imply
	\begin{align}	
	2 P_{ij}^r &= \sqrt{2} (P_{i1}^r P_{1j}^r + P_{i2}^r P_{2j}^r) \quad \forall i,j,r;
	\label{eqn:M2C4m`1}
\\	0 &= P_{i1}^s P_{1j}^r + P_{i2}^s P_{2j}^r \quad \forall i,j,r,s (r\neq s).
	\label{eqn:M2C4m`2}
	\end{align}
\end{description}
Put $S_r=\sqrt{2} P_{12}^r$, then (\ref{eqn:M2C4*}) and (\ref{eqn:M2C4m}) show
\[
	S_r^*=\sqrt{2} P_{21}^r, 
\quad	S_r^* S_r = \sqrt{2} P_{22}^r,
\quad	S_r S_r^* = \sqrt{2} P_{11}^r.
\]
and 
\[
	(S_r^* S_r)^2 =2 P_{22}^r P_{22}^r= S_r^* S_r,
\quad	(S_r S_r^*)^2 =2 P_{11}^r P_{11}^r= S_r S_r^*.
\]
Thus every $S_r$ is a partial isometry $S_r S_r^* S_r=S_r$ with source projection $\sqrt{2} P_{22}^r$ and range projection $\sqrt{2} P_{11}^r$.
By (\ref{eqn:M2C4unit}), these two projections are mutual orthocomplement
\[
S_r^* S_r + S_r S_r^* =1_\mathcal{O}.
\]
By (\ref{eqn:M2C4tau}), we have
\[
	\sum_r S_r^* S_r = \sum_r S_r S_r^* = 2,
\quad	\sum_r S_r =0.
\]
Since we have 
\[
\sum_r S_r S_r^* = \sum_r (1 - S_r^* S_r) = 4 - \sum_r S_r^* S_r,
\]
the equality $\sum_r S_r S_r^* = 2$ is redundant.
Now (\ref{eqn:M2C4m`1}) follows from (\ref{eqn:M2C4m}):
\[
\sqrt{2} (P_{i1}^r P_{1j}^r + P_{i2}^r P_{2j}^r) 
\overset{\textrm{(\ref{eqn:M2C4m})}}{=} P_{ij}^r + P_{ij}^r
=2 P_{ij}^r.
\]
Multiplying (\ref{eqn:M2C4m`2}) by $\sqrt{2}P_{2i}^s$ from left and by $\sqrt{2}P_{j2}^r$ from right reduces (\ref{eqn:M2C4m`2}) to
\[
0 = P_{21}^s P_{12}^r + P_{22}^s P_{22}^r
= S_s^* S_r + S_s^* S_s S_r^* S_r,
\]
and multiplying by their adjoints recovers (\ref{eqn:M2C4m`2}). 
Hence $S_s^* S_r = - S_s^* S_s S_r^* S_r$.

Since both $\mathcal{G}_1$ and $\mathcal{G}'_1$ are trivial graphs, the above are all the defining relations of $\mathcal{O}_1$.
Since $A^{(\prime)}_4={A^{(\prime)}_{1}}^c, A^{(\prime)}_3={A^{(\prime)}_{2}}^c$ and both $\tau_{M_2}$ and $\tau_{\C^4}$ are $2$-forms, 
we have $\mathcal{O}_4=\mathcal{O}_1, \mathcal{O}_3=\mathcal{O}_2$ by Lemma \ref{lem:bigaloiscomple}.

In the case of $\mathcal{O}_2$, subtracting $P$ from $PA_2=(A'_2 \otimes \id_\mathcal{O})P$ gives $P(A_2-\id_{M_2})=((A'-\id_{\C^4}) \otimes \id_\mathcal{O})P$. By matrix presentation with respect to the ONB's,
{\small
\begin{align*}
	(P_{ij}^r)(A_2-\id_{M_2})
	&=(A'_2-\id_{\C^4})(P_{ij}^r)
\\
	\begin{pmatrix}
	S_1 S_1^* & S_1 & S_1^* & S_1^* S_1 \\ 
	S_2 S_2^* & S_2 & S_2^* & S_2^* S_2 \\ 
	S_3 S_3^* & S_3 & S_3^* & S_3^* S_3 \\ 
	S_4 S_4^* & S_4 & S_4^* & S_4^* S_4
	\end{pmatrix}
	\begin{pmatrix}
	1 &  &  & 0 \\
	 & -1 &  &  \\
	 &  & -1 &  \\
	0 &  &  & 1 
	\end{pmatrix}
	&=
	\begin{pmatrix}
	0 & 1 & 0 & 0 \\
	1 & 0 & 0 & 0 \\
	0 & 0 & 0 & 1 \\
	0 & 0 & 1 & 0
	\end{pmatrix}
	\begin{pmatrix}
	S_1 S_1^* & S_1 & S_1^* & S_1^* S_1 \\ 
	S_2 S_2^* & S_2 & S_2^* & S_2^* S_2 \\ 
	S_3 S_3^* & S_3 & S_3^* & S_3^* S_3 \\ 
	S_4 S_4^* & S_4 & S_4^* & S_4^* S_4
	\end{pmatrix}
\\	
	\begin{pmatrix}
	S_1 S_1^* & -S_1 & -S_1^* & S_1^* S_1 \\ 
	S_2 S_2^* & -S_2 & -S_2^* & S_2^* S_2 \\ 
	S_3 S_3^* & -S_3 & -S_3^* & S_3^* S_3 \\ 
	S_4 S_4^* & -S_4 & -S_4^* & S_4^* S_4
	\end{pmatrix}
	&=
	\begin{pmatrix}
	S_2 S_2^* & S_2 & S_2^* & S_2^* S_2 \\ 
	S_1 S_1^* & S_1 & S_1^* & S_1^* S_1 \\ 
	S_4 S_4^* & S_4 & S_4^* & S_4^* S_4 \\ 
	S_3 S_3^* & S_3 & S_3^* & S_3^* S_3
	\end{pmatrix}.
\end{align*}
}
Hence $S_2=-S_1, S_4=-S_3$, and $\mathcal{O}_2=
\quotient{\mathcal{O}_1}{\langle S_1+S_2=S_3+S_4=0 \rangle}$.
Then $\sum_{r=1}^4 S_r^* S_r=2$ reduces to 
\[
S_1^* S_1+S_3^* S_3=1_\mathcal{O},
\] 
and $\sum_{r=1}^4 S_r=0$ follows automatically.
Finally $S_s^* S_r = - S_s^* S_s S_r^* S_r$ is automatic for $\{r,s\}=\{1,2\}, \{3,4\}$, and the rest $\{r,s\}$ follows from 
\begin{align*}
	- S_1^* S_1 S_3^* S_3 &=- S_1^* S_1 (1-S_1^* S_1)=0;
\\	S_1^* S_3 &= S_1^* S_3 S_3^* S_3 = S_1^* (1-S_1 S_1^*) S_3=0.
\end{align*}
\end{proof}

In order to show quantum isomorphism, we construct a nonzero $*$-representation of the bigalois extension on a Hilbert space.

\begin{thm}
The bigalois extension $\mathcal{O}_d \ (d=1,2,3,4)$ admits a two-dimensional $*$-representation $\pi:\mathcal{O}_d \to M_2$ defined by
\begin{align*}
\pi(S_1)&=\begin{pmatrix} 0 & 1 \\ 0 & 0 \end{pmatrix};
&\pi(S_2)&=\begin{pmatrix} 0 & -1 \\ 0 & 0 \end{pmatrix};
&\pi(S_3)&=\begin{pmatrix} 0 & 0 \\ 1 & 0 \end{pmatrix};
&\pi(S_4)&=\begin{pmatrix} 0 & 0 \\ -1 & 0 \end{pmatrix}.
\end{align*}
\end{thm}

\begin{proof}
It suffices to show that $\pi$ is a $*$-homomorphism for $d=2$ because 
$\mathcal{O}_2=\mathcal{O}_3$ is a quotient of $\mathcal{O}_1=\mathcal{O}_4$.
By definition $\pi(S_r)$ is a partial isometry with orthogonal source and range,  $\C \begin{pmatrix} 1 \\ 0 \end{pmatrix}$ and
$\C \begin{pmatrix} 0 \\ 1 \end{pmatrix}$,
which span $\C^2$. 
Trivially $\pi(S_1)+\pi(S_2)=\pi(S_3)+\pi(S_4)=0$ is satisfied, and we also have
\[
\pi(S_1)^* \pi(S_1)+\pi(S_3)^* \pi(S_3)
=\begin{pmatrix} 1 & 0 \\ 0 & 0 \end{pmatrix}
+\begin{pmatrix} 0 & 0 \\ 0 & 1 \end{pmatrix}=1_{M_2}.
\]
Thus $\pi$ defines a unital $*$-homomorphism $\mathcal{O}_2 \to M_2$.
\end{proof}

\begin{cor}\label{cor:M2C4qiso}
For every $d=1,2,3,4$, the quantum graph $\mathcal{G}_d$ on $(M_2,\tau)$ and the classical graph $\mathcal{G}'_d$ on four vertices are quantum isomorphic.
\end{cor}

\begin{proof}
By definition, a $*$-representation of the bigalois extension on a finite dimensional Hilbert space $H$ is equivalent to a quantum isomorphism between the quantum graphs via $H$. In other words (\ref{eqn:coefqisoM2C4}) with $S_r$ replaced by $\pi(S_r)$ is a quantum isomorphism $(H=\C^2,\pi(P)):\mathcal{G}'_d \to \mathcal{G}_d$.
\end{proof}

\begin{dfn}[{Brannan et al. \cite[Definition 3.11]{Brannan2019bigalois}}]
Quantum groups $G,G'$ are said to be monoidally equivalent 
if their representation categories $\Rep(G)$ and $\Rep(G')$ are unitarily monoidally equivalent as strict $C^*$-tesor categories, i.e., there is an fully faithful essentially surjective functor $\Rep(G)\to\Rep(G')$ that preserves the trivial representation, composition, involution, and tensor product of intertwiners.
\end{dfn}

Brannan et al. \cite[Theorem 4.7]{Brannan2019bigalois} proved that a quantum isomorphism between quantum graphs induces a monoidal equivalence between their quantum automorphism groups. 
Applying this to our result, we obtain the following.

\begin{cor}
\begin{description}
\item[(1)]
The special orthogonal group $SO(3)$ is monoidally equivalent to the quantum symmetric group $S_4^+$.
\item[(2)]
The orthogonal group $O(2)$
is monoidally equivalent to 
the hyperoctahedral quantum group $H_2^+<S_4^+$.
\end{description}
\end{cor}

\begin{proof}
\noindent\textbf{(1)}
Note that $\Qut(\mathcal{G}'_1)$ is the quantum symmetric group $S_4^+$, whose algebra $C(S_4^+)$ is generated by the universal coefficients of a $4$ by $4$ magic unitary $u$, i.e., a unitary matrix whose entries are projections, mutually orthogonal and summing up to $1$ on each row and column. 
It follows from $\mathcal{G}_1 \cong_q \mathcal{G}'_1$ that $\Qut(\mathcal{G}_1)=SO(3)$ is monoidally equivalent to $\Qut(\mathcal{G}'_1)=S_4^+$.

\noindent\textbf{(2)}
By Banica, Bichon, Collins \cite[Definition 2.1]{Banica2007hyperoctahedral}, $\Qut(\mathcal{G}'_2)$ is the hyperoctahedral quantum group $H_2^+<S_4^+$, whose algebra $C(H_2^+)=\quotient{C(S_4^+)}{\langle A'_2 u=u A'_2 \rangle}$ is generated by the universal coefficients of a magic unitary
\[
u=
\begin{pmatrix}
p_1 & q_1 & p_2 & q_2 \\
q_1 & p_1 & q_2 & p_2 \\
p_3 & q_3 & p_4 & q_4 \\
q_3 & p_3 & q_4 & p_4
\end{pmatrix}.
\]
It follows from $\mathcal{G}_2 \cong_q \mathcal{G}'_2$ that 
$\Qut(\mathcal{G}_2)=O(2)$ is monoidally equivalent to 
$\Qut(\mathcal{G}'_2)=H_2^+$.
\end{proof}

In the case of $d=1,4$, we can also construct a quantum isomorphism using the symmetry of the $24$-cell and four-dimensional hypercube.

\begin{thm}
The bigalois extension $\mathcal{O}_d \ (d=1,4)$ admits a four-dimensional $*$-representation $\rho:\mathcal{O}_d \to M_4$ defined by
{\small
\begin{align*}
\rho(S_1)&= \frac{\sqrt{2}}{6}
\begin{pmatrix}
0 & 0 & 0 & 0 \\ 
-3 & -1 & -1 & -1 \\ 
0 & 2 & 2 & 2 \\ 
3 & -1 & -1 & -1 
\end{pmatrix};
&\rho(S_2)&=\frac{\sqrt{2}}{6}
\begin{pmatrix}
-1 & 3 & -1 & 1 \\ 
0 & 0 & 0 & 0 \\ 
-1 & -3 & -1 & 1 \\ 
-2 & 0 & -2 & 2 
\end{pmatrix};
\\
\rho(S_3)&=\frac{\sqrt{2}}{6}
\begin{pmatrix}
2 & -2 & 0 & 2 \\
1 & -1 & 3 & 1 \\ 
0 & 0 & 0 & 0 \\  
-1 & 1 & 3 & -1 
\end{pmatrix};
&\rho(S_4)&=\frac{\sqrt{2}}{6}
\begin{pmatrix}
-1 & -1 & 1 & -3 \\
2 & 2 & -2 & 0 \\  
1 & 1 & -1 & -3 \\ 
0 & 0 & 0 & 0  
\end{pmatrix}.
\end{align*}
}
\end{thm}

We have a conceptually easier presentation of $\rho(S_r)$'s:
{\footnotesize
\begin{align*}
\rho(S_1)&= 
\frac{\sqrt{2}}{2}
\begin{pmatrix} 0 \\ -1 \\ 0 \\ 1 \end{pmatrix}
\begin{pmatrix} 1 & 0 & 0 & 0 \end{pmatrix}
+
\frac{\sqrt{2}}{6}
\begin{pmatrix} 0 \\ -1 \\ 2 \\ -1 \end{pmatrix}
\begin{pmatrix} 0 & 1 & 1 & 1 \end{pmatrix};
\\
\rho(S_2)&= 
\frac{\sqrt{2}}{2}
\begin{pmatrix} 1 \\ 0 \\ -1 \\ 0 \end{pmatrix}
\begin{pmatrix} 0 & 1 & 0 & 0 \end{pmatrix}
+
\frac{\sqrt{2}}{6}
\begin{pmatrix} -1 \\ 0 \\ -1 \\ -2 \end{pmatrix}
\begin{pmatrix} 1 & 0 & 1 & -1 \end{pmatrix};
\\
\rho(S_3)&=
\frac{\sqrt{2}}{2}
\begin{pmatrix} 0 \\ 1 \\ 0 \\ 1 \end{pmatrix}
\begin{pmatrix} 0 & 0 & 1 & 0 \end{pmatrix}
+
\frac{\sqrt{2}}{6}
\begin{pmatrix} 2 \\ 1 \\ 0 \\ -1 \end{pmatrix}
\begin{pmatrix} 1 & -1 & 0 & 1 \end{pmatrix};
\\
\rho(S_4)&=
\frac{\sqrt{2}}{2}
\begin{pmatrix} -1 \\ 0 \\ -1 \\ 0 \end{pmatrix}
\begin{pmatrix} 0 & 0 & 0 & 1 \end{pmatrix}
+
\frac{\sqrt{2}}{6}
\begin{pmatrix} -1 \\ 2 \\ 1 \\ 0 \end{pmatrix}
\begin{pmatrix} 1 & 1 & -1 & 0 \end{pmatrix}.
\end{align*}
}

The four vectors in $\rho(S_r)$ are mutually orthogonal and normalized by the coefficients. Put these orthonormal vectors $w_r,e_r,w_r^\perp,e_r^\perp$, so that we have 
\[
\rho(S_r)=w_r e_r^\dagger + w_r^\perp e_r^{\perp\dagger}.
\]

The row vectors $e_r^\perp$ in the second term correspond to a mutually orthogonal choice of the diagonal lines of the surface cubes of the hypercube. And the two row vectors $e_r,e_r^\perp$ span the plane $L_r$ containing the two parallel diagonal lines of the opposite surface cubes as in \figref{fig:posofL_r}.
The two column vectors $w_r,w_r^\perp$ span its orthocomplement $L_r^\perp$, which is the plane containing one of the four hexagons given by a partition of the $24$ vertices of the $24$-cell as in \figref{fig:posofL_r}.

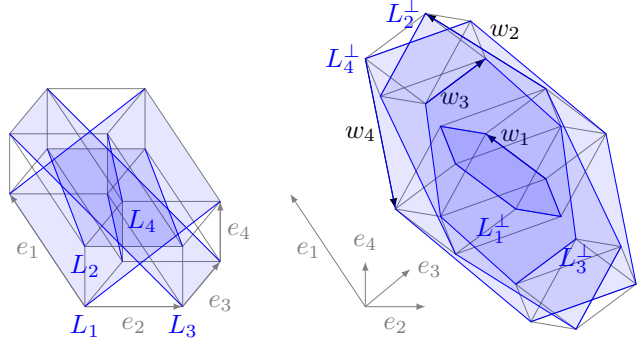
\begin{figure}[hbtp]
\centering
\caption{Positions of $L_r$ in a hypercube and $L_r^\perp$ in a 24-cell}
\label{fig:posofL_r}
\begin{tabular}{cc}
\begin{tikzpicture}
\draw (0,0)coordinate(0)  
	(-1,1.5)coordinate(1) (1.3,0)coordinate(2) (0.5,0.6)coordinate(3) (0,0.8)coordinate(4)
	($(1)+(2)$)coordinate(12) ($(1)+(3)$)coordinate(13) 
	($(1)+(4)$)coordinate(14) 	($(2)+(3)$)coordinate(23) 
	($(2)+(4)$)coordinate(24) ($(3)+(4)$)coordinate(34)
	($(12)+(3)$)coordinate(123) ($(12)+(4)$)coordinate(124) 
	($(13)+(4)$)coordinate(134) ($(23)+(4)$)coordinate(234) 
	($(123)+(4)$)coordinate(1234);
\draw[help lines] (0)--(1) (0)--(2) (0)--(3) (0)--(4)
	(1)--(12) (1)--(13) (1)--(14) (2)--(12) (2)--(23) (2)--(24)
	(3)--(13) (3)--(23) (3)--(34) (4)--(14) (4)--(24) (4)--(34)
	(12)--(123) (12)--(124) (13)--(123) (13)--(134) (14)--(124) (14)--(134)
	(23)--(123) (23)--(234) (24)--(124) (24)--(234) (34)--(134) (34)--(234)
	(123)--(1234) (124)--(1234) (134)--(1234) (234)--(1234);
\draw[help lines,->,>=latex] (0)--(1) node[midway,left]{$e_1$};
\draw[help lines,->,>=latex] (0)--(2) node[midway,below]{$e_2$};
\draw[help lines,->,>=latex] (2)--(23) node[midway,below right]{$e_3$};
\draw[help lines,->,>=latex] (23)--(234) node[midway,right]{$e_4$};
\draw[blue] (0)--(234) (1234)--(1) (0)node[below]{$L_1$};
\draw[blue] (4)--(13) (123)--(24) (4)node[below]{$L_2$};
\draw[blue] (2)--(14) (134)--(23) (2)node[below]{$L_3$};
\draw[blue] (3)--(12) (124)--(34) (34)node[below right]{$L_4$};
\fill[blue,opacity=.1] (0)--(234)--(1234)--(1)--(0);
\fill[blue,opacity=.1] (4)--(13)--(123)--(24)--(4);
\fill[blue,opacity=.1] (2)--(14)--(134)--(23)--(2);
\fill[blue,opacity=.1] (3)--(12)--(124)--(34)--(3);
\end{tikzpicture}
&
\begin{tikzpicture}
\draw (0,0)coordinate(0)  
	(-1,1.5)coordinate(1) (0.8,0)coordinate(2) (0.6,0.5)coordinate(3) (0,0.6)coordinate(4)
	($(1)+(2)$)coordinate(12) ($(1)-(2)$)coordinate(1-2) 
	($(1)+(3)$)coordinate(13) ($(1)-(3)$)coordinate(1-3) 
	($(1)+(4)$)coordinate(14) ($(1)-(4)$)coordinate(1-4) 
	($(2)+(3)$)coordinate(23) ($(2)-(3)$)coordinate(2-3) 
	($(2)+(4)$)coordinate(24) ($(2)-(4)$)coordinate(2-4) 
	($(3)+(4)$)coordinate(34) ($(3)-(4)$)coordinate(3-4) 
	($(0)-(1)+(2)$)coordinate(-12) ($(0)-(1)-(2)$)coordinate(-1-2) 
	($(0)-(1)+(3)$)coordinate(-13) ($(0)-(1)-(3)$)coordinate(-1-3) 
	($(0)-(1)+(4)$)coordinate(-14) ($(0)-(1)-(4)$)coordinate(-1-4) 
	($(0)-(2)+(3)$)coordinate(-23) ($(0)-(2)-(3)$)coordinate(-2-3) 
	($(0)-(2)+(4)$)coordinate(-24) ($(0)-(2)-(4)$)coordinate(-2-4) 
	($(0)-(3)+(4)$)coordinate(-34) ($(0)-(3)-(4)$)coordinate(-3-4);
\draw[help lines] 
	(12)--(13)--(1-2)--(1-3)--(12)--(14)--(13)--(1-4)--(1-3)--(14)--(1-2)--(1-4)--(12) 
	(23)--(24)--(2-3)--(2-4)--(23)
		(34)--(23)
		--(3-4)--(-23)--(34) 
		(34)--(24)--(-34)--(-24)--(34)
	(-12)--(-13)--(-1-2)--(-1-3)--(-12)--(-14)--(-13)--(-1-4)--(-1-3)--(-14)--(-1-2)--(-1-4)--(-12) 
	(-23)--(-24)--(-2-3)--(-2-4)--(-23)
		(-34)--(2-3)
		--(-3-4)--(-2-3)--(-34) 
		(3-4)--(2-4)--(-3-4)--(-2-4)--(3-4)
	;
\draw[help lines,->,>=latex] (-1.8,-1.8)coordinate(0')--++(1) node[midway,left]{$e_1$};
\draw[help lines,->,>=latex] (0')--++(2) node[midway,below]{$e_2$};
\draw[help lines,->,>=latex] (0')--++(3) node[right]{$e_3$};
\draw[help lines,->,>=latex] (0')--++(4) node[above]{$e_4$};
\draw[->,>=latex] (3-4)--(-23) node[midway,above]{$w_1$};
\draw[->,>=latex] (34)--(14) node[midway,above right]{$w_2$};
\draw[->,>=latex] (1-4)--(12) node[midway,below]{$w_3$};
\draw[->,>=latex] (1-2)--(-2-3) node[midway,left]{$w_4$};
\draw[blue] (2-3)--(-34)--(-24)--(-23)--(3-4)--(2-4)--cycle 
	node[below left]{$L_1^\perp$};
\draw[blue] (34)--(14)--(1-3)--(-3-4)--(-1-4)--(-13)--cycle 
	(14)node[left]{$L_2^\perp$};
\draw[blue] (-1-2)--(-14)--(24)--(12)--(1-4)--(-2-4)--cycle 
	(-14)node[below]{$L_3^\perp$};
\draw[blue] (1-2)--(-2-3)--(-1-3)--(-12)--(23)--(13)--cycle 
	node[left]{$L_4^\perp$};
\fill[blue,opacity=.1] (2-3)--(-34)--(-24)--(-23)--(3-4)--(2-4)--cycle ;
\fill[blue,opacity=.1] (34)--(14)--(1-3)--(-3-4)--(-1-4)--(-13)--cycle ;
\fill[blue,opacity=.1] (-1-2)--(-14)--(24)--(12)--(1-4)--(-2-4)--cycle ;
\fill[blue,opacity=.1] (1-2)--(-2-3)--(-1-3)--(-12)--(23)--(13)--cycle ;
\end{tikzpicture}
\end{tabular}

The hypercube is $[0,1]^4$ centered in $\R^4$. 
The 24-cell is the convex hull of $\{\pm e_i \pm e_j\}_{i \neq j}$.
For simplicity the 24-cell is drawn only on the hyperplanes of the first coordinate $x_1=\pm1,0$, which are octahedrons and a cuboctahedron.
\end{figure}

\begin{proof}
Since the four vectors in $\rho(S_r)$ are orthonormal, $\rho(S_r)$ is a partial isometry 
$\rho(S_r) \rho(S_r)^* \rho(S_r)=\rho(S_r)$ satisfying 
$\rho(S_r) \rho(S_r)^* + \rho(S_r)^* \rho(S_r)=1$. 
We have by direct computation that $\sum_{r=1}^4 \rho(S_r)=0$.
Since $\rho(S_r)^* \rho(S_r)$ is the projection onto the plane 
$L_r=\C e_r + \C e_r^\perp$, we obtain
$\sum_{r=1}^4 \rho(S_r)^*\rho(S_r)=2$
because the raw vectors $\{e_r\}_r$ and $\{e_r^\perp\}_r$ are both ONB's for $\C^4$.
Finally it suffices to show 
\[
\rho(S_s) \rho(S_s)^* \rho(S_r) = - \rho(S_s) \rho(S_r)^* \rho(S_r)
\] for all $r \neq s$, which is equivalent to \[
\rho(S_s)^* \rho(S_r) = -\rho(S_s)^* \rho(S_s) \rho(S_r)^* \rho(S_r).
\]
By direct computation, we obtain
\begin{align*}
\rho(S_s) \rho(S_s)^* \rho(S_r)
&=(w_s w_s^\dagger+w_s^\perp w_s^{\perp\dagger})
	(w_r e_r^\dagger + w_r^\perp e_r^{\perp\dagger})
\\
&=\begin{pmatrix} w_s & w_s^\perp \end{pmatrix}
\begin{pmatrix} 
\braket{w_s|w_r} & \braket{w_s|w_r^\perp} \\
\braket{w_s^\perp|w_r} & \braket{w_s^\perp|w_r^\perp} 
\end{pmatrix}
\begin{pmatrix} e_r^\dagger \\ e_r^{\perp\dagger} \end{pmatrix},
\end{align*}
and similarly
\begin{align*}
\rho(S_s) \rho(S_r)^* \rho(S_r)
&=(w_s e_s^\dagger+w_s^\perp e_s^{\perp\dagger})
	(e_r e_r^\dagger + e_r^\perp e_r^{\perp\dagger})
\\
&=\begin{pmatrix} w_s & w_s^\perp \end{pmatrix}
\begin{pmatrix} 
\braket{e_s|e_r} & \braket{e_s|e_r^\perp} \\
\braket{e_s^\perp|e_r} & \braket{e_s^\perp|e_r^\perp} 
\end{pmatrix}
\begin{pmatrix} e_r^\dagger \\ e_r^{\perp\dagger} \end{pmatrix}.
\end{align*}
Since $\{e_r\}_r, \{e_r^\perp\}_r, \{w_r\}_r, \{w_r^\perp\}_r$ are chosen to be ONB's, we have 
\[
\braket{e_s|e_r}=\braket{e_s^\perp|e_r^\perp}=\braket{w_s|w_r}=\braket{w_s^\perp|w_r^\perp} =0
\]
for all $s \neq r$.
Thus it reduces to show $\braket{w_s|w_r^\perp}=-\braket{e_s|e_r^\perp}$
for all $s \neq r$.
It indeed holds that
\begin{align*} 
\braket{w_1|w_r^\perp}&=-\frac{1}{\sqrt{3}}=-\braket{e_1|e_r^\perp} \quad (r=2,3,4); \\
\braket{w_2|w_4^\perp}&=-\frac{1}{\sqrt{3}}=-\braket{e_2|e_4^\perp}; \quad
\braket{w_2|w_r^\perp}=\frac{1}{\sqrt{3}}=-\braket{e_2|e_r^\perp}  \quad (r=1,3); \\
\braket{w_3|w_2^\perp}&=-\frac{1}{\sqrt{3}}=-\braket{e_3|e_2^\perp}; \quad
\braket{w_3|w_r^\perp}=\frac{1}{\sqrt{3}}=-\braket{e_3|e_r^\perp} \quad (r=1,4); \\
\braket{w_4|w_3^\perp}&=\frac{1}{\sqrt{3}}=-\braket{e_4|e_3^\perp}; \quad
\braket{w_4|w_r^\perp}=-\frac{1}{\sqrt{3}}=-\braket{e_4|e_r^\perp} \quad (r=1,2).
\end{align*}
Therefore $\rho$ defines a $*$-homomorphism $\mathcal{O}_1 \to M_4$.
\end{proof}

\section*{Concluding Remarks}

For future perspective, it is natural to consider the classification of general directed quantum graphs on $M_2$ and to ask which quantum subgroup of $SO_q(3)=\Qut(M_2,\omega_q)$ is obtained as a quantum automorphism group of them.
Such a classification will help us to approach a quantum graph version of the Frucht property: whether a quantum group acting on a quantum graph is isomorphic to the quantum automorphism group of some quantum graph. Its classical graph version is discussed by Banica, McCarthy \cite{Banica2021frucht} with several counterexamples.

Since we introduced the regularity of quantum graphs, it is natural to ask whether the spectrum of a regular quantum graph can characterize its properties (connected, bipartite, expander, etc.) similarly to classical cases.
It is the next step to investigate the connectedness of quantum graphs on $M_n$ introduced by Ch{\'a}vez-Dom{\'\i}nguez, Swift \cite{Chavez2021connectivity}.

\bibliographystyle{plain}
\bibliography{bunken.bib}

\end{document}